\documentclass[11pt, oneside]{amsart}

\usepackage{amsmath,amssymb,verbatim,amsthm,bbm, mathrsfs,amscd,pb-diagram}
\usepackage{aliascnt} %for environment names
\usepackage{appendix}
\DeclareMathAlphabet{\mathpzc}{OT1}{pzc}{m}{it}
\usepackage[all]{xy}
\usepackage{multirow,longtable}
\usepackage{float} %used to force figures to stay in place (use \begin{figure}[H])
\usepackage{hyperref}
\usepackage{cleveref}
\usepackage{marginnote}
\usepackage{framed}
\usepackage{mathtools}
\usepackage{csquotes}
\usepackage{tikz}

\usepackage{multicol}

\usepackage{lscape} % inorder to have landscape pages

\usepackage[disable]{todonotes}

\newcounter{todocounter}
\newcommand{\todonum}[1]{\stepcounter{todocounter}\todo{\thetodocounter: #1}}
\makeatletter
\providecommand\@dotsep{5}
\renewcommand{\listoftodos}[1][\@todonotes@todolistname]{%
  \@starttoc{tdo}{#1}}
\makeatother

\usepackage{float}
\restylefloat{table}

\crefname{table}{table}{tables}
\crefname{listing}{Program-code}{Program-codes}  
\Crefname{listing}{Program-code}{Program-codes}
\crefname{subsection}{subsection}{subsections}

\theoremstyle{plain}

\newtheorem{Thm}{Theorem}[section]
\newtheorem{Cor}[Thm]{Corollary}
\newtheorem{Prop}[Thm]{Proposition}
\newtheorem{Lem}[Thm]{Lemma}

\theoremstyle{definition}
\newtheorem{Remark}[Thm]{Remark}

\numberwithin{equation}{section}
\newcommand{\Card}[1]{\left\vert #1\right\vert} %cardinality
\newcommand{\Places}{\mathcal{P}} %Set of places of number field
\DeclareMathOperator{\zfun}{\xi} %Zeta-function
\DeclareMathOperator{\Lfun}{\mathcal{L}} %$\Lfun$-functions
\newcommand{\st}{\operatorname{\mathfrak{st}}} %standard representation
\newcommand{\coset}[1]{\left[ #1 \right]}  %square brackets
\newcommand{\gen}[1]{\left\langle #1 \right\rangle}  %square brackets
\newcommand{\FNorm}[1]{\left\vert #1 \right\vert} %Norm in a local field/absolut value
\newcommand{\Nm}{\operatorname{Nm}}

\newcommand{\Image}{\operatorname{Im}}
\newcommand{\Ind}{\operatorname{Ind}}

\newcommand{\Gal}{\operatorname{Gal}}
\newcommand{\Aut}{\operatorname{Aut}}

\newcommand{\sgn}{\operatorname{sgn}}

\newcommand{\Res}{\operatorname{Res}}

\newcommand{\C}{\mathbb C}
\newcommand{\A}{\mathbb{A}}

\newcommand{\Z}{\mathbb{Z}}
\newcommand{\R}{\mathbb{R}}
\newcommand{\N}{\mathbb{N}}

\newcommand{\mO}{\mathcal{O}}

\newcommand{\bk}[1]{\left(#1\right)} %brackets
\newcommand{\bm}{\begin{multline*}}
\newcommand{\tu}{\end  {multline*}}

\DeclareMathOperator{\Id}{\mathbf{1}} %identity element 1
\newcommand{\Ga}{\mathbb{G}_a} %aditive group
\newcommand{\Gm}{\mathbb{G}_m} %multiplicative group
\newcommand{\Eisen}{\mathcal{E}}
\DeclareMathOperator{\unif}{\varpi} %uniformizer
\newcommand{\modf}[1]{\mathcal{\delta}_{#1}} %modular function of a group
\renewcommand{\check}[1]{#1 ^{\vee}} %for coroots
\DeclareMathOperator{\Real}{\mathfrak{Re}} %Real part
\newcommand{\piece}[1]{\left\{\begin{matrix} #1 \end{matrix}\right.} %piecewise functions
\newcommand{\set}[1]{\left\{ #1 \right\}} %sets
\newcommand{\mvert}{\mathrel{}\middle\vert\mathrel{}} %midle vert line inside a set
\newcommand{\res}[1]{\Bigg\vert_{#1}}
\newcommand{\suml}{\sum\limits}
\newcommand{\prodl}{\, \prod\limits}
\newcommand{\intl}{\int\limits}

\newcommand{\rmod}{/}
\newcommand{\lmod}{\backslash}
\newcommand{\Stab}{\operatorname{Stab}}

\newcommand{\bfG}{\mathbf{G}}

\newcommand{\bfP}{\mathbf{P}}
\newcommand{\bfM}{\mathbf{M}}
\newcommand{\bfU}{\mathbf{U}}
\newcommand{\bfS}{\mathbf{S}}
\newcommand{\bfB}{\mathbf{B}}
\newcommand{\bfT}{\mathbf{T}}
\newcommand{\bfN}{\mathbf{N}}
\newcommand{\bfCT}{\mathbf{CT}}
\newcommand{\bfX}{\mathbf{X}}

\newcommand{\ds}{\displaystyle}

\newcommand{\placestimes}{\displaystyle\operatorname*{\otimes}_{\nu\in\Places}}

\makeatletter
\def\imod#1{\allowbreak\mkern10mu({\operator@font mod}\,\,#1)}
\makeatother

\makeatletter
\renewcommand\section{\@startsection{section}{1}{\z@}%
                                  {-3.5ex \@plus -1ex \@minus-.2ex}%
                                  {2.3ex \@plus.2ex}%
                                  {\center\normalfont\large\bfseries}}
\makeatother

\makeatletter
\renewcommand\subsection{\@startsection{subsection}{2}{\z@}%
	{-3.5ex \@plus -1ex \@minus-.2ex}%
	{2.3ex \@plus.2ex}%
	{\normalfont\large\bfseries}}
\makeatother

\makeatletter
\renewcommand\subsubsection{\@startsection{subsubsection}{3}{\z@}%
	{-3.5ex \@plus -1ex \@minus-.2ex}%
	{2.3ex \@plus.2ex}%
	{\normalfont\large\bfseries}}
\makeatother

\makeatletter
\newtheorem*{rep@theorem}{\rep@title} \newcommand{\newreptheorem}[2]{%
\newenvironment{rep#1}[1]{%
\def\rep@title{\bf #2 \ref{##1} }%
\begin{rep@theorem} }%
{\end{rep@theorem} } }
\makeatother
\newreptheorem{theorem}{Theorem}
\newreptheorem{lemma}{Lemma}
\protected\def\ignorethis#1\endignorethis{}
\let\endignorethis\relax

\usepackage{colortbl}
\usepackage{zref-savepos}
\usepackage{pict2e}

\newcounter{NoTableEntry}
\renewcommand*{\theNoTableEntry}{NTE-\the\value{NoTableEntry}}

\makeatletter
\newcommand*{\notableentry}{%
  \kern-\tabcolsep
  \stepcounter{NoTableEntry}%
  \vadjust pre{\zsavepos{\theNoTableEntry t}}% top
  \vadjust{\zsavepos{\theNoTableEntry b}}% bottom
  \zsavepos{\theNoTableEntry l}% left
  \raisebox{%
    \dimexpr\zposy{\theNoTableEntry b}sp
    -\zposy{\theNoTableEntry l}sp\relax
  }[0pt][0pt]{%
    \setlength{\unitlength}{1pt}%
    \edef\w{%
      \strip@pt\dimexpr\zposx{\theNoTableEntry r}sp%
      -\zposx{\theNoTableEntry l}sp\relax
    }% 
    \edef\h{%
      \strip@pt\dimexpr\zposy{\theNoTableEntry t}sp%
      -\zposy{\theNoTableEntry b}sp\relax
    }%
    \ifdim\w pt=0pt % prevent error in first run for \line(0,0)
    \else
      \begin{picture}(0,0)%
        \edef\x{%
          \noexpand\put(0,0){\noexpand\line(\w,\h){\w}}%  
          \noexpand\put(0,\h){\noexpand\line(\w,-\h){\w}}%
        }\x
      \end{picture}%
    \fi
  }%
  \hspace{0pt plus 1filll}%
  \zsavepos{\theNoTableEntry r}% right
  \kern-\tabcolsep
}

\makeatletter
\providecommand*{\cupdot}{%
	\mathbin{%
		\mathpalette\@cupdot{}%
	}%
}
\newcommand*{\@cupdot}[2]{%
	\ooalign{%
		$\m@th#1\cup$\cr
		\sbox0{$#1\cup$}%
		\dimen@=\ht0 %
		\sbox0{$\m@th#1\cdot$}%
		\advance\dimen@ by -\ht0 %
		\dimen@=.5\dimen@
		\hidewidth\raise\dimen@\box0\hidewidth
	}%
}

\providecommand*{\bigcupdot}{%
	\mathop{%
		\vphantom{\bigcup}%
		\mathpalette\@bigcupdot{}%
	}%
}
\newcommand*{\@bigcupdot}[2]{%
	\ooalign{%
		$\m@th#1\bigcup$\cr
		\sbox0{$#1\bigcup$}%
		\dimen@=\ht0 %
		\advance\dimen@ by -\dp0 %
		\sbox0{\scalebox{2}{$\m@th#1\cdot$}}%
		\advance\dimen@ by -\ht0 %
		\dimen@=.5\dimen@
		\hidewidth\raise\dimen@\box0\hidewidth
	}%
}
\makeatother
\allowdisplaybreaks

\title[The Degenerate Residual Spectrum of $Spin_8^E$ Along the Heisenberg Parabolic]{The Degenerate Residual Spectrum of Quasi-Split Forms of $Spin_8$ Associated to the Heisenberg Parabolic Subgroup}
\author{Avner Segal}
\address{Department of Mathematics, University of British Columbia, Vancouver, BC, V6T 1Z2, Canada}
\email{avners@math.ubc.ca}

\begin{document}

\begin{abstract}
In \cite{MR3284482} and \cite{MR3658191}, the twisted standard $\mathcal{L}$-function $\Lfun\bk{s,\pi,\chi,\st}$ of a cuspidal representation $ \pi$ of the exceptional group of type $G_2$ was shown to be represented by a family of new-way Rankin-Selberg integrals.
These integrals connect the analytic behaviour of $\Lfun\bk{s,\pi,\chi,\st}$ with that of a family of degenerate Eisenstein series $\Eisen_E\bk{\chi, f_s, s, g}$ on quasi-split forms $H_E$ of $Spin_8$, induced from Heisenberg parabolic subgroups.
The analytic behaviour of the series $\Eisen_E\bk{\chi, f_s, s, g}$ in the right half-plane $\Real\bk{s}>0$ was studied in \cite{SegalEisen}.
In this paper we study the residual representations associated with $\Eisen_E\bk{\chi, f_s, s, g}$.
\end{abstract}

\maketitle

\begin{center}
\end{center}

\tableofcontents

\section{Introduction}

In \cite{MR3284482} and \cite{MR3658191}, the twisted standard $\mathcal{L}$-function $\Lfun\bk{s,\pi,\chi,\st}$ of a cuspidal representation $ \pi$ of the exceptional group of type $G_2$ was shown to be represented by a family of new-way Rankin-Selberg integrals.
These integrals links the analytic behaviour of $\Lfun\bk{s,\pi,\chi,\st}$ with that of a family of degenerate Eisenstein series $\Eisen_E\bk{\chi, f, s-\frac{1}{2}, g}$ on quasi-split forms $H_E$ of $Spin_8$ induced from Heisenberg parabolic subgroups.

The analytic behaviour of $\Eisen_E\bk{\chi, f, s, g}$ in the right half-plane $\Real\bk{s}>0$ was studied in \cite{SegalEisen}.
%It was shown there that $\Eisen_E\bk{\chi, f, s, g}$ is holomorphic on $\Real\bk{s}>0$ except for, possibly, poles at $\frac{1}{2}$, $\frac{3}{2}$ or $\frac{5}{2}$.
%At $\frac{1}{2}$ and $\frac{5}{2}$, if there is a pole, it is simple.
%At $\frac{3}{2}$, if there is a pole, it is simple, unless $\chi=\Id$, in which case the pole may be double.
As a consequence, it was shown that $\Lfun\bk{s,\pi,\chi,\st}$ is holomorphic at any $s\neq 1,2$ such that $\Real\bk{s}>0$.
It was further shown that the orders of the poles of $\Lfun\bk{s,\pi,\chi,\st}$ at $s=1$ and $s=2$ are bounded as follows:
\begin{itemize}
	\item At $s=1$, $\Lfun\bk{s,\pi,\chi,\st}$ may admit at most a simple pole when $\chi$ is a quadratic character.
	For any other $\chi$, $\Lfun\bk{s,\pi,\chi,\st}$ is holomorphic there.
	
	\item At $s=2$, $\Lfun\bk{s,\pi,\chi,\st}$ may admit at most a double pole when $\chi$ is trivial and at most a simple pole when $\chi$ is of order $2$ or $3$. For any other $\chi$, $\Lfun\bk{s,\pi,\chi,\st}$ is holomorphic there.
\end{itemize}

This information was applied, in \cite[Part 2]{SegalEisen} and \cite{RallisSchiffmannPaper}, to classify all cuspidal representations $\pi$ of $G_2\bk{\A}$ such that $\Lfun\bk{s,\pi,\chi,\st}$ admits a pole at $s=2$ in terms of functorial lifts.
More precisely, it is shown that
\begin{itemize}
	\item (\cite{SegalEisen}) If $\Lfun\bk{s,\pi,\Id,\st}$ admits a pole of order two at $s=2$ or $\Lfun\bk{s,\pi,\chi,\st}$ admits a simple pole at $s=2$ for $\chi\neq\Id$, then $\pi$ is a lift from a group of finite type.

\item (\cite{RallisSchiffmannPaper}) If $\Lfun\bk{s,\pi,\Id,\st}$ admits a simple at $s=2$ then $\pi$ is a Rallis-Schiffman lift from $\widetilde{SL_2}$.
\end{itemize}
These calculations make use of the residual representation of $\Eisen_E\bk{\chi, f_s, s, g}$ at $s=\frac{3}{2}$.

The classification of cuspidal representations $\pi$ of $G_2$ such that $\Lfun\bk{s,\pi,\chi,\st}$ admits a pole at $s=1$ in terms of functorial lifts is an open problem.
This paper is a study of the residual representations of $\Eisen_E\bk{\chi, f_s, s, g}$ for $\Real\bk{s}>0$, and in particular at $s=\frac{1}{2}$.
Applying the results of this paper to the classification of cuspidal representations of $G_2$ in terms of the analytic behaviour of $\Lfun\bk{s,\pi,\chi,\st}$ at $s=1$ is a work in progress.

The square-integrable residual representations associated with $\Eisen_E\bk{\chi, f_s, s, g}$ are listed in \Cref{Thm:Main_Sq_Int}.
The statement of this theorem will not be quoted here since it requires a detailed description of the irreducible quotients of the local degenerate principal series $I_\nu\bk{\chi_\nu,s}$ associated to $\Eisen_E\bk{\chi, f_s, s, g}$.
However, we note here the main interesting feature of this theorem.
The square-integrable residues of $\Eisen_E\bk{\chi, f_s, s, g}$ at $s=\frac{1}{2}$, when $\chi$ is of order $2$, are given by parity conditions on the cardinality of certain subsets of the set  of places $S$ where the local representation is ramified.
Most of these parity conditions are similar to those found in similar computations (for example, see\cite{MR3742780,MR3437492,MR1847140,MR1300285}).
\todonum{Should maybe add more references?}
However, to the best of the author's knowledge, there is no previous example of a parity condition similar to the one in \Cref{Eq:Sq_int_thm_split_case_nontrivial_char}.

It is worth noting that the residual spectrum of $H_E$, when $E$ is a field, was computed in \cite{LaoResidualSpectrum}.
In particular, decomposing the square-integrable spectrum of $H_E$ with respect to cuspidal components, we get
\begin{equation}
L^2\bk{H_E\bk{F}\lmod H_E\bk{\A}} = \bigoplus_{\coset{M,\sigma}} L^2_{\coset{M,\sigma}},
\end{equation}
where the sum goes over all pairs of a standard Levi subgroup $M$ and a cuspidal representation $\sigma$ of $M$, up to $W$-conjugation.
The space $L^2_{\coset{M,\sigma}}$ is the subspace of $L^2\bk{H_E\bk{F}\lmod H_E\bk{\A}}$ spanned by automorphic forms with cuspidal data $\coset{M,\sigma}$.

We also write
\begin{equation}
L^2\bk{H_E\bk{F}\lmod H_E\bk{\A}} = L^2_{cusp.} \oplus L^2_{res.} \oplus L^2_{cont.},
\end{equation}
where $L^2_{cusp.}$ denotes the cuspidal spectrum of $H_E\bk{\A}$, $L^2_{res.}$ denotes its residual spectrum and $L^2_{cont.}$ denotes the continuous spectrum.
Let $L^2_{\coset{M,\sigma},res.}$ denote $L^2_{\coset{M,\sigma}} \cap L^2_{res.}$.

By the general theory of Eisenstein series, the degenerate residual spectrum, computed in \Cref{Thm:Main_Sq_Int}, is contained in
\begin{equation}
\label{eq:Part_of_Spectrum_associated_to_deg_res_spectrum}
\bigoplus_{\chi:F^\times\lmod\A^\times\to \mathcal{S}^1} L^2_{\coset{T_E,\mu_\chi},res.},
\end{equation}
%\todonum{Should be $L^2_{\coset{T_E,\mu_\chi},res}$}
where $T_E$ is a maximal torus in $H_E$ and $\mu_\chi$ is a character of $T_E$ obtained by restriction of a character of the Levi subgroup of the Heisenberg parabolic subgroup of $H_E$; see \Cref{Sec:Deg_Eis_Series_on_H_E} for more details.
Note that $L^2_{\coset{T_E,\mu_\chi},res.}$ might be $\bk{0}$.
A comparison of \Cref{Thm:Main_Sq_Int}, when $E$ is a field, and the results of \cite[Theorem 5.15]{LaoResidualSpectrum}, shows that the square-integrable degenerate residual spectrum spans all of the space \ref{eq:Part_of_Spectrum_associated_to_deg_res_spectrum}.
%Theorem 5.15 in [Lao] is in page 66

The non-square-integrable residual representations associated with $\Eisen_E\bk{\chi, f_s, s, g}$ are listed in \Cref{Thm:Main_non_Sq_Int}.
The non-square-integrable residual representation at $s=\frac{3}{2}$ was essentially computed in \cite{RallisSchiffmannPaper}.
When $s=\frac{1}{2}$, a pole occurs when $E=F\times K$ and $\chi\circ\Nm_{K/F}\equiv \Id$.
For any such $\chi$, the residual representation is shown to be irreducible, while the technique used for the proof varies for different $\chi$.

For $\chi=\Id$, we prove a Siegel-Weyl-type identity between the residue of $\Eisen_E\bk{\Id, f_s, s, g}$ at $s=\frac{1}{2}$ and the special value of another Eisenstein series.
This identity is especially interesting as the other Eisenstein series is evaluated on the unitary axis, which makes the associated degenerate principal series semi-simple.

When $\chi\neq\Id$, the computation involves a detailed study of the images of certain intertwining operators, for which the results of the previous case are surprisingly useful.

%\begin{Conj}
%	Let $\pi$ be a cuspidal representation of $G_2\bk{\A}$ such that $\Lfun\bk{s,\pi,\chi,\st}$ is holomorphic at $s=2$ for all $\chi$ and and has a pole at $s=1$ for some $\chi$.
%	Then $\pi$ is a weak lift from a representation of $PGL_3$, or one of its inner forms $PGU_3$, under a map 
%	\[
%	SL_3\bk{\C} \rightarrow G_2\bk{\C},\quad
%	SL_3\rtimes\mu_2\bk{\C} \rightarrow G_2\bk{\C} .
%	\]
%\end{Conj}
%\todonum{Do wee need holomorphicity at $s=2$? Do we need any $\chi$ or only $\chi=\Id$?}

This paper is structured as follows:
\begin{itemize}
	\item \Cref{Sec:Background} describes general notation and results used throughout this paper.
	
	\item In \Cref{Sec:Deg_Eis_Series_on_H_E}, the groups $H_E$ and the Eisenstein series $\Eisen_E\bk{\chi, f, s, g}$ are introduced.
	Also, the main results of \cite{SegalEisen} are summarized.
	
	\item \Cref{Sec_Local_Representations} studies the irreducible quotients of the local degenerate principal series $I_\nu\bk{\chi_\nu,s}$, associated with $\Eisen_E\bk{\chi, f, s, g}$.
	Some parts of the computation, for Archimedean places, are performed in \Cref{Appendix:Archimedean}.
	
	These irreducible quotients are then identified as eigenspaces of various intertwining operators acting on the maximal semi-simple quotient of $I_\nu\bk{\chi_\nu,s}$.

	\item In \Cref{Sec:SquareIntegrableResidues}, the square-integrable degenerate residual spectrum is computed. 
	The results of this section are summarized, at its end, in \Cref{Thm:Main_Sq_Int}.
	
	\item In \Cref{Sec:NonSquareIntegrableResidues}, the non-square-integrable degenerate residual spectrum is computed. 
	The results of this section are summarized, at its end, in \Cref{Thm:Main_non_Sq_Int}.
	
	\item In \Cref{Appendix:Archimedean} we demonstrate the application of the software "atlas of lie groups" (ATLAS) for certain calculations in $I_\nu\bk{\chi_\nu,s}$ at Archimedean places.
	
	\item In \Cref{App:Evaluation_of_Normalized_Eisenstein_Series}, a few complementary calculations for \Cref{Sec:NonSquareIntegrableResidues} are carried out.
\end{itemize}

{\bf Acknowledgments.} 
It is a great pleasure to thank Nadya Gurevich, Wee Teck Gan and Gordan Savin for insightful discussions during my work on \cite{SegalEisen} and this paper.

During the preparation of \Cref{Appendix:Archimedean}, I was helped by Jeffrey Adams, Siddhartha Sahi, Zhuohui Zhang and Lior Silberman and I wish to thank them all for their assistance.

%Jing Feng Lao
\todonum{Do I need to quote a grant?}

%\todonum{Finish abstract and introduction, acknowledgements}

%\todonum{Check that notations along the paper are consistent. Check spelling and grammar.}

\section{Background Theory}
\label{Sec:Background}
In this section we consider notations and preliminary results we use in this paper.
For a comprehensive account on the theory of Eisenstein series consider \cite{MR1361168}.
Some of the facts discussed in this section are described in more detail in the survey in \cite[Section 2]{SegalEisen}.

\subsection{Notation}
%We note that there exists $n_P\in\N$ and $\alpha_P\in\Delta$ such that $\modf{P} = \FNorm{n_P\cdot\alpha_P}$ is the modular quasi-character of $P$.
%
%
%Given a quasi-character $\chi:M\bk{F}\lmod M\bk{\A}\to\C^\times$ we may form the (normalized) parabolic induction
%\[
%\Ind_{P\bk{\A}}^{G\bk{\A}}\bk{\chi} = \set{f:G\bk{\A}\to\C \mvert 
%\begin{matrix}
%f\bk{mug} = \chi\bk{m} f\bk{g} ,\quad \forall m\in M\bk{\A},\ u\in U\bk{\A},\ g\in G\bk{\A}
%\end{matrix} }
%\]
%\todonum{Finsih this definition}
%
%
%There exists a quasi-character, also denoted $\chi$, of $T\bk{F}\lmod T\bk{\A}$ such that
%\begin{align}
%\begin{split}
%& \Ind_{P\bk{\A}}^{G\bk{\A}}\bk{\chi} \hookrightarrow \Ind_{B\bk{\A}}^{G\bk{\A}} \bk{\chi \modf{P}^{\frac{1}{2}} \modf{B}^{-\frac{1}{2}}} \\
%& \Ind_{P\bk{\A}}^{G\bk{\A}}\bk{\chi} \twoheadrightarrow \Ind_{B\bk{\A}}^{G\bk{\A}} \bk{\chi \modf{P}^{-\frac{1}{2}} \modf{B}^{\frac{1}{2}}} ,
%\end{split}
%\end{align}
%where all inductions are normalized.
%
%For $K$-finite section $f_s\in \Ind_{P\bk{\A}}^{G\bk{\A}}\bk{\chi\alpha_P^s}$ we define the following Eisenstein series
%\[
%\Eisen\bk{\chi,f_s,s,g} = \suml_{\gamma\in P\bk{F}\lmod G\bk{F}} f_s\bk{\gamma g} .
%\]
%This series is known to converge to a holomorphic function for $\Real\bk{s}\gg 0$ and admits a meromorphic continuation to the whole complex plane.

Let $F$ denote a number field with a set of places $\Places$ and a ring of adeles $\A=\A_F$.

Let $\bfG$ be a quasi-split, simple group of relative rank $n$, defined over $F$ and let $\bfB=\bfT\cdot \bfN$ denote a Borel subgroup of $\bfG$ with maximal torus $\bfT$ and unipotent radical $\bfN$, all defined over $F$.
Also denote by $\bfS\subset \bfT$ a maximal split torus defined over $F$.
We let $\Phi=\Phi\bk{\bfG,\bfS}$ denote the relative root system of $\bfG$ with respect to $\bfB$ with simple roots $\Delta=\Delta\bk{\bfG,\bfS}$.
For $\alpha\in\Phi$, we denote by $F_\alpha$ its field of definition.
Let $\Phi^{+}$ denote the set of positive roots in $\Phi$ with respect to $\bfB$.

Let $W=W\bk{\bfG,\bfB}$ denote the Weyl group of $\bfG$.
The Weyl group $W$ is generated by the simple reflections $w_\alpha$ along the simple roots $\alpha\in\Delta$.

We recall the correspondence
\[
\begin{array}{ccc}
\begin{Bmatrix}
\text{Subsets of $\Delta$}
\end{Bmatrix} 
&\longleftrightarrow &
\begin{Bmatrix}
\text{Standard parabolic}\\ \text{subgroups of $\bfG$}
\end{Bmatrix} \\
\Psi\subset\Delta & \longleftrightarrow & \bfP_\Psi
\end{array}
\]

Furthermore, for $\Psi\subset\Delta$, we write $\bfP_\Psi=\bfM_\Psi\cdot \bfU_\Psi$, where $\bfU_\Psi$ denote the unipotent radical of $\bfP_\Psi$ and  $\bfM_\Psi$ denotes its Levi subgroup.
We also write $\Delta_\bfM = \Psi$ for the set of simple roots of $M_\Psi$, $\Phi_M$ for the set of roots of $M_\Psi$ and $\Phi_M^{+}$ for the set of positive roots.
%For a subset $\Xi$ of $\Delta_E$ we denote by $P_\Xi$ the standard parabolic subgroup of $H_E$ whose Levi subgroup is generated by the simple roots in $\Xi$.
%The parabolic subgroup $P_\Xi$ admits a Levi subgroup $M_\Xi$, a unipotent radical $U_\Xi$ and a Levi decomposition $P_\Xi=M_\Xi\cdot U_\Xi$.

Fix a standard parabolic subgroup $\bfP=\bfP_\Psi$ of $\bfG$, with $\bfM=\bfM_\Psi$ and $\bfU=\bfU_\Psi$, all defined over $F$.
%We denote the split component of $\bfM$ by $\bfA$.
%\todonum{Do we need this? Perhaps split component of torus/center?}
%Let $X\bk{\bfA}_F$ denote the set of $F$-rational characters on $\bfA$ and
%\[
%M^{1} = \bigcap_{\chi\in X\bk{\bfA}_F}
%\]
Let $\mathfrak{a}_{\bfM,\C}^\ast=X^\ast\bk{\bfM}_F\otimes_\Z\C$, where $X^\ast\bk{\bfM}_F$ denote the $F$-rational characters of $\bfM$.
Also, let $W_\bfM=W\bk{\bfM,\bfM\cap\bfB}$ denote the relative Weyl group of $\bfM$ and note that the quotient $W_\bfM\lmod W$ is well defined; let $W\bk{\bfM,\bfG}$ denote the set of shortest representatives of the cosets $W_\bfM \lmod W$.

For any $\alpha\in\Delta$, let $\omega_\alpha$ denote the fundamental weight associated to $\alpha$.
The fundamental weights give rise to the isomorphism
\[
\begin{array}{ccc}
\C^n & \rightarrow & \mathfrak{a}_{\bfT,\C}^\ast \\
\bar{s}=\bk{s_\alpha}_{\alpha\in\Delta} & \mapsto & \lambda_{\bar{s}} = \suml_{\alpha\in\Delta} s_\alpha \omega_\alpha .
\end{array}
\]
%\todonum{Maybe also explain connection to $\mathfrak{a}_{\bfM,\C}^\ast$?}

Throughout, we denote the contragredient of a representation $\pi$ by $\pi^\ast$.
Let $\Id_G$ denote the trivial representation of $G$ and if there is no source of confusion, it will simply be denoted by $\Id$.
% on $F^\times\lmod\A^\times$.
Also, in any vector space $V$ over $\C$, we denote the zero vector by $\bar{0}$.

\subsection{Characters on Levi Subgroups and their Restriction to the Torus}

For a Levi subgroup $\bfM$ of $\bfG$, let $\bfX_\bfM$ denote the complex manifold of characters of $\bfM\bk{\A}$ trivial on $\bfM\bk{F}$.
There is a natural embedding of $\mathfrak{a}^\ast_{\bfM,\C}$ in $\bfX_\bfM$.
One can choose a direct sum complement $\bfX_\bfM = \mathfrak{a}^\ast_{\bfM,\C} \oplus \bfX_{\bfM,0}$, where the characters in $\bfX_{\bfM,0}$ are of finite order.

We note that the restriction from $\bfM$ to $\bfT$ gives rise to natural embeddings
%\todonum{Do I also care about the map in the other direction? See \cite[pg. 3]{MR2490651}}
\begin{equation}
\label{Eq:Inculsions_of_character_groups}
\iota_M:\bfX_\bfM \hookrightarrow \bfX_\bfT, \quad
X^\ast\bk{\bfM}\hookrightarrow X^\ast\bk{\bfT}, \quad
\mathfrak{a}^\ast_{\bfM,\C}\hookrightarrow \mathfrak{a}^\ast_{\bfT,\C}.
\end{equation}
The image of these embeddings can be identified by restriction to $\bfM^{der}$.
Namely, for $\chi\in \bfX_\bfT$ it holds that $\chi\in \iota_M\bk{\bfX_\bfM}$ if and only if $\gen{\chi,\check{\alpha}}=0$ for all $\alpha\in\Delta_\bfM$.
Similarly for $X^\ast\bk{\bfM}$ and $\mathfrak{a}^\ast_{\bfM,\C}$.

In particular, any element of $\lambda\in\mathfrak{a}^\ast_{\bfM,\C}$ is of the form
\[
\lambda=\suml_{\alpha\notin\Delta_\bfM} s_\alpha \omega_\alpha,
\]
where $s_\alpha\in\C$ and any $\chi\in\bfX_\bfM$ is of the form
\[
\chi = \suml_{\alpha\notin\Delta_\bfM} \chi_\alpha\circ \omega_\alpha,
\]
where $\chi_\alpha\in\bfX_{\mathbf{GL_1}}$.

\begin{Remark}
	Given the split component $A_\bfM$ of the center of $\bfM$, we have $A_\bfM\subset \bfT$.
	In particular, the restriction yields surjective maps
	\[
	r_M:\bfX_\bfT \twoheadrightarrow \bfX_{A_\bfM}, \quad
	X^\ast\bk{\bfT}\twoheadrightarrow X^\ast\bk{A_\bfM}, \quad
	\mathfrak{a}^\ast_{\bfT,\C}\twoheadrightarrow \mathfrak{a}^\ast_{A_\bfM,\C}.
	\]
	In particular, since $\mathfrak{a}^\ast_{A_\bfM,\C}\cong \mathfrak{a}^\ast_{\bfM,\C}$, the composition of the maps
	\[
	\mathfrak{a}^\ast_{\bfM,\C} \overset{\iota_M}{\hookrightarrow}
	\mathfrak{a}^\ast_{\bfT,\C} \overset{r_M}{\twoheadrightarrow}
	\mathfrak{a}^\ast_{\bfM,\C}
	\]
	is the identity map on $\mathfrak{a}^\ast_{\bfM,\C}$.
	For more details, see \cite[pg. 3]{MR2490651}.
\end{Remark}

%This supplies $\mathfrak{a}^\ast_{\bfM,\C}$ with a set of coordinates given as follows.
%For more details, see \cite[App. IV.6]{MR1721403}.

%We first recall that $\mathfrak{a}^\ast_{\bfT,\R} = X^\ast\bk{\bfT}\otimes_\Z\R$ is an inner product space, with base $\Delta$.
%We use $\bk{\cdot,\cdot}$ to denote its inner product.
%
%Let $\Delta^\ast=\set{\alpha^\ast\mvert \alpha\in\Delta}$ be the dual basis to $\Delta$. 
%Namely, assume that $\bk{\beta,\alpha^\ast}=\delta_{\alpha,\beta}$ for any $\alpha,\beta\in\Delta$.
%
%Choosing $\Psi\subset\Delta$, the orthogonal complement in $\mathfrak{a}^\ast_{\bfT,\R}$ of $Span_\R\bk{\Psi}$ is $Span_\R\set{\alpha^\ast\mvert \alpha\notin\Psi}$.
%Let
%\begin{itemize}
%	\item $\alpha^\ast_\Psi=\alpha^\ast$, if $\alpha\notin\Psi$.
%	\item $\alpha^\ast_\Psi$ be the projection of $\alpha^\ast$ to $Span_\R\bk{\Psi}$ if $\alpha\in\Psi$.
%\end{itemize}
%Now, let $\set{\alpha_\Psi}$ be the dual basis to $\set{\alpha^\ast_\Psi}$ in $\mathfrak{a}^\ast_{\bfT,\R}$.
%One checks that
%\begin{itemize}
%	\item $\alpha_\Psi=\alpha$, for any $\alpha\in\Psi$.
%	\item $\set{\alpha_\Psi\mvert \alpha\notin \Psi}$ is a basis for $X^\ast\bk{\bfM}$.
%\end{itemize}

\subsection{Degenerate Eisenstein Series}
We fix a Hecke character $\mu:M\bk{F}\lmod M\bk{\A}\to\C^\times$. We will usually assume that it is of finite order, i.e. a Dirichlet character.
For $\lambda\in\mathfrak{a}_{\bfM,\C}^\ast$, we consider the normalized parabolic induction
\[
I_{\bfP}\bk{\mu,\lambda} = \Ind_{\bfP\bk{\A}}^{\bfG\bk{\A}} \bk{\mu\otimes\lambda}.
\]
For a standard section $f_\lambda\in I_{\bfP}\bk{\mu,\lambda}$ we form the associated Eisenstein series
\begin{equation}
\label{Eq:Degenerate_Eisenstein_series_definition}
\Eisen_{\bfP}\bk{\mu,f,\lambda,g} = \suml_{\gamma\in \bfP\bk{F}\lmod \bfG\bk{F}} f_\lambda\bk{\gamma g} .
\end{equation}
This series converges for $\Real\bk{\lambda}\gg 0$ and admits a meromorphic continuation to $\mathfrak{a}_{\bfM,\C}^\ast$.
More precisely, let
\[
\mathfrak{F}_M^{+}=\set{\lambda \mvert \Real\bk{\gen{\lambda-\rho_{\mathbf{B}},\check{\alpha}}}>0 \quad \forall \alpha\in \Phi^{+}\setminus\Phi_M^{+}}.
\]
In particular, we write $\mathfrak{F}^{+}=\mathfrak{F}_\bfT^{+}$.
The series on the right hand-side of \Cref{Eq:Degenerate_Eisenstein_series_definition} converges if $\iota_M\bk{\lambda}\in \mathfrak{F}_M^{+}$.
%\todonum{This  is also the domain of convergence of the Eisenstein series. Move there and use that instead of $Re(s)>>0$.}

We note that the leading terms of this series are intertwining operators into the space of automorphic forms on $\bfG\bk{\A}$.
Namely, they give an automorphic realization to a quotient of $I\bk{\mu,\lambda}$.

Denote the half-sum of the roots in $\mathfrak{u}=Lie\bk{\bfU}$ by $\rho_{\bfP}$.
It is known that
\begin{equation}
\label{Eq:Induction_from_Parabolic_in_induction_from_Borel}
\begin{split}
& I_{\bfP}\bk{\mu,\lambda} \hookrightarrow 
\Ind_{\bfB\bk{\A}}^{\bfG\bk{\A}} \bk{\mu\otimes\lambda\otimes\FNorm{\rho_{\bfP}-\rho_{\bfB}}} \\
& \Ind_{\bfB\bk{\A}}^{\bfG\bk{\A}} \bk{\mu\otimes\lambda\otimes\FNorm{\rho_{\bfB}-\rho_{\bfP}}}  \twoheadrightarrow I_{\bfP}\bk{\lambda,\mu} ,
\end{split}
\end{equation}
where we implicitly used the inclusions in \Cref{Eq:Inculsions_of_character_groups}.
Note that, under this inclusion, $\rho_\bfB-\rho_\bfP=\rho_{\bfB\cap\bfM}$.
We note that for any $f\in I_{\bfP}\bk{\mu,\lambda}$ it holds that
\begin{equation}
\Eisen_{\bfP}\bk{\mu,f,\lambda,g} = \Eisen_{\bfB}\bk{\mu,f,\lambda\otimes\FNorm{\rho_{\bfP}-\rho_{\bfB}},g} .
\end{equation}
This is proven in \Cref{Prop:Equality_of_Eisenstein_series} bellow.

%\vertline
%
%Where $\mu$ on the right-hand side is the restriction of $\mu$ to $\bfT\bk{\A}$ and $\lambda$ on the right-hand side is the image of $\lambda$ under the natural inclusion $\mathfrak{a}_{\bfM,\C}^\ast\subset \mathfrak{a}_{\bfT,\C}^\ast$
%In particular, there exist a section $\tilde{f}\in I_{\bfB}\bk{\mu,\lambda\otimes\FNorm{\rho_{\bfP}-\rho_{\bfB}}}$ such that
%\begin{equation}
%\label{Eq:Equality_of_Degenerate_Eisenstein_Series_on_P_and_on_B}
%\Eisen_{\bfP}\bk{\mu,f,\lambda,g} = \Eisen_{\bfB}\bk{\mu,\tilde{f},\lambda\otimes\FNorm{\rho_{\bfP}-\rho_{\bfB}},g} .
%\end{equation}
%
%
%
%%There is also a section
%%$\tilde{\tilde{f}}\in I_{\bfB}\bk{\mu,\lambda\otimes\FNorm{\rho_{\bfB}-\rho_{\bfP}}}$ such that
%%\begin{equation}
%%\Eisen_{\bfP}\bk{\mu,f,\lambda,g} = \Res_{\lambda=\lambda_P} \Eisen_{\bfB}\bk{\mu,\tilde{\tilde{f}},\lambda\otimes\FNorm{\rho_{\bfP}/\rho_{\bfB}},g} .
%%\end{equation}
%%\todonum{rewrite this, if $P$ is not maximal this should actually be a consecutive residue}

Throughout this paper, we use the conventions for Dedekind $\zeta$-functions, Hecke $\Lfun$-functions and their $\epsilon$-factors, specified in \cite{SegalEisen}[Sec. 3.2].
In particular, given a number field $L$, $\zfun_L\bk{s}$ denote the $\zeta$-function of $L$ normalized so that it satisfies the functional equation $\zfun_L\bk{s}=\zfun_L\bk{1-s}$ and $\Lfun_L\bk{s,\chi}$ denotes the Hecke $\Lfun$-function of $\chi:L^\times\lmod\A_L^\times\to\C^\times$.

We also note that, if $\mu=\Id$, we may drop it from our notation.

%\begin{itemize}
%\item $UIQ$ - Unique irreducible quotient.
%\item $SUIQ$ - Spherical unique irreducible quotient.
%\item $UIS$ - Unique irreducible subrepresentation.
%\end{itemize}

\subsection{Intertwining Operators and the Constant Term Formula}
For $\lambda\in\mathfrak{a}_{\bfT,\C}^\ast$, a Hecke character $\mu:T\bk{F}\lmod T\bk{\A}\to\C^\times$ and $w\in W$ we consider the standard intertwining operator given by the integral
\begin{equation}
M\bk{w,\mu,\lambda}f_\lambda\bk{g} = \intl_{\bfN\bk{\A}\cap w^{-1}\bfN\bk{\A}w\lmod \bfN\bk{\A}} f_\lambda\bk{wug} du.
\end{equation}
This integral converges on the shifted positive Weyl chamber $\mathfrak{F}^{+}$ to a holomorphic family of operators and admits a meromorphic continuation to $\mathfrak{a}_{\bfT,\C}^\ast$.
At points of holomorphy, $M\bk{w,\mu,\lambda}$ defines an intertwining operator
\[
M\bk{w,\mu,\lambda} : I_{\bfB}\bk{\mu,\lambda} \to I_{\bfB}\bk{w^{-1}\cdot\mu,w^{-1}\cdot\lambda} .
\]

We note the following cocycle relation on the standard intertwining operators
\begin{Lem}
For any $w,w'\in W$ we have
\[
M\bk{ww',\mu,\lambda} = M\bk{w',w^{-1}\cdot\mu, w^{-1}\cdot\lambda} \circ M\bk{w,\mu,\lambda} .
\]
\end{Lem}

\begin{Remark}
	Note that is an alternative definition of $M\bk{w,\mu,\lambda}$ which yields a different cocycle relation.
	Namely, for
	\[
	\widetilde{M}\bk{w,\mu,\lambda}f_\lambda\bk{g} = \intl_{\bfN\bk{\A}\cap w\bfN\bk{\A}w^{-1}\lmod \bfN\bk{\A}} f_\lambda\bk{w^{-1}ug} du
	\]
	the following cocycle relation holds
	\[
	M\bk{ww',\mu,\lambda} = M\bk{w,w'\cdot\mu, w^{-1}\cdot\lambda} \circ M\bk{w',\mu,\lambda} .
	\]
\end{Remark}

\vspace{0.3cm}

The constant term of $\Eisen_\bfP\bk{\mu,f,\lambda,g}$ along $\bfB$ is given by
\[
\Eisen_\bfP\bk{\mu,f,\lambda,g}_\bfCT = \intl_{\bfN\bk{F}\lmod \bfN\bk{\A}} \Eisen_\bfP\bk{\mu,f,\lambda,ug} \, du .
\]
When restricted to $\bfT\bk{\A}$, this is an automorphic form on $\bfT\bk{\A}$; however, it is beneficial to consider this also as a function of $\bfG\bk{\A}$; in particular
\[
f_\lambda \mapsto \Eisen_\bfP\bk{\mu,f,\lambda,\cdot}_\bfCT
\]
is a $\bfG\bk{\A}$-equivariant map.

The constant term formula, computed as in \cite{MR1469105}, is given as follows:
\begin{equation}
\label{Eq:Constant_term}
\Eisen_\bfP\bk{\mu,f,\lambda,g}_\bfCT = \suml_{w\in W\bk{\bfM,\bfG}} M\bk{w,\mu,\lambda} f_\lambda \bk{g}.
\end{equation}

\iffalse
For the readers convinience, we repeat the proof.
\begin{proof}
	By the definition of $\Eisen_\bfP\bk{\mu,f,\lambda,g}$ and the constant term, we have
	\begin{align*}
	\Eisen_\bfP\bk{\mu,f,\lambda,g}_\bfCT 
	& = \intl_{\bfN\bk{F}\lmod \bfN\bk{\A}} \Eisen_\bfP\bk{\mu,f,\lambda,ug} \, du \\
	& = \intl_{\bfN\bk{F}\lmod \bfN\bk{\A}} \bk{\suml_{\gamma\in \bfP\bk{F}\lmod \bfG\bk{F}} f_\lambda\bk{\gamma u g}} \, du \\
	& = \intl_{\bfN\bk{F}\lmod \bfN\bk{\A}} \bk{\suml_{w\in \bfP\bk{F}\lmod \bfG\bk{F}\rmod \bfB\bk{F}} \quad \suml_{\delta \in  \bk{w^{-1}\bfP\bk{F}w\cap \bfP\bk{F}}\lmod\bfB\bk{F}} f_\lambda\bk{w \delta u g}} \, du \\
	\end{align*}
	
\end{proof}
\fi
%\todonum{Consider adding a proof}

A simple application of the constant term formula is the following useful well-known result.
\begin{Prop}
	\label{Prop:Equality_of_Eisenstein_series}
	%	For any $f\in I_{\bfP}\bk{\mu,\lambda}$	there exists a section $\tilde{f}\in I_{\bfB}\bk{\mu,\lambda\otimes\FNorm{\rho_{\bfP}-\rho_{\bfB}}}$ such that
	%	\begin{equation}
	%	\label{Eq:Equality_of_Degenerate_Eisenstein_Series_on_P_and_on_B}
	%	\Eisen_{\bfP}\bk{\mu,f,\lambda,g} = \Eisen_{\bfB}\bk{\mu,\tilde{f},\lambda\otimes\FNorm{\rho_{\bfP}-\rho_{\bfB}},g} .
	%	\end{equation}
	For any $f\in I_{\bfP}\bk{\mu,\lambda}$	it holds that
	\begin{equation}
	\label{Eq:Equality_of_Degenerate_Eisenstein_Series_on_P_and_on_B}
	\Eisen_{\bfP}\bk{\mu,f,\lambda,g} = \Eisen_{\bfB}\bk{\mu,f,\lambda\otimes\FNorm{\rho_{\bfP}-\rho_{\bfB}},g} .
	\end{equation}
\end{Prop}

\begin{proof}
	We recall the inclusion
	\[
	I_{\bfP}\bk{\mu,\lambda} \hookrightarrow 
	\Ind_{\bfB\bk{\A}}^{\bfG\bk{\A}} \bk{\mu\otimes\lambda\otimes\FNorm{\rho_{\bfP}-\rho_{\bfB}}},
	\]
	which follows from induction in parts.
	Namely,
	\[
	\Ind_{\bfB\bk{\A}}^{\bfG\bk{\A}} \bk{\mu\otimes\lambda\otimes\FNorm{\rho_{\bfP}-\rho_{\bfB}}}
	= \Ind_{\bfP\bk{\A}}^{\bfG\bk{\A}} \bk{\Ind_{\bfB\bk{\A}\cap \bfM\bk{\A}}^{\bfM\bk{\A}} \FNorm{\rho_{\bfP}-\rho_{\bfB}}} \otimes \mu\otimes\lambda.
%	 \supset I_{\bfP}\bk{\mu,\lambda} .
	\]
	Since the trivial representation of $\bfM\bk{\A}$ is the unique irreducible subrepresentation of $\Ind_{\bfB\bk{\A}\cap \bfM\bk{\A}}^{\bfM\bk{\A}} \FNorm{\rho_{\bfP}-\rho_{\bfB}}$,
	it is the kernel of $M\bk{w,\mu,\lambda}$ for any non-trivial $w\in W_\bfM$.
	
	This can be rephrased as follows.
	For any $w\notin W\bk{\bfM,\bfG}$ it holds that $M\bk{w,\mu,\lambda}f=0$.
	It follows that we have an equality of the constant terms of the two sides of \Cref{Eq:Equality_of_Degenerate_Eisenstein_Series_on_P_and_on_B} since both are equal to
	\[
	\suml_{w\in W\bk{\bfM,\bfG}} M\bk{w,\mu,\lambda} f_\lambda \bk{g}.
	\]
	We consider the difference between the left and right-hand sides in \Cref{Eq:Equality_of_Degenerate_Eisenstein_Series_on_P_and_on_B},
	\[
	\Eisen_{\bfP}\bk{\mu,f,\lambda,g} - \Eisen_{\bfB}\bk{\mu,f,\lambda\otimes\FNorm{\rho_{\bfP}-\rho_{\bfB}},g} .
	\]
	
	By construction, the cuspidal support (see \cite[pg. 38]{MR1361168}) of both $I_{\bfP}\bk{\mu,\lambda}$ and $I_{\bfB}\bk{\mu,\lambda\otimes\FNorm{\rho_{\bfP}-\rho_{\bfB}}}$ lies along $\bfT$ and hence, the above difference of Eisenstein series is cuspidal.
	However, the cuspidal spectrum is orthogonal to Eisenstein series and hence \[
	\Eisen_{\bfP}\bk{\mu,f,\lambda,g} - \Eisen_{\bfB}\bk{\mu,f,\lambda\otimes\FNorm{\rho_{\bfP}-\rho_{\bfB}},g} = 0 .
	\]
	
%	It is both a cuspidal form and a linear combination of Eisenstein series.
%	\todonum{Rephrase this sentence. See how it was stated in \Cref{NonvanishingifCTisNonvanishing}}
%	Since cusp forms are orthogonal to Eisenstein series, the equality holds.
\end{proof}

\begin{Remark}
	One can similarly prove the following, closely related, fact.
	For any $f\in I_{\bfP}\bk{\mu,\lambda}$	there exists a section $\widetilde{f}\in I_{\bfB}\bk{\mu,\lambda\otimes\FNorm{\rho_{\bfP}-\rho_{\bfB}}}$ such that
	$M\bk{w_{\bfM,l},\mu,\lambda}f=\widetilde{f}$ and
	\begin{equation}
	\Eisen_{\bfP}\bk{\mu,f,\lambda,g} =
	\bk{\prodl_{\alpha\in\Phi_\bfM^{+}} \bk{\gen{\lambda,\check{\alpha}}-1}} \Eisen_{\bfB}\bk{\mu,\tilde{f},\lambda\otimes\FNorm{\rho_{\bfB}-\rho_{\bfP}},g} .
	\end{equation}
\end{Remark}

\subsection{Local and Global Intertwining Operators}

For a Hecke character $\mu:M\bk{F}\lmod M\bk{\A}\to\C^\times$, we write
\[
\mu = \bigotimes_{\nu\in\Places}\,' \mu_\nu,
\]
where $\mu_\nu:M\bk{F_\nu}\to\C^\times$ is unramified for almost all $\nu\in\Places$.

For a place $\nu\in\Places$, we consdier the degenerate principal series
\[
I_{\bfP,\nu}\bk{\mu,\lambda} = \Ind_{\bfP\bk{F_\nu}}^{\bfG\bk{F_\nu}} \bk{\mu_\nu\otimes\lambda}
\]
of $\bfG\bk{F_\nu}$.
The global degenerate principal series is a restricited tensor product of local ones.
Namely,
\[
I_{\bfP}\bk{\mu,\lambda} = \bigotimes_{\nu\in\Places}\,' I_{\bfP,\nu}\bk{\mu,\lambda} .
\]

For a place $\nu\in\Places$, $\lambda\in\mathfrak{a}_{\bfT,\C}^\ast$, a unitary character $\mu_\nu:T\bk{F_\nu}\to\C^\times$, a section $f_{\lambda,\nu}\in I_{\bfB}\bk{\mu,\lambda}$ and $w\in W$ we consider the standard local intertwining operator given by the integral
\begin{equation}
\label{Eq:Local_Intertwining_Operator_Def}
M_\nu\bk{w,\mu_\nu,\lambda} f_{\lambda,\nu}\bk{g} = \intl_{\bfN\bk{F_\nu}\cap w^{-1}\bfN\bk{F_\nu}w\lmod \bfN\bk{F_\nu}} f_{\lambda,\nu}\bk{wng} dn .
\end{equation}
This integral converges absolutely to an analytic function in $\mathfrak{F}^{+}$ and admits a meromorphic continuation to $\mathfrak{a}_{\bfT,\C}^\ast$.
Furthermore, it holds that the global intertwining operator decomposes into a restricted tensor product of local operators
\[
M\bk{w,\mu,\lambda} = \bigotimes_{\nu\in\Places}\,' M_\nu\bk{w,\mu_\nu,\lambda} .
\]
Namely, given a pure tensor $f_\lambda = \otimes' f_{\lambda,\nu}$ it holds that
\begin{equation}
\label{Eq:Decomposition_of_global_intertwining_operator_to_local_ones}
M\bk{w,\mu,\lambda}f_\lambda = \bigotimes_{\nu\in\Places}\,' M_\nu\bk{w,\mu_\nu,\lambda}f_{\lambda,\nu}
% \quad \forall \lambda\in \mathfrak{F}^{+} .
\end{equation}

%\[
%M\bk{w,\mu,\lambda} : I_{\bfB}\bk{\mu,\lambda} \to I_{\bfB}\bk{w^{-1}\cdot\mu,w^{-1}\cdot\lambda}
%\]

\subsection{Decomposition into Rank-1 Operators}

The set of intertwining operators satisfies a cocycle condition.
Namely, for any $w, w'\in W$ it holds that
\begin{equation}
\label{eq:Global_Intertwining_Operator_Cocycle_Condition}
M\bk{ww',\mu,\lambda} = M\bk{w',w^{-1}\cdot \mu,w^{-1}\cdot\lambda} \circ M\bk{w,\mu,\lambda} .
\end{equation}

In particular, writing $w=w_{i_1}w_{i_2}\cdot...\cdot w_{i_k}$, where the $w_{i_j}$ are simple reflections, it holds that
\begin{equation}
\label{Eq:Decomposition_of_intertwining_operators_to_simple_reflections}
\begin{split}
 M\bk{w,\mu,\lambda}
& = M\bk{w_{i_k},\bk{w_{i_1}\cdot...\cdot w_{i_{k-1}}}^{-1}\cdot\mu,\bk{w_{i_1}\cdot ... \cdot w_{i_{k-1}}}^{-1}\cdot\lambda} \\ 
 & \circ..\circ
M\bk{w_{i_2},w_{i_1}^{-1}\cdot\mu,w_{i_1}^{-1}\cdot\lambda} \circ
M\bk{w_{i_1},\mu,\lambda}
\end{split}
\end{equation}

This allows us to reduce many calculations to a sequence of calculations in  rank-1.

Let $\mathcal{B}=\mathcal{T}\cdot\mathcal{N}$ be the Borel subgroup of $SL_2$ with torus $\mathcal{T}$ and unipotent radical $\mathcal{N}$.
Also let $\dot{w}=\begin{pmatrix}0&1\\-1&0\end{pmatrix}$ be a representative in $SL_2$ of the non-trivial element in the Weyl group of $SL_2$.

For any simple root $\alpha$, let $\iota_\alpha:SL_2\to\bfG$ denote the associated structure map.

\begin{Lem}
	\label{Lemma:Intertwining_operator_of_simple_reflections}
	For a place $\nu\in\Places$, the following diagram is commutative
	\[
	\xymatrix{
		I_{\bfB,\nu}\bk{\mu,\lambda} \ar@{->}[r]^{M_\nu\bk{w,\mu_\nu,\lambda}} \ar@{->}[d]_{\iota_\alpha^\ast} &
		I_{\bfP,\nu}\bk{w_\alpha^{-1}\cdot\mu,w_{\alpha}^{-1}\cdot\lambda} \ar@{->}[d]^{\iota_\alpha^\ast} \\
		\Ind_{\mathcal{B}\bk{F_\nu}}^{SL_2\bk{F_\nu}} \bk{\coset{\mu_{\nu} \otimes \lambda}\circ\iota_\alpha} \ar@{->}[r]^{M_{\dot{w}}} &
		\Ind_{\mathcal{B}\bk{F_\nu}}^{SL_2\bk{F_\nu}} \bk{w_\alpha^{-1}\cdot\coset{\mu_{\nu} \otimes \lambda}\circ\iota_\alpha} \, ,
	}
	\]	
	%	\[
	%	\xymatrix{
	%		I_{B_E}\bk{\chi_{\nu},\lambda} \ar@{->}[r]^{M_\nu\bk{w_\alpha,\chi_\nu,\lambda}} \ar@{->}[d]_{\iota_\alpha^\ast} &
	%		I_{B_E}\bk{w_\alpha\cdot \chi_\nu, w_\alpha\cdot \lambda} \ar@{->}[d]^{\iota_\alpha^\ast} \\
	%		\Ind_{\mathcal{B}\bk{F_\nu}}^{SL_2\bk{F_\nu}} \bk{\chi_{\nu} \res{T_E}\otimes \lambda} \ar@{->}[r]^{M_{\dot{w}}} &
	%		\Ind_{\mathcal{B}\bk{F_\nu}}^{SL_2\bk{F_\nu}} \bk{w_\alpha\cdot\chi_{\nu} \res{T_E}\otimes\lambda}} ,
	%	\]
	where the vertical maps should be understood as the pull-back map.
%	By restriction to $I_{P_E}\bk{\chi,s}$, this is also true for $M\bk{w_\alpha,\chi_s}$.
\end{Lem}

%\subsection{The Rank-one Gindikin-Karpelevich Formula}
\subsection{Normalized Intertwining Operators}

For $\nu\in\Places$, $s\in\C$ and a unitary character $\sigma_\nu:\mathcal{T}\bk{F_\nu}\to\C^\times$, let $f_{s,\nu}^0\in \Ind_{\mathcal{B}\bk{F_\nu}}^{SL_2\bk{F_\nu}} \bk{\sigma_\nu\otimes \bk{s\cdot\rho_{\mathcal{B}}}}$ denote the spherical vector normalized so that $f_{s,\nu}^0\bk{1}=1$.
The rank-one Gindikin-Karpelevich formula states that
\begin{equation}
\label{Eq: Rank one Gindikin Karpelevich}
M\bk{\dot{w},\sigma_\nu,s} f_{s,\nu}^0 = \frac{\Lfun_{F_\nu}\bk{s,\sigma_\nu}}{\Lfun_{F_\nu}\bk{s+1,\sigma_\nu}} f_{-s,\nu}^0 .
\end{equation}
This formula suggests a normalization of intertwining operators so that normalized spherical vectors are sent to normalized spherical vectors.

For a unitary character $\mu_\nu:\bfT\bk{F_\nu}\to\C^\times$ and $\lambda\in\mathfrak{a}_{\bfT,\C}^\ast$, define the \emph{normalized intertwining operator} to be
\begin{equation}
\label{eq:Normalized_intertwining_operators}
N_\nu\bk{w,\mu_\nu,\lambda} =
\prodl_{\alpha>0,\ w^{-1}\alpha<0} \frac{\Lfun_{F_{\alpha,\nu}}\bk{\gen{\lambda,\check{\alpha}}+1 ,\mu_\nu\circ\check{\alpha}} }{ \Lfun_{F_{\alpha,\nu}}\bk{\gen{\lambda,\check{\alpha}} ,\mu_\nu\circ\check{\alpha}} \epsilon_{F_{\alpha,\nu}}\bk{\gen{\lambda,\check{\alpha}},\mu_\nu\circ\circ\check{\alpha}, \psi_\nu}}
M\bk{w,\mu_\nu,\lambda} .
\end{equation}

%It is customary to 
%\begin{equation}
%\label{eq:Normalized_intertwining_operators}
%N_\nu\bk{w,\mu,\lambda} =
%\frac{J_\nu\bk{w,\mu,\lambda}^{-1}}{\prodl_{\alpha>0,\ w^{-1}\alpha<0} \epsilon_{F_{\alpha,\nu}}\bk{\gen{\lambda,\check{\alpha}},\mu_\nu\circ\circ\check{\alpha}, \psi_\nu}}
%M\bk{w,\chi_\nu,\lambda} .
%\end{equation}

The normalized intertwining operator satisfy a cocycle condition similar to \Cref{eq:Global_Intertwining_Operator_Cocycle_Condition}.
Namely, for any $w, w'\in W$, it holds that
\begin{equation}
\label{eq:Local_Normalized_Intertwining_Operator_Cocycle_Condition}
N_\nu\bk{ww',\mu_\nu,\lambda} = N_\nu\bk{w',w^{-1}\cdot\mu_\nu,w^{-1}\cdot\lambda}\circ N_\nu\bk{w,\mu_\nu,\lambda} .
\end{equation}
%\todonum{Consider explaining about notations sometime used when the intertwining opertor is restricted to a degenerate principal series contained in a principal series}
%For simplicity we write:
%\[
%N_\nu\bk{w,\chi_s} = N_\nu\bk{w,\chi_\nu,\lambda_s} .
%\]

Recall that for $\sigma_\nu$ unramified, it holds that $\epsilon_{F_\nu}\bk{s,\sigma_\nu,\psi_\nu}=1$.
If $\mu_\nu$ is unramified, let $f_{\lambda,\nu}^0$ be the spherical vector in $I_{\bfB,\nu}\bk{\mu_\nu,\lambda}$ so that $f_{\lambda,\nu}^0\bk{1}=1$.
It follows from \Cref{eq:Normalized_intertwining_operators}, \Cref{eq:Local_Normalized_Intertwining_Operator_Cocycle_Condition} and \Cref{Eq: Rank one Gindikin Karpelevich}, that
\begin{equation}
\label{Eq:Normalized_intertwining_operator_acting_on_spherical_vector}
N_\nu\bk{w,\mu,\lambda} f_{\lambda,\nu}^0 = f_{w^{-1}\cdot\lambda,\nu}^0.
\end{equation}
%\todonum{Make sure that references here are the ones I want and mention that $\epsilon$-factors for unramified characters are trivial}

%\begin{equation}
%M\bk{ww',\mu,\lambda} = M\bk{w',w^{-1}\cdot \mu,w^{-1}\cdot\lambda} \circ M\bk{w,\mu,\lambda} .
%\end{equation}

\subsection{The Gindikin-Karpelevich Formula}

For $\lambda\in\mathfrak{a}_{\bfT,\C}^\ast$ and a Hecke character $\mu:T\bk{F}\lmod T\bk{\A}\to\C^\times$ we define the \emph{local Gindikin-Karpelevich factor} to be
\begin{equation}
J_\nu\bk{w,\mu_\nu,\lambda}
= \prodl_{\alpha>0,\ w^{-1}\alpha<0} \frac{\Lfun_{F_{\alpha,\nu}}\bk{\gen{\lambda,\check{\alpha}} ,\mu_\nu\circ\check{\alpha}}}{\Lfun_{F_{\alpha,\nu}}\bk{\gen{\lambda,\check{\alpha}}+1 ,\mu_\nu\circ\check{\alpha}}}
\end{equation}
and the  \emph{global Gindikin-Karpelevich factor} is defined by
\begin{equation}
J\bk{w,\mu,\lambda} = \prodl_{\nu\in\Places}
J_\nu\bk{w,\mu_\nu,\lambda} 
= \prodl_{\alpha>0,\ w^{-1}\alpha<0} \frac{\Lfun_{F_{\alpha}}\bk{\gen{\lambda,\check{\alpha}} ,\mu\circ\check{\alpha}}}{\Lfun_{F_{\alpha}}\bk{\gen{\lambda,\check{\alpha}}+1 ,\mu\circ\check{\alpha}}} .
\end{equation}

%\todonum{Introduce notations for zeta functions and Hecke L-functions}

Let $f_\lambda = \otimes' f_{\lambda,\nu}$ be a pure tensor in $I_{\bfB}\bk{\mu,\lambda}$ and let $S\subset\Places$ be a finite set so that $f_{\lambda,\nu}=f_{\lambda,\nu}^0$ for all $\nu\notin S$.

By \Cref{Eq:Decomposition_of_global_intertwining_operator_to_local_ones} and
\Cref{Eq:Normalized_intertwining_operator_acting_on_spherical_vector}, for any $w\in W$ and $\lambda\in\mathfrak{a}_{\bfT,\C}^\ast$ it holds that
%\todonum{After rewriting the subsection about GK-formula, double check that these references are the ones I want}
\begin{equation}
\label{Eq: Global Gindikin-Karpelevich}
\begin{split}
& M\bk{w,\mu,\lambda} f_{\lambda}
 = \bk{\displaystyle\operatorname*{\otimes}_{\nu\in S} M_\nu\bk{w,\mu_\nu,\lambda} f_{\lambda,\nu} }\bigotimes \bk{\displaystyle\operatorname*{\otimes}_{\nu\notin S} J_\nu\bk{w,\mu_\nu,\lambda} f_{\lambda,\nu}^0} \\
& = J\bk{w,\mu,\lambda} \bk{\displaystyle\operatorname*{\otimes}_{\nu\in S}  J_\nu\bk{w,\mu_\nu,\lambda}^{-1} M_\nu\bk{w,\mu,\lambda}f_{\lambda,\nu} }\bigotimes \bk{\displaystyle\operatorname*{\otimes}_{\nu\notin S} f_{\lambda,\nu}^0} \\
& = \bk{\prodl_{\alpha>0,\ w^{-1}\alpha<0} \epsilon_{F_\alpha}\bk{\gen{\lambda,\check{\alpha}},\mu_\nu\circ\check{\alpha}}} J\bk{w,\mu,\lambda} \bk{\displaystyle\operatorname*{\otimes}_{\nu\in S}  N_\nu\bk{w,\mu_\nu,\lambda}f_{\lambda,\nu} }\bigotimes \bk{\displaystyle\operatorname*{\otimes}_{\nu\notin S} f_{\lambda,\nu}^0} \ .
\end{split}
\end{equation}

This implies that the analytic behaviour of $M\bk{w,\mu,\lambda} f_{\lambda}$ depends on the analytic behaviour of $J\bk{w,\mu,\lambda}$ and the $N_\nu\bk{w,\mu_\nu,\lambda}f_{\lambda,\nu}$ for $\nu\notin S$.

Since the partially normalized intertwining operators
\[
\frac{1}{\prodl_{\alpha>0,\ w^{-1}\alpha<0} \Lfun_{F_{\alpha,\nu}}\bk{\gen{\lambda,\check{\alpha}} ,\mu_\nu\circ\check{\alpha}}} M_\nu\bk{w,\mu_\nu,\lambda}
\]
are entire for all $\nu\in\Places$ (see \cite{MR944102} when $\nu\vert\infty$ and \cite{MR517138} when $\nu\not\vert\infty$), it follows that $N_\nu\bk{w_\alpha,\mu_\nu,\lambda}$ is holomorphic whenever $\Real\bk{\gen{\lambda,\check{\alpha}}}>-1$ and $\alpha\in\Delta$.

\subsection{Degenerate Eisenstein Series Attached to Maximal Parabolic Subgroups}
Let $\bfP=\bfM\cdot\bfU$ be the maximal parabolic subgroup of $\bfG$ associated to the set $\Delta\setminus\set{\alpha}$.
The space of characters $\mathfrak{a}_{\bfM,\C}^\ast$ is one-dimensional and we fix an isomorphism
\[
\begin{array}{ccl}
\C & \overset{\sim}{\to} & \mathfrak{a}_{\bfM,\C}^\ast \\
s & \mapsto & \Omega_{\bfP,s}=s\cdot\omega_\alpha .
\end{array}
\]
%\todonum{Do I have better notation than $\FNorm{s\cdot\omega_\alpha}$ for this?}
Also, for a Hecke character $\chi:F^\times\lmod\A^\times\to\C^\times$, we denote
\[
\mu_\chi = \chi\circ\omega_\alpha.
\]
In what follows, we replace $\Omega_{\bfP,s}$ by $s$ and $\mu_\chi$ by $\chi$ in all notations, e.g. $I_{\bfP}\bk{\chi,s}=I_{\bfP}\bk{\mu_\chi,\Omega_{\bfP,s}}$.

Revisiting \Cref{Eq:Induction_from_Parabolic_in_induction_from_Borel}, for
\begin{equation}
\begin{split}
& \lambda_s = \Omega_{\bfP,s}\otimes \FNorm{\rho_{\bfP}-\rho_{\bfB}}\\
& \eta_s = \Omega_{\bfP,s}\otimes \FNorm{\rho_{\bfB}-\rho_{\bfP}}, \\
\end{split}
\end{equation}
it holds that
\begin{equation}
\begin{split}
& I_{\bfP}\bk{\chi,s} \hookrightarrow 
I_{\bfB}\bk{\mu_\chi\otimes\lambda_s} \\
& I_{\bfB}\bk{\mu_\chi\otimes\eta_s} \twoheadrightarrow I_{\bfP}\bk{\chi,s} .
\end{split}
\end{equation}

In what follows, we write
\begin{equation}
\chi_s = \lambda_s\otimes\mu\chi
\end{equation}
and $M\bk{w,\chi_s}$ for the restriction of $M\bk{w,\mu_\chi,\lambda_s}$ to $I_{\bfP}\bk{\chi,s}$.

%\todonum{Introduce the line $\lambda_s$}

%\begin{Lem}
%\label{NonvanishingifCTisNonvanishing}
%Let $P$ be a maximal parabolic subgroup of $G$ and let $f_s\in I_{\bfP}\bk{\chi,s}$ be a holomorphic section.
%For $s_0\in\C$ assume $\Eisen_P\bk{\chi,f,s}$ admits a pole of order $l$ and let
%\[
%\varphi\bk{g} = \lim_{s=s_0} \coset{\bk{s-s_0}^l \Eisen_P\bk{\chi,f,s,g}} .
%\]
%Let $\varphi_{const}$ denote the constant term of $\varphi$ along $N$.
%Then $\varphi\equiv 0$ if and only if $\varphi_{const}\equiv 0$.
%\end{Lem}
\begin{Lem}
	\label{NonvanishingifCTisNonvanishing}
	Let $f_s\in I_{\bfP}\bk{\chi,s}$ be a holomorphic section.
	Assume that $\Eisen_{\bfP}\bk{\chi,f,s}$ admits a pole of order $m$ at $s_0\in\C$ and let
	\[
	\varphi\bk{g} = \lim_{s=s_0} \coset{\bk{s-s_0}^m \Eisen_{\bfP}\bk{\chi,f,s,g}} .
	\]
	Let $\varphi_{\bfCT}$ denote the constant term of $\varphi$ along $\bfN$.
	Then, $\varphi\equiv 0$ if and only if $\varphi_{\bfCT}\equiv 0$.
\end{Lem}

\begin{proof}
	the proof here is similar to that of \Cref{Prop:Equality_of_Eisenstein_series}.	
	It is enough to prove that if $\varphi_{\bfCT}\equiv 0$ then $\varphi\equiv 0$.
	By the construction, the cuspidal support of $I_{\bfP}\bk{\chi,s}$ lies along $\bfB$ and hence, $\varphi_{\bfCT}\equiv 0$ implies that $\varphi$ is cuspidal.
	However, the cuspidal spectrum is orthogonal to Eisenstein series and hence $\varphi\equiv 0$.
	
	%As $\varphi_{\bfCT}\equiv 0$ and $\varphi$ is supported on $B$ then $\varphi$ is cuspidal and hence rapidly decreasing.
	%% In $GL_n$ perhaps it has a non-trivial central character, but for a simple group this is irrelevant.
	%We let
	%\[
	%I_s = \intl_{G\bk{F}\lmod G\bk{\A}} \overline{\varphi\bk{g}} \bk{s-s_0}^l \Eisen_P\bk{\chi,f,s,g} \ dg .
	%\]
	%This integral defines a meromorphic function on $\C$.
	%Namely, for any $s\in\C$ where $\Eisen_P\bk{\chi,f,s,g}$ is analytic the integral converges to an analytic function in a neighborhood of $s$.
	%
	%Our aim is to show that $I_{s_0} = 0$.
	%In fact, we shall show that $I_s = 0$ for any $s\in\C$.
	%
	%Indeed, for $Re\bk{s}\gg 0$ it holds that
	%\[
	%\Eisen_P\bk{\chi,f,s,g} = \suml_{\gamma\in P\bk{F}\lmod G\bk{F}} f_s\bk{\gamma g}
	%\]
	%and hence, by absolute convergence,
	%\begin{align*}
	%I_s
	%& = \bk{s-s_0}^l \intl_{P\bk{F}\lmod G\bk{\A}} \overline{\varphi\bk{g}} f_s\bk{g} \\
	%& = \bk{s-s_0}^l \intl_{P\bk{F}\lmod G\bk{\A}} \intl_{M\bk{F}\lmod M\bk{\A}} \intl_K \modf{P}^{-1}\bk{m} f_s\bk{mk} \coset{\integral{U} \overline{\varphi\bk{umk}} \ du}\ dm\ dk .
	%\end{align*}
	%By assumption it follows that
	%\[
	%\integral{U} \overline{\varphi\bk{umk}} \ du = \overline{\varphi_{\bfCT}\bk{mk}} \equiv 0
	%\]
	%and hence $I_s=0$ for $Re\bk{s}\gg 0$ and by meromorphicity for any $s\in\C$.
\end{proof}

\begin{Cor}
	\label{Cor:Kernel_of_Series_is_kernel_of_CT}
	Under the assumptions of the previous lemma,
	\begin{equation}
	\label{Eq:Kernel_of_Series_is_kernel_of_CT}
	\begin{array}{l}
	Span_{\C}\set{\lim\limits_{s\to s_0} \bk{s-s_0}^m \Eisen_P\bk{\chi,f_s} \mvert f_s\in I_{\bfP}\bk{\chi,s}} \\
	\cong Span_{\C} \set{\lim\limits_{s\to s_0} \bk{s-s_0}^m \suml_{w\in W_M\lmod W} M\bk{w,\chi_s}f_s\res{\bfT\bk{\A}} \mvert f_s\in I_{\bfP}\bk{\chi,s}}
	%Span_{\C} \set{\lim\limits_{s\to s_0} \bk{s-s_0}^l \suml_{w\in \Sigma_{\bk{E,\chi,s_0,l}}} M\bk{w,\chi_s}f_s}
	\end{array}
	\end{equation}
	%Under the assumptions of the previous lemma, we have an isomorphism of vector spaces
	%\begin{align*}
	%& Span_{\C}\set{\lim\limits_{s\to s_0} \bk{s-s_0}^l \Eisen_P\bk{\chi,f_s} \mvert f_s\in I_{\bfP}\bk{\chi,s}} \\
	%& \cong Span_{\C} \set{\lim\limits_{s\to s_0} \bk{s-s_0}^l \suml_{w\in W_M\lmod W} M\bk{w,\chi_s}f_s\res{\bfT\bk{\A}} \mvert f_s\in I_{\bfP}\bk{\chi,s}}
	%%Span_{\C} \set{\lim\limits_{s\to s_0} \bk{s-s_0}^l \suml_{w\in \Sigma_{\bk{E,\chi,s_0,l}}} M\bk{w,\chi_s}f_s}
	%\end{align*}
\end{Cor}

\begin{proof}
	We consider two maps $\textbf{R}$ and $\textbf{R}_\bfCT$ given by
	\begin{align*}
	I_{\bfP}\bk{\chi,s} & \overset{\textbf{R}}{\longrightarrow} Span_{\C}\set{\lim\limits_{s\to s_0} \bk{s-s_0}^m \Eisen_P\bk{\chi,f_s} \mvert f_s\in I_{\bfP}\bk{\chi,s}} \\
	f_s & \mapsto \varphi = \lim\limits_{s\to s_0} \bk{s-s_0}^m \Eisen_P\bk{\chi,f_s} \\
	I_{\bfP}\bk{\chi,s} & \overset{\textbf{R}_\bfCT}{\longrightarrow} Span_{\C} \set{\lim\limits_{s\to s_0} \bk{s-s_0}^m \suml_{w\in W_M\lmod W} M\bk{w,\chi_s}f_s\res{\bfT\bk{\A}} \mvert f_s\in I_{\bfP}\bk{\chi,s}} \\
	f_s & \mapsto \varphi_{\bfCT} = \lim\limits_{s\to s_0} \bk{s-s_0}^m \suml_{w\in W_M\lmod W} M\bk{w,\chi_s}f_s\res{\bfT\bk{\A}} .
	\end{align*}
	%where $\textbf{R}$ is a $G$-map and $\textbf{R}_\bfCT$ is a $T$-map.
	Obviously, $\textbf{R}_\bfCT= \bfCT\circ \textbf{R}$ and by the previous lemma $\ker\bk{\textbf{R}}=\ker\bk{\textbf{R}_\bfCT}$.
%	\todonum{Decide on the notation of kernels and make sure it is consistent}
	It follows that both sides of \Cref{Eq:Kernel_of_Series_is_kernel_of_CT} are isomorphic to
	\[
	I_{\bfP}\bk{\chi,s}\rmod ker\bk{\textbf{R}} \cong I_{\bfP}\bk{\chi,s}\rmod ker\bk{\textbf{R}_\bfCT} .
	\]
	
%	The claim then follows.
\end{proof}

\subsection{The Constant Term Formula Revisited}
Let $\bfP=\bfM\cdot\bfU$ be the maximal parabolic subgroup of $\bfG$ associated to the set $\Delta\setminus\set{\alpha}$.
We consider \Cref{Eq:Constant_term} in light of \Cref{Eq: Global Gindikin-Karpelevich}.

For $s_0\in\C$ and a Hecke character $\chi:F^\times\lmod \A^\times\to\C^\times$ let
\[
n=\sup\set{\operatorname{ord}_{s=s_0} M\bk{w,\chi,\lambda_s}f_s\bk{g} \mvert w\in W\bk{\bfM,\bfG},\ f_s\in I_{\bfP}\bk{\chi,s},\ g\in \bfG\bk{\A}},
\]
where the order $\operatorname{ord}_{s=s_0}h\bk{s}$ of a pole of a complex function $h\bk{s}$ at $s_0$ is the unique integer $n$ such that
\[
\lim\limits_{s\to s_0}\bk{s-s_0}^n h\bk{s}\in\C^\times .
\]

By \Cref{Eq: Global Gindikin-Karpelevich}, $n$ is finite for $s_0>0$.
We assume that $n>0$.

For $w\in W\bk{\bfM,\bfG}$, we denote
%\[
%ord_{s=s_0} M\bk{w,\chi,\lambda_s} = \sup\set{\operatorname{ord}_{s=s_0} M\bk{w,\chi,\lambda_s}f_s\bk{g} \mvert f_s\in I_{\bfP}\bk{\chi,s},\ g\in \bfG\bk{\A}}.
%\]
\[
ord_{s=s_0} M\bk{w,\chi_s} = \sup\set{\operatorname{ord}_{s=s_0} M\bk{w,\chi_s}f_s\bk{g} \mvert f_s\in I_{\bfP}\bk{\chi,s},\ g\in \bfG\bk{\A}}.
\]

For $0 < m\leq n$ let
\[
\Sigma^{\bfP}_{\bk{\chi,s_0,m}} = \set{w\in W\bk{\bfM,\bfG} \mvert \operatorname{ord}_{s=s_0} M\bk{w,\chi_s} \geq m} .
\]
%\[
%\Sigma^{\bfP}_{\bk{\chi,s_0,m}} = \set{w\in W\bk{\bfM,\bfG} \mvert \operatorname{ord}_{s=s_0} M\bk{w,\chi,\lambda_s} \geq m} .
%\]
We say that the pole of order $m$ cancels if
\begin{align*}
& \lim\limits_{s\to s_0} \bk{s-s_0}^m \suml_{w\in W\bk{\bfM,\bfG}} M\bk{w,\chi_s} \\
& = \lim\limits_{s\to s_0} \bk{s-s_0}^m \suml_{w\in \Sigma^{\bfP}_{\bk{\chi,s_0,m}}} M\bk{w,\chi_s} \equiv 0 .
\end{align*}
%\begin{align*}
%& \lim\limits_{s\to s_0} \bk{s-s_0}^m \suml_{w\in W\bk{\bfM,\bfG}} M\bk{w,\chi,\lambda_s} \res{I_{\bfP}\bk{\chi,s}} \\
%& = \lim\limits_{s\to s_0} \bk{s-s_0}^m \suml_{w\in \Sigma^{\bfP}_{\bk{\chi,s_0,m}}} M\bk{w,\chi,\lambda_s} \res{I_{\bfP}\bk{\chi,s}} \equiv 0 .
%\end{align*}

%Assume that after decomposing $\Sigma_{\bk{E,\chi,s_0,n}}$ into equivalence classes $\Sigma_i$ and that for all $i$
%\begin{equation}
%\label{Reason_for_cancellation_of_Poles}
%\lim\limits_{s\to s_0} \bk{s-s_0}^n \suml_{w\in \Sigma_i} M_w \equiv 0
%\end{equation}
%then the pole of order $n$ cancels.

After, maybe, cancellation of higher orders of a pole, we wish to determine its actual order.
Namely, for $0<m\leq n$ we say that $\Eisen_{\bfP}\bk{\chi,\lambda_s}_{\bfCT}$ admits a pole of order $m$ at $s_0$ if
\begin{align*}
& \lim\limits_{s\to s_0} \bk{s-s_0}^{m+1} \suml_{w\in W\bk{\bfM,\bfG}} M\bk{w,\chi_s} \\
& = \lim\limits_{s\to s_0} \bk{s-s_0}^{m+1} \suml_{w\in \Sigma^{\bfP}_{\bk{\chi,s_0,m+1}}} M\bk{w,\chi_s} \equiv 0
\end{align*}
%\begin{align*}
%& \lim\limits_{s\to s_0} \bk{s-s_0}^{m+1} \suml_{w\in W\bk{\bfM,\bfG}} M\bk{w,\chi,\lambda_s} \res{I_{\bfP}\bk{\chi,s}} \\
%& = \lim\limits_{s\to s_0} \bk{s-s_0}^{m+1} \suml_{w\in \Sigma^{\bfP}_{\bk{\chi,s_0,m+1}}} M\bk{w,\chi,\lambda_s} \res{I_{\bfP}\bk{\chi,s}} \equiv 0
%\end{align*}
and
\begin{align*}
& \lim\limits_{s\to s_0} \bk{s-s_0}^m \suml_{w\in W\bk{\bfM,\bfG}} M\bk{w,\chi_s} \\
& = \lim\limits_{s\to s_0} \bk{s-s_0}^m \suml_{w\in \Sigma^{\bfP}_{\bk{\chi,s_0,m}}} M\bk{w,\chi_s} \not\equiv 0.
\end{align*}
%\begin{align*}
%& \lim\limits_{s\to s_0} \bk{s-s_0}^m \suml_{w\in W\bk{\bfM,\bfG}} M\bk{w,\chi,\lambda_s} \res{I_{\bfP}\bk{\chi,s}} \\
%& = \lim\limits_{s\to s_0} \bk{s-s_0}^m \suml_{w\in \Sigma^{\bfP}_{\bk{\chi,s_0,m}}} M\bk{w,\chi,\lambda_s} \res{I_{\bfP}\bk{\chi,s}} \not\equiv 0.
%\end{align*}

In particular, for any holomorphic section $f_s\in I_{\bfP}\bk{\chi,s}$ and any $t\in T_E\bk{\A}$, it holds that
\[
\lim\limits_{s\to s_0} \bk{s-s_0}^m
\suml_{w\in W\bk{\bfM,\bfG}} M\bk{w,\chi_s}f_s\bk{t} \in \C 
\]
%\[
%\lim\limits_{s\to s_0} \bk{s-s_0}^m
% \suml_{w\in W\bk{\bfM,\bfG}} M\bk{w,\chi,\lambda_s}f_s\bk{t} \in \C 
%\]
and the limit is non-zero for some $f_s\in I_{\bfP}\bk{\chi,s}$ and $t\in T_E\bk{\A}$.

We define an equivalence relation on $\Sigma^{\bfP}_{\bk{\chi,s_0,m}}$ by:
\begin{equation}
\label{Equivalence relation on Sigma-m}
w\sim_{\bk{\chi,s_0}} w' \Longleftrightarrow w^{-1}\cdot\bk{\chi_{s_0}}=w'^{-1}\cdot\bk{\chi_{s_0}} .
\end{equation}
Clearly, cancellations of poles of intertwining operators can occur only within the same equivalence class.
It is thus useful to write the above sums as follows
\begin{align*}
& \lim\limits_{s\to s_0} \bk{s-s_0}^m \suml_{w\in W\bk{\bfM,\bfG}} M\bk{w,\chi_s}\\
& = \lim\limits_{s\to s_0} \bk{s-s_0}^m \suml_{\coset{w'}\in \Sigma^{\bfP}_{\bk{\chi,s_0,m}}\rmod\sim_{\bk{\chi,s_0}}} \coset{\suml_{w\in\coset{w'}} M\bk{w,\chi_s} } .
\end{align*}
In view of \Cref{Cor:Kernel_of_Series_is_kernel_of_CT}, in order to calculate the kernel of $\lim_{s=s_0} \coset{\bk{s-s_0}^m \Eisen_{\bfP}\bk{\chi,f,s,g}}$, it is enough to calculate the kernel of the above sum.
Since cancellations happen only within equivalency classes, it follows that
\[
\ker\bk{\lim_{s=s_0} \bk{s-s_0}^m \Eisen_{\bfP}\bk{\chi,f,s,g}} = \bigcap_{\coset{w'}\in \Sigma^{\bfP}_{\bk{\chi,s_0,m}}\rmod\sim_{\bk{\chi,s_0}}} \coset{\lim\limits_{s\to s_0} \bk{s-s_0}^m \suml_{w\in\coset{w'}} M\bk{w,\chi_s}} .
\]

%\todonum{Write the above sums as double sums (over equivalence classes and in the equivalence classes) and add more explanations about using this together with \Cref{Cor:Kernel_of_Series_is_kernel_of_CT} for our calculations}

%\vertline
%
%We now determine, for a point where $\Eisen_E\bk{\chi, f_s, s, g}$ admits a pole, whether the residual representation is square-integrable or not.
%Before doing so, we recall the following criterion from \cite[pg. 104]{MR0579181}.
%
%For $w\in W\bk{P_E,H_E}$ the element $\Real\bk{w^{-1}\cdot\chi_s}\in \mathfrak{a}_\R^\ast$ is known as the \emph{exponent of $I_{P_E}\bk{\chi,s}$ corresponding to $w$}.
%
%Assume $\Eisen_E\bk{\chi, f_s, s, g}$ admits a pole of order $n$ at $s_0$.
%We recall the equivalence relation defined in \Cref{Equivalence relation on Sigma-m} and define the quotient set
%\[
%\Sigma_{s_0} = \Sigma_{\bk{E,\chi,s_0,n}}\rmod \sim_{\bk{\chi,s_0}} .
%\]
%Note that the exponent is well defined for equivalence classes, namely $\Real\bk{w^{-1}\cdot\chi_s}=\Real\bk{w'^{-1}\cdot\chi_s}$ when $w\sim_{\bk{\chi,s_0}} w'$.
%And we consider the elements contributing to the residual representation at $s_0$, namely:
%\[
%\Sigma_{s_0}^0 = \set{\Omega \in \Sigma_{s_0} \mvert \lim\limits_{s\to s_0} \bk{s-s_0}^n \suml_{w\in \Omega} M\bk{w,\chi_s} \not\equiv 0 } .
%\]

\subsection{The Langlands Quotient Theorem}
We recall the Langlands classification, which will be used in \Cref{Sec_Local_Representations}. For a more detailed discussion, the reader may consult \cite[Chapter IV, Sec. XI.2]{MR1721403}.
% or \cite{MR2050093}.
We fix a place $\nu\in\Places$.

%For a Levi subgroup $\bfM$, a representation $\sigma_\nu$ of $\bfM\bk{F_\nu}$ is called tempered if it is unitarizable and all of its matrix coefficients lie in $L^{2+\epsilon}\bk{\bfM\bk{F_\nu}\rmod Z_{M,\nu}}$ for all $\epsilon>0$, where $Z_{M,\nu}$ denotes the center of $\bfM\bk{F_\nu}$.
%If $\nu\not\vert\infty$, this is equivalent to $\sigma_\nu$ being a direct summand of a parabolic induction from a discrete series representation.

\begin{Thm}[Langlands' Unique Irreducible Quotient Theorem]
	Let $\bfP=\bfM\cdot\bfU$ be a Levi subgroup of $\bfG$ and let $\sigma_\nu$ be a tempered representation of $\bfM\bk{F_\nu}$.
	Also let $\lambda\in\mathfrak{a}_{\bfM,\C}^\ast$ satisfy
	\begin{equation}
	\label{Eq:Standard_Module_Positivity_Condition}
	\gen{\lambda,\check{\alpha}}>0,\quad \forall \alpha\in\Delta_M.
	\end{equation}
	Then, $\Ind_{\bfM\bk{F_\nu}}^{\bfG\bk{F_\nu}}\bk{\sigma_\nu\otimes\lambda}$ admits a unique irreducible quotient which is the image of $M_\nu\bk{w,\sigma_\nu,\lambda}$, where $w$ is the representative in $W\bk{\bfM,\bfG}$ of the coset of the longest Weyl element $w_l\in W$.
\end{Thm}
In this theorem, $M_\nu\bk{w,\sigma_\nu,\lambda}$ is constructed in a similar way to \Cref{Eq:Local_Intertwining_Operator_Def}.
In particular, if $\sigma$ is a subrepresentation of some $\Ind_{\bfB\bk{F_\nu}\cap \bfM\bk{F_\nu}}^{\bfM\bk{F_\nu}}\chi$, then $M_\nu\bk{w,\sigma_\nu,\lambda}$ equals the restriction of $M_\nu\bk{w,\chi,\lambda}$.

In this case, we say that $\Ind_{\bfM\bk{F_\nu}}^{\bfG\bk{F_\nu}}\bk{\sigma_\nu\otimes\lambda}$ is a \emph{standard module}.
In fact, Langlands have shown that every irreducible admissible representation of $\bfG\bk{F_\nu}$ can be attained as a Langlands quotients of some standard module and that the inducing data is unique up to conjugation.

%In fact, the Langlands classification theorem states that any irreducible representation of $\bfG\bk{F_\nu}$ can be attained as the unique irreducible quotient of some standard module $\Ind_{\bfM\bk{F_\nu}}^{\bfG\bk{F_\nu}}\bk{\sigma_\nu\otimes\lambda}$ and that the triple $\bk{\bfP,\sigma_\nu,\lambda}$ is determined up to $W$-conjugation.

\begin{Remark}
	We note that there is an equivalent form of this theorem in terms of subrepresentations (see \cite{MR2490651}).
	Assuming that $\lambda$ satisfy
	\[
	\gen{\lambda,\check{\alpha}}<0,\quad \forall \alpha\in\Delta_M,
	\]
	the induction $\Ind_{\bfM\bk{F_\nu}}^{\bfG\bk{F_\nu}}\bk{\sigma_\nu\otimes\lambda}$ admits a unique irreducible subrepresentation.
	This subrepresentation is the kernel of $M_\nu\bk{w,\sigma,\lambda}$.
\end{Remark}

\subsection{Harish-Chandra's Commuting Algebra Theorem and the $R$-group}
\label{Subsec:R_groups}

We recall Harish-Chandra's commuting algebra theorem, the definition of the $R$-groups and a few properties of it.
For further information, the reader may consult \cite{MR582703} when $\nu\vert\infty$ and \cite{MR675406} when $\nu\not\vert\infty$.

%\todonum{Change $\bfP$ to $\bfB$ and $\bfM$ to $\bfT$}

Let $\nu$ be a place of $F$.
% and let $\bfP=\bfM\cdot\bfU$ be a parabolic subgroup of $\bfG$.

\begin{Thm}[Harish-Chandra's Commuting Algebra Theorem]
	For a unitary character $\mu:\bfT\bk{F_\nu}\to\C^\times$, the representation $I_{\bfB,\nu}\bk{\mu,\bar{0}}$ is semi-simple and, by Harish-Chandra's commuting algebra theorem, its endomorphism ring $End\bk{I_{\bfB,\nu}\bk{\mu,\bar{0}}}$ is spanned by the intertwining operators $N_\nu\bk{w,\mu,\bar{0}}$ such that $w\in \Stab_W\bk{\mu}$.
\end{Thm}

Let 
\[
\mathcal{K}_W\bk{\mu} = \set{w\in \Stab_W\bk{\mu}\mvert N_\nu\bk{w,\mu,\bar{0}}=c\cdot I,\, c\in\C}
\]
and let $\mathcal{R}_W\bk{\mu}$ denote the quotient group $\Stab_W\bk{\mu}\rmod \mathcal{K}_W\bk{\mu}$.
%\todonum{State the rest of this section as a theorem and double check facts vs. Gelbart-Knapp, Tadic, Knapp-Stein(?)}

The following theorem summarizes some of the properties of $\mathcal{R}_W\bk{\mu}$ and its connection to the structure of $I_{\bfB,\nu}\bk{\mu,\bar{0}}$.
For more details, consider \cite{MR582703} when $\nu\vert\infty$ or \cite{MR620252,MR1141803,MR517138,MR675406} when $\nu\not\vert\infty$.
Essentially, it is an analogue of the Artin-Wedderburn theorem.
\begin{Thm}
	The following hold:
	\begin{enumerate}		
		\item The exact sequence
		\[
		\set{1}\to \mathcal{K}_W\bk{\mu} \to \Stab_W\bk{\mu} \to \mathcal{R}_W\bk{\mu} \to \set{1}
		\]
		splits and $\Stab_W\bk{\mu}=\mathcal{R}_W\bk{\mu} \ltimes \mathcal{K}_W\bk{\mu}$.
		
		\item $End\bk{I_{\bfB,\nu}\bk{\mu,\bar{0}}}\cong \C\coset{\mathcal{R}_W\bk{\mu}}$.
		
		\item $\dim_\C\bk{End\bk{I_{\bfB,\nu}\bk{\mu,\bar{0}}}}=\FNorm{\mathcal{R}_W\bk{\mu}}$
		
		\item The number of inequivalent irreducible components in $I_{\bfB,\nu}\bk{\mu,\bar{0}}$ equals the number of conjugacy classes in $\mathcal{R}_W\bk{\mu}$.
		
		\item $I_{\bfP,\nu}\bk{\mu,\bar{0}}$ decomposes wit h multiplicities equal to 1 if and only if $\mathcal{R}_W\bk{\mu}$ is Abelian.
		
		\item If $\C\coset{\mathcal{R}_W\bk{\mu}} = M_{n_1}\bk{\C} \oplus\dots\oplus M_{n_2}\bk{\C}$, then $I_{\bfB,\nu}\bk{\mu,\bar{0}}$ admits $k$ inequivalent irreducible subrepresentations with multiplicities  $n_1$,..., $n_k$.
	\end{enumerate}
\end{Thm}

\section{The Degenerate Eisenstein Series on A Quasi-Split Form of $Spin_8$ Associated to the Heisenberg Parabolic Subgroup}
\label{Sec:Deg_Eis_Series_on_H_E}

This section is a reminder of the groups and series relevant to this paper.
For further details please consult \cite{MR2268487} or \cite[Section 2]{SegalEisen}.

\subsection{Quasi-Split Forms of $Spin_8$}
We recall the bijection
\[
\set{\begin{matrix}
	\text{Quasi-split forms} \\ \text{of $Spin_8$ over $F$}
	\end{matrix}}
\longleftrightarrow
\set{\varphi:\Gal\bk{\overline{F}\rmod F}\to S_3}
\longleftrightarrow
\set{\begin{matrix} \text{Isomorphism classes of} \\ \text{\'etale cubic algebras over $F$} \end{matrix}}.
\]
For any  \'etale cubic algebra $E$ over $F$ we denote by $H_E=Spin_8^E$ the corresponding simply-connected quasi-split form of $H=Spin_8$.
We denote by $\varphi_E:\Gal\bk{\overline{F}\rmod F}\to S_3$ the corresponding action of $\Gal\bk{\overline{F}\rmod F}$ on the Dynkin diagram of type $D_4$.
\begin{figure}[H]
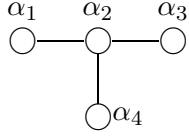

\label{fig:D4}
\begin{center}
\[
\xygraph{
	!{<0cm,0cm>;<0cm,1cm>:<1cm,0cm>::}
	!{(0,-1)}*{\bigcirc}="1"
	!{(0.4,-1)}*{\alpha_1}="label1"
	!{(0,0)}*{\bigcirc}="2"
	!{(0.4,0)}*{\alpha_2}="label2"
	!{(0,1)}*{\bigcirc}="3"
	!{(0.4,1)}*{\alpha_3}="label3"
	!{(-1,0)}*{\bigcirc}="4"
	!{(-1,0.4)}*{\alpha_4}="label4"
	"1"-"2" "2"-"3" "2"-"4"
}
\]
\end{center}
\caption{Dynkin diagram of type $D_4$}
\end{figure}
We note that $\varphi_E$ also fix a twisted form $S_E=\Aut_F\bk{E}$ of $S_3$.

Also recall that an \'etale cubic algebra over $F$ is one of the following:
\begin{enumerate}
\item $F\times F\times F$ (the \emph{split cubic algebra}).
\item $F\times K$, where $K$ is a quadratic field extension of $F$.
\item $E$, where $E$ is a cubic Galois field extension of $F$.
\item $E$, where $E$ is a cubic non-Galois field extension of $F$.
\end{enumerate}
We call the first three \emph{Galois \'etale cubic algebras}.

Given an \'etale cubic algebra $E$ over $F$ we fix a Chevalley-Steinberg system of \'epinglage \cite[Sections 4.1.3-4.1.4]{MR756316}
\[
\set{T_E, B_E, x_\gamma:\Ga\rightarrow\bk{H_E}_\gamma,\gamma\in \Phi_{D_4}} \ ,
\]
where $T_E\subset B_E$ is a maximal torus contained in a Borel subgroup (both defined over $F$) and $\Phi_{D_4}$ are the roots of $H_E\otimes \overline{F}\cong Spin_8\bk{\overline{F}}$.
For any $\gamma$ in the reduced root system of $H_E$ we denote by $F_\gamma$ the field of definition of $\gamma$.
Let $\Phi_E$ denote the root system of $H_E$ with respect to $T_E$.
%The reduced root system $\Phi_E$ of $H_E$ with respect to $T_E$ is given by the fixed point $\Phi_{D_4}^{S_E}$ of $\Phi_{D_4}$ under the action induced by $\varphi_E$.
Also, let $\Phi_E^{+}$ denote the positive roots of $H_E$ with respect to $B_E$ and let $\Delta_E$ denote the set of simple roots.
For a root $\gamma$, we denote by $F_\gamma$ the field of definition of $x_\gamma$.

\vspace{0.5cm}

Let $W=W_{H_E}$ denote the relative Weyl group of $H_E$ with respect to $T_E$ and denote by $w_l$ the longest element of $W$.
For elements of the Weyl group and of $\mathfrak{a}_{T_E,\C}^\ast$, we use the notations introduced at the end of Section 2.1 of \cite{SegalEisen}.
In particular, we denote by $\widetilde{\alpha_i}$ the simple roots of the absolute root system while $\alpha_i$ denote the roots of relative root system.
We denote by $\widetilde{w_i}$ the simple reflection in the Weyl group of the absolute root system associated to $\widetilde{\alpha_i}$ and by $w_i$ the simple reflection in the Weyl group of the relative root system associated to $\alpha_i$.
Similarly, we denote by $\widetilde{\omega_{\alpha_i}}$ the fundamental weights of the absolute root system of $H_E$ and by $\omega_{\alpha_i}$ the fundamental weights of the relative root system.
Also, we write $\widetilde{w_{i_1...i_k}}$ or $\widetilde{w\coset{i_1,...,i_k}}$ for $\widetilde{w_{i_1}}\cdots \widetilde{w_{i_k}}$ and write
$w_{i_1,...,i_k}$ or $w\coset{i_1...i_k}$ for $w_{i_1}\cdots w_{i_k}$.
%Similarly, we denote simple roots in the relative root system by $\alpha_i$ and simple roots in the absolute root system by $\widetilde{\alpha_i}$ and the corresponding fundamental weights by $\omega_{\alpha_i}$ and $\widetilde{\omega_{\alpha_i}}$, respectively.

%For example, for $E\rmod F$ and $\Gal\bk{E\rmod F}=\set{1,\sigma,\sigma^2}$ it holds that
For $E$ non-split we have:
\begin{itemize}
	\item For a cubic extension $E\rmod F$ it holds that
\begin{align*}
& w_1=\widetilde{w_{134}}=\widetilde{w_{1}}\widetilde{w_{3}}\widetilde{w_{4}},\quad
w_2=\widetilde{w_{2}}, \\
& \alpha_1=\widetilde{\alpha_1}+\widetilde{\alpha_3}+\widetilde{\alpha_4},\quad
\alpha_2=\widetilde{\alpha_2} \\
& \omega_{\alpha_1}=\widetilde{\omega_{\alpha_1}}+\widetilde{\omega_{\alpha_3}}+\widetilde{\omega_{\alpha_4}},\quad
\omega_{\alpha_2}=\widetilde{\omega_{\alpha_2}} .
%& \alpha_1=\widetilde{\alpha_1}+\widetilde{\alpha_3}^\sigma+\widetilde{\alpha_4}^{\sigma^2},\quad
%\alpha_2=\widetilde{\alpha_2} \\
%& \omega_{\alpha_1}=\widetilde{\omega_{\alpha_1}}+\widetilde{\omega_{\alpha_3}}^\sigma+\widetilde{\omega_{\alpha_4}}^{\sigma^2},\quad
%\omega_{\alpha_2}=\omega_{\widetilde{\alpha_2}} .
\end{align*}

\item In the case where $E=F\times K$, we always use the following convention
\begin{align*}
& w_1=\widetilde{w_{1}} ,\quad
w_2=\widetilde{w_{2}}, \quad
w_3=\widetilde{w_{34}}=\widetilde{w_{3}}\widetilde{w_{4}} \\
& \alpha_1=\widetilde{\alpha_1},\quad
\alpha_2=\widetilde{\alpha_2},\quad
\alpha_3=\widetilde{\alpha_3}+\widetilde{\alpha_4} \\
& \omega_{\alpha_1}=\widetilde{\omega_{\alpha_1}},\quad
\omega_{\alpha_2}=\widetilde{\omega_{\alpha_2}},\quad
\omega_{\alpha_3}=\widetilde{\omega_{\alpha_3}}+\widetilde{\omega_{\alpha_4}} .
\end{align*}

Note that here we make a choice of a distinct simple root in the absolute root system, $\widetilde{\alpha_1}$, in the construction of the relative root system.
Any other choice of simple "external" root in $\Delta$ would induce a different action of $S_E$ on the Dynkin diagram; this, in turn, corresponds to a different quasi-split form of $Spin_8$.
However, all such groups are isomorphic (as algebraic groups); this is phenomenon is called \emph{triality}.
\end{itemize}

Denote by $w_0\in W\bk{M_E,T_E}$ the representative of the class of the longest element $w_l$; it holds that $w_l=\widetilde{w_{134}}w_0$.
%\todonum{Write generators for the various cases}

\vspace{0.5cm}

We denote by $P_E=M_E\cdot U_E$ the Heisenberg parabolic subgroup of $H_E$ which is the maximal parabolic subgroup given by $P_\Psi$, where $\Psi=\set{\alpha_1, \alpha_3, \alpha_4}^{S_E}$.
The Levi subgroup $M_E$ of $P_E$ is isomorphic to 
\[
\bk{\Res_{E\rmod F}GL_2}^0 = \set{g\in \Res_{E\rmod F}GL_2\mvert \operatorname{det}\bk{g}\in \Gm} .
\]
Associated to $M_E$ is a determinant map, ${det}_{M_E}:M_E\to\Gm$.
The unipotent radical $U_E$ of $P_E$ is a 9-dimensional Heisenberg group over $F$.

%We fix the following set of representatives for $W_{H_E}\lmod W$.
%\todonum{Complete this}

%\todonum{Explain about the notations for Weyl elements of relative and absolute root systems}

\subsection{The Degenerate Eisenstein Series Associated to the Heisenberg Parabolic Subgroup}
Given $s\in\C$ and a finite order character $\chi$ of $F^\times\lmod \A^\times$, we form the normalized parabolic induction
\begin{equation}
I\bk{\chi,s} = \Ind_{P_E}^{H_E} \bk{\chi\circ{det}_{M_E}}\otimes \FNorm{{det}_{M_E}}^s .
\end{equation}

Let
\begin{equation}
\begin{split}
& \lambda_s = \bk{-1,s+\frac{3}{2},-1,-1} \\
& \eta_s = \bk{1,s-\frac{3}{2},1,1} \\
& \mu_\chi = \chi\circ {det}_{M_E} \\
& \chi_s = \mu_\chi\otimes \lambda_s ,
\end{split}
\end{equation}
where $\lambda_s, \eta_s\in \mathfrak{a}_{T_E,\C}^\ast$ are written in the coordinates coming from the absolute root system.
As in \Cref{Eq:Induction_from_Parabolic_in_induction_from_Borel}, it holds that
\[
\begin{split}
& \Ind_{B_E}^{H_E}\mu_\chi\otimes\eta_s \twoheadrightarrow I\bk{\chi,s} \\
& I\bk{\chi,s} \hookrightarrow \Ind_{B_E}^{H_E}\mu_\chi\otimes\lambda_s = I_{B_E}\bk{\chi,s} .
\end{split}
\]
In particular, note that, $I\bk{\chi,s}$ is the image (the leading term in the Laurent series) of $N_\nu\bk{\widetilde{w_{134}}, \mu_\chi, \eta_s}$ and the kernel of $N_\nu\bk{\widetilde{w_{134}}, \mu_\chi, \lambda_s}$.

We also note that we also write $N\bk{w,\chi_s}$ for $N\bk{w,\chi,\lambda_s}$.

For a holomorphic section $f_s\in I\bk{\chi,s}$ we define the degenerate Eisenstein series
\begin{equation}
\Eisen_E\bk{\chi,f,s,g} = \Eisen_{P_E}\bk{\mu_\chi,f,\lambda_s,g} = \suml_{\gamma\in P_E\bk{F}\lmod H_E\bk{F}} f_s\bk{\gamma g}.
\end{equation}

To any Galois \'etale cubic algebra $E$, we associate a finite order Hecke character $\chi_E$ as follows:
\begin{itemize}
\item If $E$ is a Galois field extension of $F$, let $\chi_E$ denote the cubic character associated to it by global class field theory.
\item If $E=F\times K$, where $K$ is a field, let $\chi_E=\chi_K$ be the quadratic character assocaited to $K\rmod F$ by global class field theory.
\item IF $E=F\times F\times F$ let $\chi=\Id$.
\end{itemize}

We recall the order of the poles of $\Eisen\bk{\chi,f,s,g}$ in the half-plane $\Real\bk{s}>0$ given in \cite[Theorem 4.1]{SegalEisen}.

\begin{Thm}
The degenerate Eisenstein series $\Eisen\bk{\chi,f,s,g}$ is holomorphic at $s=s_0$ except for the following poles:
\begin{itemize}
\item $s_0=\frac{1}{2}$ and $\chi^2=\Id$, a simple pole.
\item $s_0=\frac{3}{2}$ and $\chi=\chi_E$ a simple pole for $E$ non-split and a double pole when $E=F\times F\times F$.
\item $s_0=\frac{3}{2}$ and $\chi=\Id$ when $E=F\times K$ and $K$ is a field, a simple pole.
\item $s_0=\frac{5}{2}$ and $\chi=\Id$, a simple pole.
\end{itemize}
Further more, the residual representations at these points are square-integrable except for the following cases:
\begin{itemize}
\item $s_0=\frac{1}{2}$, $E=F\times K$ and $\chi\circ\Nm_{K\rmod F}=\Id$ (Here $K$ is either a field or split).
\item $s_0=\frac{3}{2}$, $\chi=\Id$ and $E=F\times K$ is non-split.
\end{itemize}
\end{Thm}
%\todonum{Mention which are square-integrable}

%In this paper we study the residues of the residual representation of $\Eisen\bk{\chi,f,s,g}$ in the various poles.
Let $E$ be an \'etale cubic algebra over $F$, $\chi:F^\times\lmod\A^\times\to\C^\times$ a quadratic character and $s_0\in\set{\frac{1}{2},\frac{3}{2},\frac{5}{2}}$ such that $\Eisen_E\bk{\chi,f_s,s,g}$ admits a pole of order $n$ at $s_0$.
We define the residue of $\Eisen_E\bk{\chi,f_s,s,g}$ at $s_0$ to be
\[
Res\bk{s_0,\chi,E} = Span_\C \set{\lim\limits_{s\to s_0} \bk{s-s_0}^n \Eisen_E\bk{\chi,f_s,s,\cdot}\mvert f_s\in I_{P_E}\bk{\chi,s}} .
\]
In this paper we study $Res\bk{s_0,\chi,E}$ for the various triples $\bk{s_0,\chi,E}$.
The results are summarized in \Cref{Thm:Main_Sq_Int} and \Cref{Thm:Main_non_Sq_Int}.

%\todonum{Do I want to state the results here?}
%The proof of this theorem will occupy sections ...

%\section{Behavior of Local Intertwining Operators and Structure of Local Degenerate Principal Series}
\section{Local Degenerate Principal Series}
\label{Sec_Local_Representations}
%{\color{red} \noindent\rule{\textwidth}{3pt}}

Write $I\bk{\chi,s} = \otimes' I_\nu\bk{\chi_\nu,s}$, where
\[
I_\nu\bk{\chi_\nu,s} = \Ind_{P_E\bk{F_\nu}}^{H_E\bk{F_\nu}} \bk{\chi\circ{det}_{M_E}}\otimes \FNorm{{det}_{M_E}}_{F_\nu}^s .
\]
In this section, we discuss the structure of $I_\nu\bk{\chi_\nu,s}$ in cases where $\Eisen\bk{\chi,f,s,g}$ admits a pole. We also discuss the behaviour of these representations under various intertwining operators of importance to the calculations performed in \Cref{Sec:SquareIntegrableResidues} and \Cref{Sec:NonSquareIntegrableResidues}.
%In this section we compute the image of $I_\nu\bk{\chi_\nu,s}$ under various intertwining operators $N_\nu\bk{w,\mu_{\chi,\nu},\lambda_{s,\nu}}$ (or their residues) at the points $s=\frac{1}{2}$, $s=\frac{3}{2}$ and $s=\frac{5}{2}$ with $\chi$ a quadratic character.
We record the structure of those local representations in the following theorem.
%\todonum{Collect the results of this section as a theorem, basically use the list in subsection 5.4}
\begin{Thm}
	\label{Thm:local_results}
	\begin{enumerate}
		\item For any place $\nu$, the representation $I_\nu\bk{\Id_\nu,\frac{5}{2}}$ admits a unique irreducible quotient which is trivial.
%		 and is the image of $N_\nu\bk{w_0,\Id,\lambda_{\frac{5}{2}}}$.
		
		\item For any place $\nu$ such that $E_\nu$ is a Galois \'etale cubic algebra, the representation $I_\nu\bk{\chi_{E,\nu},\frac{3}{2}}$ admits a unique irreducible quotient which is the minimal representation of $H_E\bk{F_\nu}$.
		
		\item For any place $\nu$ such that $E_\nu=F_\nu\times K_\nu$ and $K_\nu$ is a field, the representation $I_\nu\bk{\Id_\nu,\frac{3}{2}}$ admits a unique irreducible quotient which is spherical.
		
%		\item For any place $\nu$ such that $E_\nu$ is not a field or $\R\times \C$, the representation $I_\nu\bk{\Id_\nu,\frac{1}{2}}$ admits a unique irreducible quotient which is spherical.
		\item For any place $\nu$ such that $E_\nu$ is not a field, the representation $I_\nu\bk{\Id_\nu,\frac{1}{2}}$ admits a unique irreducible quotient which is spherical.

		\item For any place $\nu$ such that $E_\nu$ is a field, $I_\nu\bk{\Id_\nu,\frac{1}{2}}$ admits a maximal semi-simple quotient of the form $\pi_1\oplus\pi_{-2}$, where $\pi_1$ is spherical and both irreducible quotients are eigenspaces of certain intertwining operators (of eigenvalues $1$ and $-2$).
		
		\item For any place $\nu$ such that $E_\nu$ is not a field and $\chi_\nu\circ\Nm_{E_\nu\rmod F_\nu} = \Id$, $I_\nu\bk{\chi_\nu,\frac{1}{2}}$ admits a unique irreducible quotient which is spherical.
		
		\item For any place $\nu$ such that $\chi_\nu$ is quadratic and $\chi_\nu\circ\Nm_{E_\nu\rmod F_\nu}\neq \Id$, $I_\nu\bk{\chi_\nu,\frac{1}{2}}$ admits a maximal semi-simple quotient of length
		\begin{itemize}
			\item 1 if $E_\nu$ is a field,
			\item 2 if $E_\nu=F_\nu\times K_\nu$ and $K_\nu$ is a field or
			\item 4 if $E_\nu=F_\nu\times F_\nu\times F_\nu$.
		\end{itemize}
		
	\end{enumerate}
\end{Thm}

The proof of this theorem will occupy the rest of this section, while some calculations performed at Archimedean places are performed in \Cref{Appendix:Archimedean}.

\begin{Remark}
	In all cases, we shall identify the irreducible quotients of the various $I_\nu\bk{\chi_\nu,s_0}$ as eigenspaces of certain normalized intertwining operators.
	We will denote the irreducible quotients of $I_\nu\bk{\chi_\nu,s_0}$ by $\pi_{\epsilon,\nu}^{\bk{\chi_\nu,E_\nu,s_0}}$, where $\epsilon$ is this eigenvalue, or tuple of eigenvalues.
	Usually, when there is no source of confusion, we will denote $\pi_{\epsilon,\nu}^{\bk{\chi_\nu,E_\nu,s_0}}$ by $\pi_{\epsilon,\nu}$.
\end{Remark}

%\begin{Remark}
%	From the discussion in this proof it is clear that for a place $\nu$ such that $E_\nu=\R\times\C$, $I_\nu\bk{\Id_\nu,\frac{1}{2}}$ admits a maximal semi-simple quotient of length one or two which contains a spherical representation.
%	However, it is unclear whether it is irreducible or a direct sum of two irreducible representations.
%\end{Remark}

\begin{proof}
We consider the various possible $s$, $\chi$ and $E$.
\subsection{$s=\frac{5}{2}$ and $\chi_\nu=\Id_\nu$}
In this case, $I_\nu\bk{\Id,\frac{5}{2}}$ is a standard module with Langlands operator $N_\nu\bk{w_0,\Id,\lambda_{\frac{5}{2}}}$. Since $I_\nu\bk{\Id,-\frac{5}{2}}\cong {I_\nu\bk{\Id,\frac{5}{2}}}^\ast$ admits $\Id_{H_E}$ as a subrepresentation, the unique irreducible quotient of $I_\nu\bk{\Id,\frac{5}{2}}$ is isomorphic to the $\Id_{H_E}$.

\subsection{$s=\frac{3}{2}$ and $\chi_\nu=\chi_{E,\nu}$}
We recall, from \cite[Proposition 4.3]{MR1918673}, that, for any $\nu\in\Places$, the unique irreducible quotient of $I_\nu\bk{\chi_{E,\nu},\frac{3}{2}}$ is the minimal representation $\Pi_{E,\nu}$ of $H_E\bk{F_\nu}$.

\subsection{$s=\frac{3}{2}$ and $\chi_\nu=\Id_\nu$}
We assume that $E_\nu=F_\nu\times K_\nu$, where $K_\nu$ is a quadratic field extension of $F_\nu$.
%It follows from \Cref{Appendix:Archimedean} (for $\nu\vert\infty$) and \cite{Gan-Savin-D4} (for $\nu\not\vert\infty$) that $I_\nu\bk{\Id,\frac{3}{2}}$ admits a unique irreducible quotient (which is spherical).
It follows from \cite[Theorem 3.3(1)]{RallisSchiffmannPaper} that $I_\nu\bk{\Id,\frac{3}{2}}$ admits a unique irreducible quotient (which is spherical).

\subsection{$s=\frac{1}{2}$ and $\chi_\nu=\Id_\nu$}
Fix a place $\nu\in\Places$.
We start by recalling from \cite{SegalSingularities} that
\[
\Ind_{B_E{\bk{F_\nu}}}^{H_E{\bk{F_\nu}}} \lambda_{\bk{-1,1,-1,-1}} =
\Pi_{-2,\nu}\oplus \Pi_{1,\nu},
\]
where:
\begin{itemize}
	\item $\Pi_{1,\nu}=\Ind_{P_{\set{\alpha_2}}\bk{F_\nu}}^{H_E\bk{F_\nu}} \lambda_0$, where $\lambda_0\in \mathfrak{a}_{M_{\set{\alpha_2}},\C}^\ast$ satisfy that,
	$\iota_M\bk{\lambda_0}=\lambda_{\bk{-\frac{1}{2},0,-\frac{1}{2}-\frac{1}{2}}}$, in terms of \Cref{Eq:Inculsions_of_character_groups}.
%	 upon restriction to the torus, $\lambda_0\res{T_E\bk{F_\nu}}=\lambda_{\bk{-\frac{1}{2},0,-\frac{1}{2}-\frac{1}{2}}}$.
%	\todonum{Describe $\lambda_0$ using \Cref{Eq:Inculsions_of_character_groups} instead.}
	\item $\Pi_{-2,\nu}=\Ind_{P_{\set{\alpha_2}}\bk{F_\nu}}^{H_E\bk{F_\nu}} St_{M_{\set{\alpha_2}}}\otimes\lambda_0$, where $St_{M_{\set{\alpha_2}}}$ is the Steinberg representation of $M_{\set{\alpha_2}}$.
	\item Each of the representations $\Pi_\epsilon$ is the $\epsilon$-eigenspace of $\lim_{s\to\frac{1}{2}} N_\nu\bk{w_{21342},\Id,\lambda_{\bk{-1,1,-1,-1}}}$ and each of them admits a unique irreducible subrepresentation.
%	The irreducible subrepresentation of $\Pi_{1,\nu}$ is spherical.
\end{itemize}

It follows that
\[
\Ind_{B_E{\bk{F_\nu}}}^{H_E{\bk{F_\nu}}} \lambda_{\bk{1,-1,1,1}}
= \coset{\Ind_{B_E{\bk{F_\nu}}}^{H_E{\bk{F_\nu}}} \lambda_{\bk{-1,1,-1,-1}}}^\ast = \Pi_{-2,\nu}^\ast\oplus \Pi_{1,\nu}^\ast
\]
and that both $\Pi_{1,\nu}^\ast$ and $\Pi_{-2,\nu}^\ast$ admit unqiue irreducible quotients, denoted by $\pi_1$ and $\pi_{-2}$ respectively.
Note that $\pi_1$ is a spherical representation.

We also note that:
\begin{itemize}
	\item $\lambda_{\bk{-1,1,-1,-1}} = \lambda_{-\frac{1}{2}}=\widetilde{w_{2134}}^{-1}\cdot\lambda_{\frac{1}{2}} = \widetilde{w_{213421342}}^{-1}\cdot\lambda_{\frac{1}{2}}$ .
	\item $\lambda_{\bk{1,-1,1,1}}=\eta_{\frac{1}{2}}$.
\end{itemize}

Hence, the maximal semi-simple quotient of $I_\nu\bk{\Id,\frac{1}{2}}$ is a quotient of$I_{B_E,\nu}\bk{\Id,\eta_{\frac{1}{2}}}$.
Since $I_\nu\bk{\Id,\frac{1}{2}}$ is spherical, its maximal semi-simple quotient is either $\pi_{1,\nu}$ or $\pi_{1,\nu}\oplus\pi_{-2,\nu}$.

From the calculations in \cite{SegalSingularities}, it also follows that the maximal semi-simple quotient of $I_\nu\bk{\Id,\frac{1}{2}}$ is given by $N_\nu\bk{\widetilde{w_{2134}},\Id,\lambda_{\frac{1}{2}}}\bk{I_\nu\bk{\Id,\frac{1}{2}}}$.

Finally, it holds that:
\begin{Lem}
	The maximal semi-simple quotient of $I_\nu\bk{\Id,\frac{1}{2}}$ is given as follows:
	\begin{itemize}
		%	\item If $E_\nu$ is not a field or $\R\times\C$ then $\pi_{1,\nu}$ is the unique irreducible quotient of $I_\nu\bk{\Id,\frac{1}{2}}$.
		\item If $E_\nu$ is not a field then $\pi_{1,\nu}$ is the unique irreducible quotient of $I_\nu\bk{\Id,\frac{1}{2}}$.
		\item If $E_\nu$ is a field then the maximal semi-simple quotient of $I_\nu\bk{\Id,\frac{1}{2}}$ is $\pi_{1,\nu}\oplus \pi_{-2,\nu}$.
		%	\todonum{What about case $E_\nu=\R\times\C$?}
	\end{itemize}
\end{Lem}

\begin{proof}
	For $\nu\vert\infty$ it follows from \Cref{Appendix:Archimedean}.
	For $\nu\not\vert\infty$ it follows from \cite{Gan-Savin-D4}; however, we supply a different proof.
	
	We show that $I_\nu\bk{\Id,-\frac{1}{2}}$ has a semi-simple subrepresentation of length $2$ if $E_\nu$ is a field and a unique irreducible subrepresentation if $E_\nu=F_\nu\times K_\nu$.
	
	\begin{itemize}
		\item Assume that $E_\nu$ is a field.
		In this case, a straight forward computation of Jacquet modules along $B_E$ yields that the multiplicity of $\bk{-1,1,-1,-1}$:
		\begin{itemize}
			\item The multiplicity of $\bk{-1,1,-1,-1}$ in $\mathcal{J}_{T_E}^{H_E} \bk{\Ind_{B_E}^{H_E}\chi_{-\frac{1}{2}}}$ is $2$.
			
			\item The multiplicity of $\bk{-1,1,-1,-1}$ in $\mathcal{J}_{T_E}^{H_E} \bk{I_\nu\bk{\Id,-\frac{1}{2}}}$ is $2$.
			
			\item The multiplicity of $\bk{-1,1,-1,-1}$ in $\mathcal{J}_{T_E}^{H_E} \bk{\pi_{1,\nu}}$ and $\mathcal{J}_{T_E}^{H_E} \bk{\pi_{2,\nu}}$ is $1$.
		\end{itemize}
		Since both $\pi_{1,\nu}$ and $\pi_{2,\nu}$ are constituents of $\Ind_{B_E}^{H_E}\chi_{-\frac{1}{2}}$, the claim follows.
		
		\item Assume that $E_\nu=F_\nu\times K_\nu$.
		In particular, we assume that $\alpha_1$ is defined over $F_\nu$.
		Let $\Omega$ denote the $1$-dimensional representation of $M_{\set{\alpha_1}}$ given by the Jacquet functor $\mathcal{J}_{M_{\set{\alpha_1}}}^{H_E} \bk{\FNorm{{det}_{M_E}}^{-\frac{1}{2}}}$.
		Note that, by taking Jacquet module in stages,  $\mathcal{J}_{T_E}^{M_{\set{\alpha_1}}} \bk{\Omega} = \bk{-1,1,-1,-1}$.
		
		The Levi subgroup $M_{\alpha_1,\alpha_2}$ is isomorphic to $GL_3\times \Res_{K\rmod F}\bk{GL_1}$ and hence the (normalized) induction $\Ind_{M_{\set{\alpha_1}}}^{M_{\set{\alpha_1,\alpha_2}}}\Omega$ is irreducible.
		It follows that
		\begin{align*}
		I_\nu\bk{\Id,-\frac{1}{2}}
		& \hookrightarrow \Ind_{P_E}^{G_E} \bk{\Ind_{P_{\set{\alpha_1}}}^{P_E}\Omega} \\
		& \cong \Ind_{P_{\set{\alpha_1}}}^{H_E}\Omega \\
		& \cong \Ind_{P_{\set{\alpha_1,\alpha_2}}}^{H_E} \bk{\Ind_{M_{\set{\alpha_1}}}^{M_{\set{\alpha_1,\alpha_2}}}\Omega} \\
		& \hookrightarrow \Ind_{P_{\set{\alpha_1,\alpha_2}}}^{H_E} \bk{\Ind_{T_E}^{M_{\set{\alpha_1,\alpha_2}}} \lambda_{\bk{0,-1,0,0}}} \\
		& \cong \Ind_{T_E}^{H_E} \lambda_{\bk{0,-1,0,0}}.
		\end{align*}
		Since $\Ind_{T_E}^{H_E} \lambda_{\bk{0,-1,0,0}}$ admits $\pi_{1,\nu}$ as its unique irreducible subrepresentation, the claim follows.

	\end{itemize}
	
\end{proof}

\subsection{$s=\frac{1}{2}$ and $\chi_\nu^2=\Id_\nu\neq\chi_\nu$}
We first note that
\[
I_{B_E,\nu}\bk{\mu_{\chi_\nu},\eta_{\frac{1}{2}}} \cong 
I_{B_E,\nu}\bk{\chi_\nu\circ\bk{\widetilde{\omega_{\alpha_1}}+\widetilde{\omega_{\alpha_2}}+\widetilde{\omega_{\alpha_3}}+\widetilde{\omega_{\alpha_4}}},\lambda_{\bk{0,1,0,0}}},
\]
where the isomorphism is given by $N_\nu\bk{w_2,\mu_{\chi_\nu},\eta_{\frac{1}{2}}}$.
%\todonum{Fix this}

By induction in stages, we have
\begin{align*}
& I_{B_E,\nu}\bk{\chi_\nu\circ\bk{\widetilde{\omega_{\alpha_1}}+\widetilde{\omega_{\alpha_2}}+\widetilde{\omega_{\alpha_3}}+\widetilde{\omega_{\alpha_4}}},\lambda_{\bk{0,1,0,0}}} = \\
& \Ind_{P_E\bk{F_\nu}}^{H_E\bk{F_\nu}} \coset{\bk{\Ind_{B_E\bk{F_\nu}\cap M_E\bk{F_\nu}}^{M_E\bk{F_\nu}} \chi_\nu\circ\bk{\widetilde{\omega_{\alpha_1}}+\widetilde{\omega_{\alpha_2}}+\widetilde{\omega_{\alpha_3}}+\widetilde{\omega_{\alpha_4}}}}  \otimes \omega_{\alpha_2} } .
\end{align*}

We first consider the inner induction.
Write $\widetilde{\chi}=\chi_\nu\circ\bk{\widetilde{\omega_{\alpha_1}}+\widetilde{\omega_{\alpha_2}}+\widetilde{\omega_{\alpha_3}}+\widetilde{\omega_{\alpha_4}}}$.

As $\widetilde{\chi}$ is unitary, $\Ind_{B_E\bk{F_\nu}\cap M_E\bk{F_\nu}}^{M_E\bk{F_\nu}}\widetilde{\chi}$ is semi-simple and its structure could be understood via its $R$-group (see \Cref{Subsec:R_groups}).

We first note that $Stab_W\bk{\widetilde{\chi}}\subset W_{M_E}$.
Also, since any element in $W_{M_E}$ is an involution, for any $w\in Stab_W\bk{\widetilde{\chi}}$, the space
$\Ind_{B_E\bk{F_\nu}\cap M_E\bk{F_\nu}}^{M_E\bk{F_\nu}}\bk{\widetilde{\chi}}$ decompose into a direct sum of eigenspaces of $N\bk{w,\widetilde{\chi},\bar{0}}$ with eigenvalues $1$ and $-1$.
While these eigenspaces are subrepresentations, they are not necessarily irreducible.

We calculate $Stab_W\bk{\widetilde{\chi}}\subset W_{M_E}$ and $\mathcal{R}_{W_{M_E}}\bk{\widetilde{\chi}}$ for the various $E_\nu$ and $\chi_\nu$.
\begin{itemize}
	\item If $E_\nu$ is a field, then $\Stab_{W_{M_E}}\bk{\widetilde{\chi}}=\set{1}$.
	It follows that $\mathcal{R}_{W_{M_E}}\bk{\widetilde{\chi}}=\set{1}$.
	In particular, $\Ind_{B_E\bk{F_\nu}\cap M_E\bk{F_\nu}}^{M_E\bk{F_\nu}}\bk{\widetilde{\chi}}$ is irreducible.
	We also note that $N_\nu\bk{w_1,\widetilde{\chi},\bar{0}}$ acts as an isomorphism, but is not an endomorphism.
%	\item If $E_\nu$ is a field, $\Stab_{W_{M_E}}\bk{\widetilde{\chi}}=\set{1,w_1}$.
%	Since $N_\nu\bk{w_1,\widetilde{\chi},\bar{0}}=\Id$, it follows that $\mathcal{K}_{W_{M_E}}\bk{\widetilde{\chi}}=\set{1,w_1}$ and $\mathcal{R}_{W_{M_E}}\bk{\widetilde{\chi}}=\set{1}$.
	
	\item If $E_\nu=F_\nu\times K_\nu$, where $K_\nu$ is a field, then $\Stab_{W_{M_E}}\bk{\widetilde{\chi}}=\set{1,w_3}$.
	
	\begin{itemize}
		\renewcommand{\labelitemii}{$\ast$}
		
		\item If $\chi_\nu\circ\Nm_{E_\nu\rmod F_\nu}=\Id$, then $N_\nu\bk{w_{3},\widetilde{\chi},\bar{0}}=\Id$.
		It follows that $\mathcal{R}_{W_{M_E}}\bk{\widetilde{\chi}}=\set{1}$.
	
		\item If $\chi_\nu\circ\Nm_{E_\nu\rmod F_\nu}\neq\Id$, then $N_\nu\bk{w_{3},\widetilde{\chi},\bar{0}}$ is not a constant multiple.
		In fact, in this case, we have a decomposition $\Ind_{B_E\bk{F_\nu}\cap M_E\bk{F_\nu}}^{M_E\bk{F_\nu}}\bk{\widetilde{\chi}}=\sigma^{(-1)}_\nu\oplus\sigma^{(1)}_\nu$, where the irreducible subrepresentation $\sigma^{(\epsilon)}$ is the $\epsilon$-eigenspace of $N_\nu\bk{w_{3},\widetilde{\chi},\bar{0}}$.
	It follows that $\mathcal{R}_{W_{M_E}}\bk{\widetilde{\chi}}=\set{1,w_3}$.
	\end{itemize}
	
	We note that $N_\nu\bk{w_1,\widetilde{\chi},\bar{0}}$ acts as an isomorphism, but is not an endomorphism.

%	\item If $E_\nu=F_\nu\times K_\nu$, where $K_\nu$ is a field, $\Stab_{W_{M_E}}\bk{\widetilde{\chi}}=\set{1,w_1,w_3,w_{13}}$.
%	Since $N_\nu\bk{w_{13},\widetilde{\chi},\bar{0}}=\Id$, it follows that $\mathcal{K}_{W_{M_E}}\bk{\widetilde{\chi}}=\set{1,w_{13}}$ and $\mathcal{R}_{W_{M_E}}\bk{\widetilde{\chi}}=\set{1,\overline{w_1}}=\set{1,\overline{w_3}}$.

	\item If $E_\nu=F_\nu\times F_\nu\times F_\nu$, $\Stab_{W_{M_E}}\bk{\widetilde{\chi}}=\gen{w_{13},w_{14},w_{34}}$.
	
	We note that each of the intertwining operators $N_\nu\bk{w_{13},\widetilde{\chi},\bar{0}}$, $N_\nu\bk{w_{14},\widetilde{\chi},\bar{0}}$ and $N_\nu\bk{w_{34},\widetilde{\chi},\bar{0}}$ is not a constant multiple; namely, its $-1$-eigenspace is non-zero.
	
%	Indeed, consider, for example $N_\nu\bk{w_{13},\widetilde{\chi},\bar{0}}$.
	Indeed, by \Cref{Lemma:Intertwining_operator_of_simple_reflections}, any one of these intertwining operators factors through a subgroup $SL_2\bk{F_\nu}\times SL_2\bk{F_\nu}\times GL_2\bk{F_\nu}$ of $M_E\bk{F_\nu}$ associated to it.
	The decomposition in this case is a well known fact.

	It follows that $\mathcal{R}_{W_{M_E}}\bk{\widetilde{\chi}}=\set{1,w_{13},w_{14},w_{34}}$.
	
%	From the discussion above it follows that $\dim_\C\bk{End\bk{I_{\bfP,\nu}\bk{\mu,\bar{0}}}}=4$.
%	In this case, we either have $4$ inequivalent components in $End\bk{I_{\bfP,\nu}\bk{\mu,\bar{0}}}$ or $2$ equivalent components.
	Since $\mathcal{R}_{W_{M_E}}\bk{\widetilde{\chi}}$ has $4$ conjugacy classes, it follows that $I_{\bfP,\nu}\bk{\mu,\bar{0}}$ admits $4$ inequivalent components.
	We write
	$\Ind_{B_E\bk{F_\nu}\cap M_E\bk{F_\nu}}^{M_E\bk{F_\nu}}\bk{\widetilde{\chi}}=\sigma^{(-1,-1)}_\nu \oplus \sigma^{(-1,1)}_\nu \oplus \sigma^{(1,-1)}_\nu \oplus \sigma^{(1,1)}_\nu$,
	where:
	\begin{itemize}
		\item $\sigma^{(\epsilon,\delta)}_\nu$ is a $\epsilon$-eigenspace for $N_\nu\bk{w_{13},\widetilde{\chi},\bar{0}}$.
		
		\item $\sigma^{(\epsilon,\delta)}_\nu$ is a $\delta$-eigenspace for $N_\nu\bk{w_{14},\widetilde{\chi},\bar{0}}$.
		
		\item $\sigma^{(\epsilon,\delta)}_\nu$ is a $\bk{\epsilon\cdot\delta}$-eigenspace for $N_\nu\bk{w_{34},\widetilde{\chi},\bar{0}}$.
	\end{itemize}
	We note that all of $N_\nu\bk{w_{1},\widetilde{\chi},\bar{0}}$, $N_\nu\bk{w_{3},\widetilde{\chi},\bar{0}}$ and $N_\nu\bk{w_{4},\widetilde{\chi},\bar{0}}$ are isomorphisms, but are not endomorphisms.
	
	Another way to view this decomposition, in the spirit of \cite{MR620252}, is by restriction to $SL_2\bk{F_\nu}\times SL_2\bk{F_\nu}\times SL_2\bk{F_\nu}\subset M_{E}\bk{F_\nu}$.
	Write $\Ind_{\mathcal{B}\bk{F_\nu}}^{SL_2\bk{F_\nu}}\chi=\sigma_1\oplus\sigma_{-1}$.
	Then we have
	\[
	\Ind_{\coset{\mathcal{B}\times \mathcal{B}\times \mathcal{B}}\bk{F_\nu}}^{\coset{SL_2\times SL_2\times SL_2}\bk{F_\nu}} \bk{\chi_\nu\otimes \chi_\nu\otimes \chi_\nu} = \bigoplus_{\epsilon_1,\epsilon_3,\epsilon_4\in\set{1,-1}} \sigma_{\bk{\epsilon_1,\epsilon_3,\epsilon_4}},
	\]
	where
	\[
	\sigma_{\bk{\epsilon_1,\epsilon_3,\epsilon_4}} = \sigma_{\epsilon_1} \boxtimes \sigma_{\epsilon_3} \boxtimes \sigma_{\epsilon_4} .
	\]
	By applying $h_{\alpha_2}\bk{\unif}\in M_E$, we get
	\begin{align*}
	& \sigma^{\bk{1,1}}_\nu = \sigma_{\bk{1,1,1}} \oplus \sigma_{\bk{-1,-1,-1}} \\
	& \sigma^{\bk{1,-1}}_\nu = \sigma_{\bk{1,1,-1}} \oplus \sigma_{\bk{-1,-1,1}} \\
	& \sigma^{\bk{-1,1}}_\nu = \sigma_{\bk{1,-1,1}} \oplus \sigma_{\bk{-1,1,-1}} \\
	& \sigma^{\bk{-1,-1}}_\nu = \sigma_{\bk{-1,1,1}} \oplus \sigma_{\bk{1,-1,-1}} .
	\end{align*}
	
%	\item If $E_\nu=F_\nu\times F_\nu\times F_\nu$, $\Stab_{W_{M_E}}\bk{\widetilde{\chi}}=\gen{w_1,w_3,w_4}$.
%	Since $N_\nu\bk{w_{134},\widetilde{\chi},\bar{0}}=\Id$, it follows that $\mathcal{K}_{W_{M_E}}\bk{\widetilde{\chi}}=\set{1,w_{134}}$ and $\mathcal{R}_{W_{M_E}}\bk{\widetilde{\chi}}=\set{1,w_{13},w_{14},w_{34}}$.
\end{itemize}

\mbox{}

Note that, for any $E_\nu$ and $\chi_\nu$, any component $\sigma$ of $\Ind_{B_E\bk{F_\nu}\cap M_E\bk{F_\nu}}^{M_E\bk{F_\nu}}\widetilde{\chi}$ is tempered.
By the Langlands Quotient Theorem, $\Ind_{P_E\bk{F_\nu}}^{H_E\bk{F_\nu}} \coset{\sigma \otimes \omega_{\alpha_2}}$ admits a unique irreducible quotient;
this quotient is the image of $N_\nu\bk{w_0,w_2^{-1}\cdot\mu_{\chi_\nu},\lambda_{\bk{0,1,0,0}}}$.

It follows that the maximal semi-simple quotient of $I_{B_E,\nu}\bk{w_2^{-1}\cdot\mu_{\chi_\nu},\lambda_{\bk{0,1,0,0}}}$ is given by
\begin{equation}
\label{Eq:chi_quadratic_s_1/2_maximal_ss_quotient}
\piece{
	\pi_{1,\nu},& \text{if $E_\nu$ is a field or $\chi_\nu\circ\Nm_{E_\nu\rmod F_\nu}=\Id$,} \\
	\pi_{1,\nu}\oplus \pi_{-1,\nu},& \text{if $E_\nu=F_\nu\times K_\nu$, where $K_\nu$ is a field} \\ & \text{and $\chi_\nu\circ\Nm_{E_\nu\rmod F_\nu}\neq\Id$,} \\
	\pi_{\bk{1,1},\nu} \oplus \pi_{\bk{-1,1},\nu} \oplus \pi_{\bk{1,1},\nu} \oplus \pi_{\bk{-1,-1},\nu},& \text{if $E_\nu=F_\nu\times F_\nu\times F_\nu$},
}
\end{equation}
where $\pi_{\epsilon,\nu}$ is the unique irreducible quotient of $\Ind_{P_E\bk{F_\nu}}^{H_E\bk{F_\nu}} \coset{\sigma^\epsilon_\nu \otimes \omega_{\alpha_2}}$.
Furthermore, this maximal semi-simple quotient is the image of $N_\nu\bk{w_0,\widetilde{\chi},\omega_{\alpha_2}}$.

By \Cref{eq:Local_Normalized_Intertwining_Operator_Cocycle_Condition}, we have
\[
N_\nu\bk{\widetilde{w_{213421342}},w_2\cdot\mu_{\chi_\nu},\omega_{\alpha_2}} =
N_\nu\bk{\widetilde{w_{21342}},\mu_{\chi_\nu},\lambda_{\frac{1}{2}}} \circ
N_\nu\bk{\widetilde{w_{134}},\mu_{\chi_\nu},\eta_{\frac{1}{2}}} \circ
N_\nu\bk{\widetilde{w_{2}},w_2\cdot\mu_{\chi_\nu},\omega_{\alpha_2}} .
\]
Since $N_\nu\bk{\widetilde{w_{2}},w_2\cdot\mu_{\chi_\nu},\omega_{\alpha_2}}$ is an isomorphism and the image of $N_\nu\bk{\widetilde{w_{134}},\mu_{\chi_\nu},\eta_{\frac{1}{2}}}$ is $I_\nu\bk{\chi_\nu,\frac{1}{2}}$, it follows that the image of $N_\nu\bk{\widetilde{w_{213421342}},\widetilde{\chi},\omega_{\alpha_2}}$ is a quotient of $I_\nu\bk{\chi_\nu,\frac{1}{2}}$.
In conclusion, \Cref{Eq:chi_quadratic_s_1/2_maximal_ss_quotient} is the maximal semi-simple quotient of $I_\nu\bk{\chi,\frac{1}{2}}$.

We note that, by \Cref{Lemma:Intertwining_operator_of_simple_reflections} and by the discussion above, the $\pi_{\bk{\epsilon},\nu}$ are subrepresentations of $I_{B_E,\nu}\bk{w_2\cdot\mu_{\chi_\nu},-\omega_2}$ and as such, they are eigenspaces of the intertwining operators mentioned above, in the following sense:
%\vertline
%We note that, by the discussion above and \Cref{Lemma:Intertwining_operator_of_simple_reflections}, the $\pi_{\bk{\epsilon},\nu}$ are eigenspaces of certain intertwining operators.
%Namely:
\begin{itemize}
%	\item If $E_\nu$ is a field, $\pi_{\bk{1},\nu}$ is a $1$-eigenspace of $N_\nu\bk{w_1,w_2\cdot\widetilde{\chi},-\omega_2}$.
	
%	\item If $E_\nu=F_\nu\times K_\nu$, where $K_\nu$ is a field, then $\pi_{\bk{\epsilon},\nu}$ is an $\epsilon$-eigenspace of $N_\nu\bk{w_1,w_2\cdot\mu_{\chi_\nu},-\omega_2}$ and $N_\nu\bk{w_3,w_2\cdot\mu_{\chi_\nu},-\omega_2}$.

	\item If $E_\nu=F_\nu\times K_\nu$, where $K_\nu$ is a field, then $\pi_{\bk{\epsilon},\nu}$ is an $\epsilon$-eigenspace of $N_\nu\bk{w_3,w_2\cdot\mu_{\chi_\nu},-\omega_2}$.
			
	\item If $E_\nu=F_\nu\times F_\nu\times F_\nu$, then
	\begin{itemize}
		\item $\pi_{(\epsilon,\delta),\nu}$ is a $\epsilon$-eigenspace for $N_\nu\bk{w_{13},w_2\cdot\mu_{\chi_\nu},-\omega_2}$.
		
		\item $\pi_{(\epsilon,\delta),\nu}$ is a $\delta$-eigenspace for $N_\nu\bk{w_{14},w_2\cdot\mu_{\chi_\nu},-\omega_2}$.
		
		\item $\pi_{(\epsilon,\delta),\nu}$ is a $\bk{\epsilon\cdot\delta}$-eigenspace for $N_\nu\bk{w_{34},w_2\cdot\mu_{\chi_\nu},-\omega_2}$.
	\end{itemize}
\end{itemize}

\begin{Remark}
	In the case $E_\nu=F_\nu\times F_\nu\times F_\nu$, we denote $\pi_{1,\nu}=\pi_{\bk{1,1},\nu}$.
	For all cases, if $\chi_\nu$ is unramified, then $\pi_{1,\nu}$ is unramified.
\end{Remark}

\end{proof}

\section{Square Integrable Residues}
\label{Sec:SquareIntegrableResidues}
We compute the square-integrable residues $Res\bk{s_0,\chi,E}$ for the various values of $s_0$ and $\chi$ separately.
The crux of the computations of the square integrable residues is that they decompose as a direct sum of irreducible representations.
Namely, one can write
\begin{equation}
\label{eq:SquareIntegrableDecomposition}
Res\bk{s_0,\chi,E} = \widehat{\oplus} \sigma_i,
\end{equation}
where $\sigma_i$ are irreducible quotients of $I_{P_E}\bk{\chi,s_0}$.
In particular, if, using Flath's theorem, we write $\sigma_i=\otimes_{i,\nu} \sigma_\nu$ then $\sigma_\nu$ is an irreducible quotient of $I_{\nu}\bk{\chi_\nu,s_0}$, unramified for almost all $\nu\in\Places$.

For each place $\nu$ denote by $\Sigma_\nu\bk{s_0,\chi_E}$ the (finite) set of irreducible quotients of $I_{\nu}\bk{\chi_\nu,s_0}$described in \Cref{Thm:local_results}.

\begin{Remark}
	Note that in all cases, the maximal semi-simple quotient of $I_{\nu}\bk{\chi_\nu,s_0}$ is a direct sum of inequivalent representations.
\end{Remark}

It follows that
\[
Res\bk{s_0,\chi,E} = \widehat{\oplus} \sigma_i \subset \bigotimes{} ' \bk{\bigoplus_{\pi_\nu\in\Sigma_\nu}\pi_\nu}.
\]
Hence, in order to compute $Res\bk{s_0,\chi,E}$, it is enough to check which of the direct summands in the right hand side is realized by the residue of $\Eisen_E\bk{\chi,f_s,s,\cdot}$.
Namely, assuming the pole of $\Eisen_E\bk{\chi,f_s,s,\cdot}$ at $s_0$ is of order $n$, $Res\bk{s_0,\chi,E}$ is the direct sum of all $\pi=\otimes ' \pi_\nu$, where $\pi_\nu\in \Sigma_\nu$ is spherical for almost all $\nu$, and appears in the image of
\[
\lim\limits_{s\to s_0} \bk{s-s_0}^n \Eisen_E\bk{\chi,f_s,s,\cdot} .
\]
In order to determine which $\pi=\otimes ' \pi_\nu$ appears in the image it is enough to check that 
\[
\lim\limits_{s\to s_0} \bk{s-s_0}^n \Eisen_E\bk{\chi,f_s,s,\cdot} \neq 0
\]
for some $f_s\in I_{P_E}\bk{\chi,s_0}$ whose image in $\pi$, under the quotient map, is non-zero.
In order to do this, we apply \Cref{Cor:Kernel_of_Series_is_kernel_of_CT} and show that
\[
\lim\limits_{s\to s_0} \bk{s-s_0}^m \bk{\suml_{w\in \Sigma^{\bfP}_{\bk{\chi,s_0,n}}} M\bk{w,\chi,\lambda_s}f_s} \not\equiv 0.
\]

\subsection{$s_0=\frac{5}{2}$ and $\chi=\Id$}

In this case, $\Eisen_E\bk{\Id,f,s,g}$ admits a simple pole.
We recall, from \cite{SegalEisen}, that
\[
\Sigma^{P_E}_{\bk{E,\Id,\frac{1}{2},1}} = \set{w_o}
\]
On the other hand, for any $\nu\in\Places$ it holds that $N_\nu\bk{w_0,\Id,\lambda_{\frac{5}{2}}}\bk{I_\nu\bk{\Id,\frac{5}{2}}}=\Id_{H_E,\nu}$.
Hence,
$Res\bk{\frac{5}{2},\Id,E}\equiv\Id_{H_E}$, where $\Id_{H_E}$ is the trivial representation of $H_E$.

%We note that $I\bk{\Id,\frac{5}{2}}$ is a standard model and since $I\bk{\Id,-\frac{5}{2}}\cong \hat{I\bk{\Id,\frac{5}{2}}}$ admits $\Id_{H_E}$ as a subrepresentation.
%We conclude that $\Id_{H_E}$ is the unique irreducible quotient of $I\bk{\Id,\frac{5}{2}}$ with Langlands operator $N_\nu\bk{w_0\Id, \lambda_{s,\nu}}$.
%We recall from \cite{SegalEisen} that $M\bk{w_0,\Id,\lambda_s}$ is the only element in $\Eisen\bk{\chi,f,s,g}_\bfCT$ that admits a pole.
%The claim follows.

%\rule{\textwidth}{2pt}
%In this case, we note that for any place $\nu\in\Places$ the representation $\Ind_{B_E\bk{F_\nu}}^{H_E\bk{F_\nu}}\eta_{s,\nu}$ is a standard model, its unique irreducible quotient is the trivial representation $\Id_{H_E\bk{F_\nu}}$ of $H_E$ and the Langlands operator in this case is $N_\nu\bk{w_0,\mu_\chi,\eta_s}$.

%We note that
%\[
%N_\nu\bk{w_0,w_2\cdot\mu_\chi,w_2\cdot\eta_s} = 
%N_\nu\bk{w_0,w_{134}^{-1}\cdot \mu_\chi,w_{134}^{-1}\cdot \eta_s}
%\circ N_\nu\bk{w_{134}, \mu_\chi, \eta_s}
%\circ N_\nu\bk{w_2,w_2\cdot\mu_\chi,w_2\cdot\eta_s} .
%\]
%As $N_\nu\bk{w_2,w_2\cdot\mu_\chi,w_2\cdot\eta_s}$ is an isomorphism and $\Image\bk{N_\nu\bk{w_{134}, \mu_\chi, \eta_s}}=I\bk{\Id,s}$ it follows that the image of $I\bk{\Id,s}$ under $N_\nu\bk{w_0,w_{2134}^{-1}\cdot \mu_\chi,w_{2134}^{-1}\cdot \eta_s}$ is the trivial representation $\Id_{H_E\bk{F_\nu}}$.
%
%We recall from \cite{SegalEisen} that $M\bk{w_0,\chi\lambda_s}$ is the only element in $\Eisen\bk{\chi,f,s,g}_\bfCT$ that admits a pole from which the claim follows.
%\begin{equation}
%Res\bk{\frac{5}{2},\Id,E} = \Id_{H_E} .
%\end{equation}

\subsection{$s_0=\frac{3}{2}$ and $\chi=\chi_E$}
In \cite[Section 5]{MR1918673} it is proven that 
\begin{equation}
Res\bk{\frac{3}{2},\chi_E,E} \equiv \Pi_E ,
\end{equation}
where $\Pi_E$ is the minimal representation of $H_E\bk{\A}$.
Indeed, this follows from \cref{Thm:local_results} and \Cref{eq:SquareIntegrableDecomposition}.

\subsection{$s_0=\frac{1}{2}$, $\chi=\Id$ and $E$ is a Field}
In this case, $\Eisen_E\bk{\Id,f,s,g}$ admits a simple pole.
We recall, from \cite{SegalEisen}, that
\[
\begin{split}
& \Sigma^{P_E}_{\bk{E,\Id,\frac{1}{2},2}} = \set{w_{212}, w_{2121}} \\
& \Sigma^{P_E}_{\bk{E,\Id,\frac{1}{2},1}} = \set{w_{21}, w_{212}, w_{2121}, w_{21212}}
\end{split}
\]
and that
\[
\Sigma^{P_E}_{\bk{E,\Id,\frac{1}{2},1}}\rmod\sim_{\bk{\chi,s_0}} = \set{\set{w_{21},w_{21212}}, \set{w_{212}, w_{2121}}} .
%\Sigma^{P_E}_{\bk{E,\Id,\frac{1}{2},1}} = \set{w_{21}, w_{212}, w_{2121}, w_{21212}}
\]

%\todonum{Can I remove the following paragraph? Double check that all the intertwinig operators appear in section 4}
%We also recall from \Cref{Sec_Local_Representations} that for a place $\nu\in\Places$:
%\begin{itemize}
%\item If $E_\nu$ is not a field then $I_\nu\bk{\Id_\nu,\frac{1}{2}}$ admits a unique irreducible quotient $\pi_{1,\nu}$; it is spherical and is the image of $N_\nu\bk{w_{21},\Id,\lambda_{\frac{1}{2}}}$.
%\item If $E_\nu$ is a field (in particular $\nu\nmid\infty$) then the maximal semi-simple quotient of $I_\nu\bk{\Id_\nu,\frac{1}{2}}$ is a direct sum of two irreducible representations $\pi_{1,\nu}$ and $\pi_{-2,\nu}$; where $\pi_{1,\nu}$ is the spherical constituent.
%The maximal semi-simple quotient is the image of $N_\nu\bk{w_{21},\Id,\lambda_{\frac{1}{2}}}$.
%Furthermore, $\pi_{\epsilon,\nu}$ is the $\epsilon$-eigenspace of $N_\nu\bk{w_{212},\Id,w_{21}^{-1}\cdot\lambda_{\frac{1}{2}}}$.
%\end{itemize}

Let
\begin{equation}
\label{Eq:Definition_of_inert_places}
\Places_{inert} = \set{\nu\in\Places\mvert \text{$E_\nu$ is a field}}.
\end{equation}

From \Cref{eq:SquareIntegrableDecomposition} it follows that
\begin{equation}
Res\bk{\frac{1}{2},\Id,E} = \widehat{\oplus} \Sigma_i \subseteq \bigoplus_{\stackrel{S\subset \Places_{inert}}{\Card{S}<\infty}} \pi_S ,
\end{equation}
where 
\begin{equation}
\label{eq:Pi_S_1/2_chi_trivial_E_field}
\pi_S = \bigotimes_{\nu\in S} \pi_{-2,\nu} \otimes \bigotimes_{\nu\notin S}{'} \pi_{1,\nu} .
\end{equation}
In particular, in order to determine $Res\bk{\frac{1}{2},\Id,E}$ it is enough to determine for which finite subsets $S\subset\Places_{inert}$ will $\pi_S$ appear in $Res\bk{\frac{1}{2},\Id,E}$.

We note that, given a finite subsets $S\subset\Places_{inert}$ and an element $\varphi\in \pi_S$, it holds that
\[
M\bk{w_{212},w_{21}^{-1}\cdot\chi_s}\varphi = \bk{-2}^{\Card{S}-1} \varphi
\]
So
\[
\lim\limits_{s\to\frac{1}{2}} \bk{s-\frac{1}{2}} \coset{M\bk{w_{21},\chi_s} + M\bk{w_{21212},\chi_s}}
\]
vanish on $I\bk{\Id,\frac{1}{2}}$ if and only if $\Card{S}=1$.
Hence,
\begin{equation}
\label{eq:Res_subset_of _sum_E_1/2}
\bigoplus_{\stackrel{S\subset \Places_{inert}}{1\neq\Card{S}<\infty}} \pi_S \subseteq Res\bk{\frac{1}{2},\Id,E}.
\end{equation}
We wish to show that this is an equality.

We consider the residue of the constant term
\begin{align*}
& \lim\limits_{s\to\frac{1}{2}} \bk{s-\frac{1}{2}}
\coset{\Eisen_E\bk{\Id,f,s,g}_\bfCT} \\
& =
\lim\limits_{s\to\frac{1}{2}} \bk{s-\frac{1}{2}} \coset{\bk{M\bk{w_{21},\chi_s} + M\bk{w_{21212},\chi_s}}f_s\bk{g}} \\
& + \lim\limits_{s\to\frac{1}{2}} \bk{s-\frac{1}{2}} \coset{\bk{M\bk{w_{212},\chi_s} + M\bk{w_{2121},\chi_s}}f_s\bk{g}} .
\end{align*}
We note that
\begin{align*}
& \ker\coset{\Res_{s=\frac{1}{2}}
\Eisen_E\bk{\Id,\cdot,s,\cdot}_\bfCT} \\
& = \ker \coset{\Res_{s=\frac{1}{2}}\bk{M\bk{w_{21},\chi_s} + M\bk{w_{21212},\chi_s}}} \bigcap
\ker \coset{\Res_{s=\frac{1}{2}}\bk{M\bk{w_{212},\chi_s} + M\bk{w_{2121},\chi_s}}}
\end{align*}

It is thus enough to show that
\begin{equation}
\label{eq:ker_contained_in_ker}
ker \coset{\Res_{s=\frac{1}{2}}\bk{M\bk{w_{21},\chi_s} + M\bk{w_{21212},\chi_s}}} \subseteq
\ker \coset{\Res_{s=\frac{1}{2}}\bk{M\bk{w_{212},\chi_s} + M\bk{w_{2121},\chi_s}}}
\end{equation}

Noting that
\[
M\bk{w_{212},\chi_s} + M\bk{w_{2121},\chi_s} = \bk{I+M\bk{w_1,w_{212}^{-1}\cdot\chi_s}} M\bk{w_{212},\chi_s}
\]
We write the Laurent series of $I+M\bk{w_1,w_{212}^{-1}\cdot\chi_s}$ and $M\bk{w_{212},\chi_s}$ in a neighbourhood of $s=\frac{1}{2}$:
\begin{align*}
& I+M\bk{w_1,w_{212}^{-1}\cdot\chi_s} = A_1 \bk{s-\frac{1}{2}} +A_2 \bk{s-\frac{1}{2}}^2+... \\
& M\bk{w_{212},\chi_s} = \frac{B_{-2}}{\bk{s-\frac{1}{2}}^2} + \frac{B_{-1}}{\bk{s-\frac{1}{2}}} + B_0 + ... 
\end{align*}
Composing the two series we get
\[
M\bk{w_{212},\chi_s} + M\bk{w_{2121},\chi_s}
 = \frac{A_1\circ B_{-2}}{\bk{s-\frac{1}{2}}} + \bk{A_2\circ B_{-2}+A_1\circ B_{-1}} + ...
\]
Hence
\[
\lim\limits_{s\to\frac{1}{2}} \bk{s-\frac{1}{2}} \coset{\bk{M\bk{w_{212},\chi_s} + M\bk{w_{2121},\chi_s}}f_s\bk{g}} = A_1\circ B_{-2} .
\]
Since $\Image\bk{A_{-2}}=\pi_\emptyset$ \Cref{eq:ker_contained_in_ker} follows.

Alternatively, the equality in \Cref{eq:Res_subset_of _sum_E_1/2} follows from the results of \cite{LaoResidualSpectrum}.
%By the general theory of spectral decomposition,
We note that
\begin{equation}
Res\bk{\frac{1}{2},\Id,E} \subseteq L^2_{\coset{B,\mu_\chi}} ,
\end{equation}
where $L^2_{\bk{B,\mu_\chi}}$ is the subspace of $L^2\bk{H_E\bk{F}\lmod H_E\bk{\A}}$ spanned by automorphic forms with cuspidal data $\coset{B,\mu_\chi}$.
We recall from \cite{LaoResidualSpectrum} that 
\[
L^2_{\coset{B,\mu_\chi}} = \bigoplus_{\stackrel{S\subset \Places_{inert}}{1\neq\Card{S}<\infty}} \pi_S
\]
and hence also
\begin{equation}
\label{eq:SI-1/2-triv-field}
Res\bk{\frac{1}{2},\Id,E} = \bigoplus_{\stackrel{S\subset \Places_{inert}}{1\neq\Card{S}<\infty}} \pi_S .
\end{equation}

%\subsection{$s_0=\frac{1}{2}$, $\chi^2=\Id$ and $\chi\neq\Id\neq \chi\circ \Nm_{E\rmod F}$}
\subsection{$s_0=\frac{1}{2}$, $\chi^2=\Id\neq \chi\circ \Nm_{E\rmod F}$}

We separate the discussion into the various kinds of \'etale cubic algebras.
We partition $\Places$ as follows,
\begin{equation}
\Places = \Places^{\bk{E,\chi}}_{sph} \bigcupdot \Places^{\bk{E,\chi}}_{1} \bigcupdot \Places^{\bk{E,\chi}}_{2} \bigcupdot \Places^{\bk{E,\chi}}_{3},
\end{equation}
where
\begin{itemize}
\item If $\nu\in\Places^{\bk{E,\chi}}_{sph}$ then $I_\nu\bk{\chi_\nu,\frac{1}{2}}$ admits a unique irreducible (spherical) quotient. Namely, $E_\nu$ is not a field and $\chi_\nu\circ\Nm_{K_\nu\rmod F_\nu}=\Id_\nu$ or, $E_\nu$ is a field and $\chi_\nu\circ\Nm_{E_\nu\rmod F_\nu}\neq\Id_\nu$.
\item If $\nu\in\Places^{\bk{E,\chi}}_{1}$ then $E_\nu$ is a field and $\chi_\nu=Id_\nu$.
\item If $\nu\in\Places^{\bk{E,\chi}}_{2}$ then $E_\nu=F_\nu\times K_\nu$ where $K_\nu$ is a field and $\chi_\nu\circ\Nm_{K_\nu\rmod F_\nu}\neq\Id_\nu$.
\item If $\nu\in\Places^{\bk{E,\chi}}_{3}$ then $E_\nu=F_\nu\times F_\nu\times F_\nu$ and $\chi_\nu\neq\Id_\nu$.
\end{itemize}

Also, let
\begin{equation}
\label{Eq:Definition_of_non-sph_Places}
\Places^{\bk{E,\chi}}_{non-sph} = \Places^{\bk{E,\chi}}_{1} \bigcupdot \Places^{\bk{E,\chi}}_{2} \bigcupdot \Places^{\bk{E,\chi}}_{3} .
\end{equation}

We now turn to compute $Res\bk{\frac{1}{2},\chi,E}$ for the various possible cases.
This is done by comparing the action of
\[
\lim\limits_{s\to s_0} \bk{s-s_0} \suml_{w\in \Sigma^{P_E}_{\bk{\chi,\frac{1}{2},1}}} M\bk{w,\chi,\lambda_s}
\]
on quotients of $I\bk{\chi,\frac{1}{2}}$.
This is done by identifying these quotients as images or eigenspaces of various intertwining operators.

%\todonum{Improve the use of $S$ in what follows, maybe just write the sum over $S_1$, $S_2$ etc. separately? Also, maybe add the superscript in $\pi_{\epsilon,\nu}^{\bk{\chi_\nu,E_\nu,s_0}}$? Update \Cref{Thm:Main_Sq_Int}}

\subsubsection{$E$ is a Field}
In this case
\[
\Sigma^{P_E}_{\bk{E,\chi,\frac{1}{2},1}}\rmod\sim_{\bk{\chi,\frac{1}{2}}} = \set{\set{w_{212}},\set{w_{2121}},\set{w_{21212}}} ,
\]
where all three elements are mutually inequivalent and their associated intertwining operators admit simple poles at $s=\frac{1}{2}$.
Let $S\subset \Places^{\bk{E,\chi}}_{1} \bigcupdot \Places^{\bk{E,\chi}}_{2} \bigcupdot \Places^{\bk{E,\chi}}_{3}$ be a finite subset such that
\[
%S= S_{1} \bigcupdot S_{2} \bigcupdot S_3,
\dot{S}= S_{1} \bigcupdot S_{2} \bigcupdot \bk{S_{3,\bk{-1,1}} \bigcupdot S_{3,\bk{1,-1}} \bigcupdot S_{3,\bk{-1,-1}} },
\]
with
%$S_{i} \subset \Places^{\bk{E,\chi}}_{i}$
\begin{itemize}
\item $S_{1} = S\cap \Places^{\bk{E,\chi}}_{1} \subset \Places^{\bk{E,\chi}}_{1}$.
\item $S_{2} = S\cap \Places^{\bk{E,\chi}}_{2} \subset \Places^{\bk{E,\chi}}_{2}$.
\item $S_{3,\bk{\epsilon_1,\epsilon_2}} \subset \Places^{\bk{E,\chi}}_{3}$.
\end{itemize}
Note that $\dot{S}$ is a set with a choice of the distinct subsets $S_{3,\bk{\epsilon_1,\epsilon_2}}$ and $S$ denotes its underlying set.
We let
\begin{equation}
\label{eq:Pi_S_1/2_chi_non-trivial_E_field}
\pi_{\dot{S}} = \bk{\bigotimes_{\nu\in S_1} \pi_{-2,\nu}} \otimes \bk{\bigotimes_{\nu\in S_2} \pi_{-1,\nu}} \otimes \bk{
%	\otimes_{\bk{\epsilon_1,\epsilon_2}\neq \bk{1,1}}
	\bigotimes_{\nu\in S_{3,\bk{\epsilon_1,\epsilon_2}}} \pi_{\bk{\epsilon_1,\epsilon_2},\nu}} \otimes \bk{\bigotimes_{\nu\notin S}{'} \pi_{1,\nu}} .
\end{equation}

It follows from the discussion above that $\pi_{S,\nu}\subset N_\nu\bk{w_{21212},\chi_{\frac{1}{2},\nu}}\bk{I_\nu\bk{\chi_\nu,\frac{1}{2}}}$ .
It follows that $\pi_S$ is a  direct summand of $Res\bk{\frac{1}{2},\chi,E}$.
On the other hand,
\[
Res\bk{\frac{1}{2},\chi,E} \subseteq \bigoplus_{\stackrel{\dot{S}\subset \Places^{\bk{E,\chi}}_{non-sph}}{\Card{S}<\infty}} \pi_{\dot{S}}
\]
and hence
\begin{equation}
Res\bk{\frac{1}{2},\chi,E} = \bigoplus_{\stackrel{\dot{S}\subset \Places^{\bk{E,\chi}}_{non-sph}}{\Card{S}<\infty}} \pi_{\dot{S}}.
\end{equation}

\subsubsection{$E=F\times K$ and $K$ is a Field}
In this case
\[
\Sigma^{P_E}_{\bk{E,\chi,\frac{1}{2},1}}\rmod\sim_{\bk{\chi,\frac{1}{2}}} =
\set{ \set{w_{2321}, w_{2132132}}, \set{w_{2132}, w_{21323}}, \set{w_{21321}, w_{213213}}} .
\]
We recall from \cite{SegalEisen} that the poles of $M\bk{w_{2132},\chi_s}$ and $M\bk{w_{21323},\chi_s}$ cancel each other.
The other poles do not cancel and we wish to analyze the image of the intertwining operator
\[
\lim\limits_{s\to \frac{1}{2}} \bk{s-\frac{1}{2}} \coset{M\bk{w_{232123},\chi_s} + M\bk{w_{23212},\chi_s} + M\bk{w_{2321232},\chi_s} + M\bk{w_{2321},\chi_s}} .
\]
Note that
\begin{align*}
& w_{2321232} = w_{2321}w_{232}, \\
& w_{213213} = w_{21321}w_{3} \\
%& w_{2321232} \sim_{\bk{\chi,s_0}} w_{2321}, \quad  w_{2321232} = w_{2321}w_{232}, \\
%& w_{213213} \sim_{\bk{\chi,s_0}} w_{21321}, \quad  w_{213213} = w_{21321}w_{3} \\
\end{align*}
and
\begin{align*}
& J\bk{w_3,w_{21321}^{-1}\cdot\chi_s} = \frac{\Lfun_K\bk{s-\frac{1}{2},\chi\circ\Nm_{K\rmod F}}}{\Lfun_K\bk{s+\frac{1}{2},\chi\circ\Nm_{K\rmod F}}} \\
& J\bk{w_{232},w_{2321}^{-1}\cdot\chi_s} = 
\frac{\Lfun_F\bk{s-\frac{3}{2},\chi} \Lfun_F\bk{s+\frac{1}{2},\chi} \Lfun_K\bk{s-\frac{1}{2},\chi\circ\Nm_{K\rmod F}}}{\Lfun_F\bk{s-\frac{1}{2},\chi} \Lfun_F\bk{s+\frac{3}{2},\chi} \Lfun_K\bk{s+\frac{1}{2},\chi\circ\Nm_{K\rmod F}}} .
\end{align*}
According to \cite[Theorem 2.2]{MR2882696} it follows that $\epsilon_F\bk{s,\chi}\equiv 1$ and $\epsilon_K\bk{s,\chi\circ\Nm_{K\rmod F}}\equiv 1$ and hence
\begin{align*}
& \Lfun_F\bk{s,\chi} = \Lfun_F\bk{1-s,\chi} \\
& \Lfun_K\bk{s,\chi\circ\Nm_{K\rmod F}} = \Lfun_K\bk{1-s,\chi\circ\Nm_{K\rmod F}} .
\end{align*}
It follows that
\[
\lim\limits_{s\to\frac{1}{2}} J\bk{w_3,w_{21321}^{-1}\cdot\chi_s} = 1,\quad
\lim\limits_{s\to\frac{1}{2}} J\bk{w_{232},w_{2321}^{-1}\cdot\chi_s} = 1 .
\]

Let $S\subset \Places^{\bk{E,\chi}}_{2} \bigcupdot \Places^{\bk{E,\chi}}_{3}$ be a finite subset such that
\[
%S= S_{1} \bigcupdot S_{2} \bigcupdot S_3,
\dot{S}= S_{2} \bigcupdot \bk{S_{3,\bk{-1,1}} \bigcupdot S_{3,\bk{1,-1}} \bigcupdot S_{3,\bk{-1,-1}} },
\]
with
\begin{itemize}
\item $S_{2} \subset \Places^{\bk{E,\chi}}_{2}$.
\item $S_{3,\bk{\epsilon_1,\epsilon_2}} \subset \Places^{\bk{E,\chi}}_{3}$.
\end{itemize}
We let
\begin{equation}
\label{eq:Pi_S_1/2_chi_non-trivial_FxK}
\pi_{\dot{S}} = \bk{\bigotimes_{\nu\in S_2} \pi_{-1,\nu}} \otimes \bk{
%	\otimes_{\bk{\epsilon_1,\epsilon_2}\neq \bk{1,1}}
	\bigotimes_{\nu\in S_{3,\bk{\epsilon_1,\epsilon_2}}} \pi_{\bk{\epsilon_1,\epsilon_2},\nu}} \otimes \bk{\bigotimes_{\nu\notin S}{'} \pi_{1,\nu}} .
\end{equation}

The irreducible representation $\pi_{\dot{S}}$ appears in the images of all
\begin{align*}
& \lim\limits_{s\to\frac{s}{2}} \bk{s-\frac{1}{2}} M\bk{w_{2321},\chi_s} \\
& \lim\limits_{s\to\frac{s}{2}} \bk{s-\frac{1}{2}} M\bk{w_{21321},\chi_s} \\
& \lim\limits_{s\to\frac{s}{2}} \bk{s-\frac{1}{2}} M\bk{w_{213213},\chi_s} \\
& \lim\limits_{s\to\frac{s}{2}} \bk{s-\frac{1}{2}} M\bk{w_{2132132},\chi_s} .
\end{align*}

On the other hand, since $w_3=\widetilde{w_{34}}$ and $w_{232}=\widetilde{w_{2342}}$, we note that both $M\bk{w_3,w_{21321}^{-1}\cdot\chi_{\frac{1}{2}}}$ and $M\bk{w_{232},w_{2321}^{-1}\cdot\chi_{\frac{1}{2}}}$ acts on $\pi_S$ by
\[
\bk{-1}^{\Card{S_2}+\Card{S_{3,\bk{-1,1}}} +\Card{S_{3,\bk{-1,1}}} } \operatorname{Id} .
\]
It follows that
\begin{equation}
Res\bk{\frac{1}{2},\chi,F\times K} = \bigoplus_{\stackrel{\dot{S}\subset \Places^{\bk{E,\chi}}_{non-sph}}{\Card{S^\ast} \text{ is even}}} \pi_{\dot{S}} ,
\end{equation}
where
\begin{equation}
\label{Eq:S^ast_def}
S^\ast = S_{2} \bigcupdot S_{3,\bk{-1,1}} \bigcupdot S_{3,\bk{-1,1}}
\end{equation}

\subsubsection{$E=F\times F\times F$}
In this case
%\begin{align*}
%\Sigma^{P_E}_{\bk{F\times F\times F,\chi,\frac{1}{2},1}}
%= & \left\{ w_{21324}, w_{21423}, w_{23421}, w_{21342}, w_{213421}, w_{213423}, w_{213424}, \right. \\
%& \left. w_{2134213}, w_{2134214}, w_{2134234}, w_{21342134}, w_{213421342} \right\} ,
%\end{align*}
%while
\[
\begin{split}
\Sigma^{P_E}_{\bk{E,\chi,\frac{1}{2},1}}\rmod\sim_{\bk{\chi,\frac{1}{2}}} = &
\left\{ \set{w_{21324}, w_{21423}, w_{23421}, w_{213421342}} \right. \\
& \set{w_{21342}, w_{2134213}, w_{2134214}, w_{2134234}} \\
& \left. \set{w_{213421}, w_{213423}, w_{213424}, w_{21342134}} \right\}
\end{split}
\]

Note that
\begin{align*}
& w_{213421342} = w_{21324} w_{2132} = w_{21423} w_{2142} = w_{23421} w_{2342} \\
& w_{21342134} = w_{213424} w_{13} = w_{213423} w_{14} = w_{213421} w_{34} \\
& w_{2134213} = w_{21342} w_{13} \\
& w_{2134214} = w_{21342} w_{14} \\
& w_{2134234} = w_{21342} w_{34} .
\end{align*}

Furthermore,
\begin{align*}
& J\bk{w_{2132},w_{21324}^{-1}\cdot\chi_s} = J\bk{w_{2142},w_{21423}^{-1}\cdot\chi_s} = J\bk{w_{2342},w_{23421}^{-1}\cdot\chi_s} = \frac{\Lfun\bk{s-\frac{1}{2},\chi} \Lfun\bk{s-\frac{3}{2},\chi}}{\Lfun\bk{s+\frac{1}{2},\chi} \Lfun\bk{s+\frac{3}{2},\chi}} \\
& J\bk{w_{13}, w_{21342}^{-1}\cdot\chi_s} = J\bk{w_{14}, w_{21342}^{-1}\cdot\chi_s} = J\bk{w_{34}, w_{21342}^{-1}\cdot\chi_s} = \bk{\frac{\Lfun\bk{s-\frac{1}{2},\chi}}{\Lfun\bk{s+\frac{1}{2},\chi}}}^2 \\
& J\bk{w_{13}, w_{213424}^{-1}\cdot\chi_s} = J\bk{w_{14}, w_{213423}^{-1}\cdot\chi_s} = J\bk{w_{34}, w_{213421}^{-1}\cdot\chi_s} = \bk{\frac{\Lfun\bk{s-\frac{1}{2},\chi}}{\Lfun\bk{s+\frac{1}{2},\chi}}}^2 .
\end{align*}

According to \cite[Theorem 2.2]{MR2882696} it follows that $\epsilon_F\bk{s,\chi}\equiv 1$ and hence
\[
\Lfun_F\bk{s,\chi} = \Lfun_F\bk{1-s,\chi}
\]
It follows that
\[
\lim\limits_{s\to\frac{1}{2}} \frac{\Lfun\bk{s-\frac{1}{2},\chi} \Lfun\bk{s-\frac{3}{2},\chi}}{\Lfun\bk{s+\frac{1}{2},\chi} \Lfun\bk{s+\frac{3}{2},\chi}} =
\lim\limits_{s\to\frac{1}{2}} \bk{\frac{\Lfun\bk{s-\frac{1}{2},\chi}}{\Lfun\bk{s+\frac{1}{2},\chi}}}^2 = 1 .
\]

Let $S\subset \Places^{\bk{E,\chi}}_{3}$ be a finite subset such that
\[
%S= S_{1} \bigcupdot S_{2} \bigcupdot S_3,
\dot{S} = S_{3,\bk{-1,1}} \bigcupdot S_{3,\bk{1,-1}} \bigcupdot S_{3,\bk{-1,-1}},
\]
with $S_{3,\bk{\epsilon_1,\epsilon_2}} \subset \Places^{\bk{E,\chi}}_{3}$.
We let
\begin{equation}
\label{eq:Pi_S_1/2_chi_non-trivial_FxFxF}
\pi_{\dot{S}} = \bk{\bigotimes_{\nu\in S_2} \pi_{-1,\nu}} \otimes \bk{
%	\otimes_{\bk{\epsilon_1,\epsilon_2}\neq \bk{1,1}}
	\bigotimes_{\nu\in S_{3,\bk{\epsilon_1,\epsilon_2}}} \pi_{\bk{\epsilon_1,\epsilon_2},\nu}} \otimes \bk{\bigotimes_{\nu\notin S}{'} \pi_{1,\nu}} .
\end{equation}
The irreducible representation $\pi_{\dot{S}}$ appears in the images of all
\[
\lim\limits_{s\to\frac{s}{2}} \bk{s-\frac{1}{2}} M\bk{w,\chi_s},
\]
where $w\in \Sigma^{P_E}_{\bk{F\times F\times F,\chi,\frac{1}{2},1}}$.

Fix a standard section $f_s\in I\bk{\chi,s}$ so that
\[
\lim\limits_{s\to\frac{s}{2}} \bk{s-\frac{1}{2}} M\bk{w_{21342},\chi_s} f_s = \varphi \in \pi_S
\]
A similar analysis to the one performed in the previous cases shows that
%\begin{align*}
%& \lim\limits_{s\to\frac{s}{2}} \bk{s-\frac{1}{2}} \suml_{w\in \Sigma^{P_E}_{\bk{F\times F\times F,\chi,\frac{1}{2},1}}} M\bk{w,\chi_s} = \\
%& 4\bk{1 + \bk{-1}^{\Card{S_{3,\bk{1,-1}}} + \Card{S_{3,\bk{-1,1}}}} + \bk{-1}^{\Card{S_{3,\bk{1,-1}}} + \Card{S_{3,\bk{-1,-1}}}}  + \bk{-1}^{\Card{S_{3,\bk{-1,1}}} + \Card{S_{3,\bk{-1,-1}}}}} \varphi .
%\end{align*}
%\todonum{Fix this!}
%It follows that
%\begin{equation}
%Res\bk{\frac{1}{2},\chi,F\times F\times F} = \bigoplus_{\stackrel{\dot{S}\subset \Places^{\bk{E,\chi}}_{non-sph}}{ \Card{S_{3,\bk{-1,1}}} \equiv \Card{S_{3,\bk{1,-1}}} \equiv \Card{S_{3,\bk{-1,-1}}} \pmod{2} }} \pi_{\dot{S}} .
%\end{equation}
\begin{align*}
& \lim\limits_{s\to\frac{s}{2}} \bk{s-\frac{1}{2}} \suml_{w\in \Sigma^{P_E}_{\bk{F\times F\times F,\chi,\frac{1}{2},1}}} M\bk{w,\chi_s} = \\
& 4\bk{1 + \bk{-1}^{\Card{S_{3,\bk{1,-1}}}} + \bk{-1}^{\Card{S_{3,\bk{-1,1}}} }  + \bk{-1}^{\Card{S_{3,\bk{-1,-1}}}}} \varphi .
\end{align*}
It follows that
\begin{equation}
Res\bk{\frac{1}{2},\chi,F\times F\times F} = \bigoplus_{\stackrel{\dot{S}\subset \Places^{\bk{E,\chi}}_{non-sph}}{ \dot{S}\in \mathcal{V}^{\bk{E,\chi}} }} \pi_{\dot{S}} ,
\end{equation}
%\todonum{Fix this!}
where
\begin{equation}
\label{Eq:sets_of_places_which_appear_in_residue_E_split_chi_quad}
\mathcal{V}^{\bk{E,\chi}} =
\set{\dot{S}\subset \Places^{\bk{E,\chi}}_{non-sph} \mvert 
	\begin{array}{l}
	\Card{S_{3,\bk{-1,1}}} \cdot \Card{S_{3,\bk{1,-1}}} \\
	\equiv \Card{S_{3,\bk{-1,1}}} \cdot \Card{S_{3,\bk{-1,-1}}} \\
	\equiv \Card{S_{3,\bk{1,-1}}} \cdot \Card{S_{3,\bk{-1,-1}}} \pmod{2}
	\end{array}} .
\end{equation}

\subsection{Summary}

We record the results of this section as a theorem.
\begin{Thm}
	\label{Thm:Main_Sq_Int}
	The square integrable residues $Res\bk{s_0,\chi,E}$ with $\Real\bk{s_0}>0$ are given as follows:
	\begin{enumerate}
		\item If $s_0=\frac{5}{2}$ and $\chi=\Id$ then $Res\bk{s_0,\chi,E}=\Id$, the trivial representation of $H_E\bk{\A}$, for any $E$.
		
		\item If $s_0=\frac{3}{2}$ and $\chi=\chi_E$, where $E$ is a Galois \'etale cubic algebra over $F$, then $Res\bk{s_0,\chi,E}=\Pi_E$, the minimal representation $\Pi_E$ of $H_E\bk{\A}$.
		
		\item If $s_0=\frac{1}{2}$, $\chi=\Id$ and $E$ is a field extension then
		\begin{equation}
		Res\bk{\frac{1}{2},\Id,E} = \bigoplus_{\stackrel{S\subset \Places_{inert}}{1\neq\Card{S}<\infty}} \pi_S ,
		\end{equation}
		where $\Places_{inert}$ is defined in \Cref{Eq:Definition_of_inert_places} and $\pi_S$ is given by \Cref{eq:Pi_S_1/2_chi_trivial_E_field}.
		
		\item Assume that $s_0=\frac{1}{2}$ and $\chi^2=\Id\neq\chi\circ\Nm_{E\rmod F}$ and recall the definition of $\Places^{\bk{E,\chi}}_{non-sph}$ from \Cref{Eq:Definition_of_non-sph_Places}.
		\begin{itemize}
			\item If $E$ is field then
			\begin{equation}
			Res\bk{\frac{1}{2},\chi,E} = \bigoplus_{\stackrel{\dot{S}\subset \Places^{\bk{E,\chi}}_{non-sph}}{\Card{S}<\infty}} \pi_{\dot{S}}.
			\end{equation}
			where $\pi_S$ is given by \Cref{eq:Pi_S_1/2_chi_non-trivial_E_field}.
			
			\item If $E=F\times K$, where $K$ is field, then
			\begin{equation}
			Res\bk{\frac{1}{2},\chi,F\times K} = \bigoplus_{\stackrel{\dot{S}\subset \Places^{\bk{E,\chi}}_{non-sph}}{\Card{S^\ast} \text{ is even}}} \pi_{\dot{S}} ,
			\end{equation}
			where $\pi_{\dot S}$ is given by \Cref{eq:Pi_S_1/2_chi_non-trivial_FxK} and $S^\ast$ is given in \Cref{Eq:S^ast_def}.
			
			\item If $E=F\times F\times F$ then
			\begin{equation}
			\label{Eq:Sq_int_thm_split_case_nontrivial_char}
			Res\bk{\frac{1}{2},\chi,F\times F\times F} = \bigoplus_{\stackrel{\dot{S}\subset \Places^{\bk{E,\chi}}_{non-sph}}{ \dot{S}\in \mathcal{V}^{\bk{E,\chi}} }} \pi_{\dot{S}} ,
			\end{equation}
			where $\pi_{\dot S}$ is given by \Cref{eq:Pi_S_1/2_chi_non-trivial_FxFxF} and $\mathcal{V}^{\bk{E,\chi}}$ is given by \Cref{Eq:sets_of_places_which_appear_in_residue_E_split_chi_quad}.
		\end{itemize}
	\end{enumerate}
\end{Thm}

We note here again that, in item \emph{(4)}, $\dot{S}$ is a finite set $S\subset\Places$ together with a choice $S_3=S_{3,\bk{-1,1}} \bigcupdot S_{3,\bk{1,-1}} \bigcupdot S_{3,\bk{-1,-1}}$, where $S_3=S\cap \Places^{\bk{E,\chi}}_{3}$.
%$S_{3,\bk{\epsilon_1,\epsilon_2}} \subset \Places^{\bk{E,\chi}}_{3}$.

\section{Non-Square-integrable Residues}
\label{Sec:NonSquareIntegrableResidues}
We compute the non-square-integrable residues $Res\bk{s_0,\chi,E}$ for the various values of $s_0$ and $\chi$.

For $\chi=\Id$, we compute $Res\bk{s,\chi,E}$ using a Siegel-Weil type identity.
Namely, we prove an identity between $Res\bk{s,\chi,E}$ and a special value, or residue, of an Eisenstein series associated to an induction from a different parabolic subgroup of $H_E$.
The case of $s=\frac{3}{2}$ was essentially computed in \cite{RallisSchiffmannPaper}, we recall the relevant results and compute $Res\bk{\frac{3}{2},\Id,F\times K}$.
For $s=\frac{1}{2}$, we prove an identity between $Res\bk{\frac{1}{2},\Id,F\times K}$ and the special value of the Eisenstein series associated to the degenerate principal series induced from $M_{\set{\alpha_1,\alpha_2}}$.
We then use the fact that the relevant local degenerate principal series is semi-simple in order to compute $Res\bk{\frac{1}{2},\Id,F\times K}$.

For $\chi\neq \Id$ we use \Cref{Cor:Kernel_of_Series_is_kernel_of_CT}, similarly to the square-integrable case.
The non-square integrable case is more involved as \Cref{eq:SquareIntegrableDecomposition} is not valid.
The computation of $Res\bk{\frac{1}{2},\chi,F\times K}$ here, relies on a calculation of the images of certain local intertwining operators; surprisingly, some of this calculations make use of our results about $Res\bk{\frac{1}{2},\Id,F\times K}$.

%For most cases we use similar methods to the square-integrable case and rely on \Cref{Cor:Kernel_of_Series_is_kernel_of_CT} while some complications arise since \Cref{eq:SquareIntegrableDecomposition} is not valid.
%However, in the first case ($s_0=\frac{3}{2}$, $\chi=\Id$ and $E=F\times K$ Where $K$ is a Field) we use results attained in \cite{RallisSchiffmannPaper} using a Siegel-Weil-type identity.
%We use two methods to compute non-square-integrable residues, one is similar to the square-integrable case and relies \Cref{Cor:Kernel_of_Series_is_kernel_of_CT} while complications arise since \Cref{eq:SquareIntegrableDecomposition} is not valid.
%The other method is to apply Siegel-Weil-type identities.

\subsection{$\chi=\Id$}

\subsubsection{$s_0=\frac{3}{2}$}
In this case, $E=F\times K$ Where $K$ is a Field.
It was essentially dealt with in \cite{RallisSchiffmannPaper}.
We recall the results from there and use them in order to compute $Res\bk{\frac{3}{2},\Id,F\times K}$.
We start with the Siegel-Weil-type identity \cite[Corollary 3.17]{RallisSchiffmannPaper}:
\begin{Prop}
There exists a non-zero constant $C$ such that for every $f\in I_{P_{F\times K}}\bk{s}$ it holds that
\[
\coset{\bk{s-\frac{3}{2}}\Eisen_{F\times K}\bk{f,s,g}} \res{s=\frac{3}{2}} = C \cdot \Eisen_{P_{\set{2,3,4}}}\bk{A\bk{w_{232}}f,1,g} ,
\]
where 
\begin{itemize}
	\item $\Eisen_{P_{\set{2,3,4}}}$ is the degenerate Eisenstein series associated to the normalized parabolic induction $I_{P_{2,3,4}}\bk{s}$ of $\omega_{\alpha_1}^s$ from the maximal parabolic $P_{\set{2,3,4}}$ and
	\item $A\bk{w_{232}}$ is the leading term of $M\bk{w_{232},\lambda}$ at $\lambda_{\frac{3}{2}}$, given by
	\[
	A\bk{w_{232}} = \lim_{s\to\frac{3}{2}} \coset{\bk{s-\frac{3}{2}} M\bk{w_{232},\lambda_{s}}} .
	\]
\end{itemize}

%Where $\Eisen_{P_{\set{2,3,4}}}$ is the degenerate Eisenstein series associated to the normalized parabolic induction, from the maximal parabolic $P_{2,3,4}$, of $\modf{P_{2,3,4}}^s$ and $A\bk{w_{232}}$ is the leading term of $M\bk{w_{232},\lambda}$ at $\mu_{\frac{3}{2}}^{P}$.	
\end{Prop}

We also recall that $A\bk{w_{232}}=\placestimes A_\nu$ and $A_\nu$ is onto for any place $\nu\in\Places$. 

\begin{Prop}
\[
Res\bk{\frac{3}{2},\Id,F\times K} \cong I_{P_{\set{2,3,4}}}\bk{1} .
\]
\end{Prop}

\begin{proof}
We recall from the proof of \cite[Proposition 3.5]{RallisSchiffmannPaper} that
\[
%\Sigma_{\frac{3}{2}} 
\Sigma^{P_{\set{2,3,4}}}_{\bk{\Id,1,0}} \rmod \sim_{\bk{\Id,1}}
= \set{\set{1},\set{w_1},\set{w_{12}},\set{w_{123},w_{1232}},\set{w_{12321}}}.
\]
In particular, the term associated with $w=1$ is not canceled in the constant term of $\Eisen_{P_{\set{2,3,4}}}\bk{A\bk{w_{232}}f,1,g}$.
It follows from \Cref{Cor:Kernel_of_Series_is_kernel_of_CT} that 
\[
ker\bk{\Eisen_{P_{\set{2,3,4}}}\bk{\cdot,1,g}} = \bk{0}
\]
from which the claim follows.
\end{proof}
%\todonum{Rewrite this using the notations of this paper}
%\todonum{Update this, we have $I_Q(1/6)=I_Q^0(1/6)$ and make sure everything I use does appear in RS paper.}

% [1,0,0,0] [0,0,1,0] [0,0,0,1]
% [0,1,0,0]
% [1,1,0,0] [0,1,1,0] [0,1,0,1]
% [1,1,1,0] [1,1,0,1] [0,1,1,1]
% [1,1,1,1]
% [1,2,1,1]

% [2,2,0,0] = (2,2,-2,-2)
% So \rho_{M_{1,2}}=(1,1,-1,-1)

\subsubsection{$s_0=\frac{1}{2}$}

In this case, we have $E=F\times K$, where $K$ is a quadratic \'etale algebra over $F$.
We approach this using a similar approach to the previous case.
We first establish a Siegel-Weil type identity and apply this identity to calculate $Res\bk{\frac{1}{2},\Id,F\times K}$.

We recall \cite[Prop. 3.9]{RallisSchiffmannPaper}.
Let us define the following normalized spherical Eisenstein series
\begin{equation}
\label{Eq:NormalizedEisensteinSeries}
\Eisen_{B_E}^\sharp\bk{\lambda,g} = 
\coset{\prodl_{\alpha\in\Phi^+} \zfun_{F_\alpha}\bk{\gen{\lambda,\check{\alpha}}+1} l_\alpha^+\bk{\lambda} l_\alpha^-\bk{\lambda}}
\Eisen_{B_E}\bk{f^0_\lambda,\lambda,g} ,
\end{equation}
where
\[
l_\alpha^{\pm}\bk{\lambda} = \gen{\lambda,\check{\alpha}}\pm 1 .
\]
%\todonum{Check that notations for spherical sections are defined and consistent}

\begin{Prop}
	The normalized Eisenstein series $\Eisen_{B_E}^\sharp\bk{\lambda,g}$ is entire and $W_H$-invariant in the sense that for any $w\in W_H$ it holds that $\Eisen_{B_E}^\sharp\bk{w\cdot \lambda,g}=\Eisen_{B_E}^\sharp\bk{\lambda,g}$.
\end{Prop}

In particular, it holds that
\[
\Eisen_{B_E}^\sharp\bk{\widetilde{\lambda_{\bk{-1,2,-1,-1}}},g} = \Eisen_{B_E}^\sharp\bk{\widetilde{\lambda_{\bk{-1,-1,1,1}}},g} .
\]

We evaluate both sides of the equation at \Cref{App:Evaluation_of_Normalized_Eisenstein_Series}.
In particular, evaluating the left-hand side yields
\[
\piece{
2^7 \cdot 3 \cdot \zfun_F\bk{2}^2 \cdot \zfun_F\bk{3} \cdot \zfun_K\bk{2} \cdot R_F^3 \cdot R_K^2 \cdot \Res_{s=\frac{1}{2}}\coset{\Eisen_{E}\bk{f^0,s,g}},& \text{if $K$ is a field,} \\ 
-2^9  \cdot 3 \cdot \zfun_F\bk{2}^4 \cdot \zfun_F\bk{3} \cdot R_F^7 \cdot \Res_{s=\frac{1}{2}}\coset{\Eisen_{E}\bk{f^0,s,g}} ,& \text{if $K=F\times F$,} }
\]
while the right-hand side is given by
\[
\piece{
2^6 \cdot 3 \cdot \zfun_F\bk{2}^2 \cdot \zfun_K\bk{2} \cdot R_F^4\cdot R_K^2 \Eisen_{P_{\set{1,2}}}\bk{f^0,0,g}	,& \text{if $K$ is a field,} \\ 
-2^8 \cdot 3 \cdot \zfun_F\bk{2}^4 \cdot R_F^8 \cdot \Eisen_{P_{\set{1,2}}}\bk{f^0,\bar{0},g}, & \text{if $K=F\times F$.} }
\]
We conclude that for any $K$ it holds that
\[
\Res_{s=\frac{1}{2}}\coset{\Eisen_{E}\bk{f^0,s,g}} 
= \frac{R_F}{2 \cdot \zfun_F\bk{3}} \cdot \Eisen_{P_{\set{1,2}}}\bk{f^0,\bar{0},g} .
\]

We let
\[
A\bk{w_{21}} = \lim_{s\to\frac{1}{2}} \coset{\bk{s-\frac{1}{2}} M\bk{w_{21},\lambda_{s}}} 
\]
and note that
\[
A\bk{w_{21}} f^0_{\lambda_{\frac{1}{2}}} = \frac{R_F}{\zfun_F\bk{3}} f^0_{\lambda_{\bk{-1,-1,1,1}}} .
\]
It follows that
\[
\Res_{s=\frac{1}{2}}\coset{\Eisen_{E}\bk{f^0,s,g}} 
= \frac{1}{2} \cdot \Eisen_{P_{\set{1,2}}}\bk{A\bk{w_{21}}f^0,\bar{0},g} .
\]

On the other hand, by \cite[Proposition 6]{MR2767521}, $\Eisen_{P_{1,2}} \bk{f ,\lambda_{\bk{s_1,s_2}},g}$ is holomorphic at $\lambda_{\bk{0,0}}$ for any section and hence the map $\Eisen_{P_{1,2}} \bk{\widetilde{f}^0,\bar{0},g}$ is $H_E\bk{\A}$-equivariant.

Since $f^0_{\lambda_{\frac{1}{2}}}$ generates $I\bk{\Id,\frac{1}{2}}$ we have the following result.
\begin{Prop}
	For any standard section $f_s\in I\bk{\Id,s}$ it holds that
	\[
	\Res_{s=\frac{1}{2}}\coset{\Eisen_{E}\bk{f,s,g}} 
	= \frac{1}{2} \cdot \Eisen_{P_{\set{1,2}}}\bk{A\bk{21}f,\bar{0},g} .
	\]
\end{Prop}

We note that $I_{P_{\set{1,2}}} \bk{\bar{0}}$ is unitary and hence semi-simple.
On the other hand, the maxim semi-simple quotient of $I\bk{\Id,\frac{1}{2}}$ is $\displaystyle\operatorname*{\otimes}_{\nu\in\Places}{} ' \pi_{1,\nu}$.

It follows that
\begin{equation}
Res\bk{\frac{1}{2},\Id,F\times K} = \operatorname*{\otimes}_{\nu\in\Places}{} ' \pi_{1,\nu} .
\end{equation}

\subsection{$\chi\neq\Id$}
Here we assume that $s_0=\frac{1}{2}$, $\chi=\chi_K$ and $E=F\times K$, where $K$ is a quadratic field extension of $F$.

We recall, from \cite{SegalEisen}, the intertwining operators involved in the simple pole of $\Eisen_E\bk{\chi_K, f_s, s, g}_{B_E}$ at $s=\frac{1}{2}$.

\begin{align*}
\Sigma^{P_E}_{\bk{F\times K,\chi_K,\frac{1}{2},2}} = \set{w_{2321}, w_{2132}, w_{21321}, w_{21323}, w_{213213}, w_{2132132}}
\end{align*}
and
\begin{align*}
& w_{2321}^{-1}\cdot\chi_{\frac{1}{2}}\bk{t} = w_{2132132}^{-1}\cdot\chi_{\frac{1}{2}}\bk{t} = \chi_K\bk{t_2} \frac{\FNorm{t_2}_F}{\FNorm{t_1}_F \FNorm{t_3}_K} \\
& w_{2132}^{-1}\cdot\chi_{\frac{1}{2}}\bk{t} = w_{21323}^{-1}\cdot\chi_{\frac{1}{2}}\bk{t} = \chi_K\bk{t_1} \frac{1}{\FNorm{t_2}_F} \\ 
& w_{21321}^{-1}\cdot\chi_{\frac{1}{2}}\bk{t} = w_{213213}^{-1}\cdot\chi_{\frac{1}{2}}\bk{t} = \chi_K\bk{t_1t_2} \frac{1}{\FNorm{t_2}_F} .
\end{align*}

\begin{align*}
& \Sigma^{P_E}_{\bk{F\times K,\chi_K,\frac{1}{2},1}}\setminus \Sigma^{P_E}_{\bk{F\times K,\chi_K,\frac{1}{2},2}} = \set{w_{23}, w_{232}, w_{213}} \\
& w_{23}^{-1}\cdot\chi_{\frac{1}{2}}\bk{t} = \chi_K\bk{t_1t_2} \frac{\FNorm{t_1}_F}{\FNorm{t_3}_K},\quad w_{232}^{-1}\cdot\chi_{\frac{1}{2}}\bk{t} = \chi_K\bk{t_2} \frac{\FNorm{t_1}_F}{\FNorm{t_3}_K},\quad w_{213}^{-1}\cdot\chi_{\frac{1}{2}}\bk{t} = \chi_K\bk{t_1} \frac{\FNorm{t_2}_F}{\FNorm{t_1}_F \FNorm{t_3}_K} .
\end{align*}

Given a standard section $f_s\in I_{P_E}\bk{s}$ it generates a finite dimensional $\mathbf{K}$-representation, where $\mathbf{K}$ is a fixed maximal compact subgroup of $H_E\bk{\A}$ as in \cite[Section I.1.1.]{MR1361168}.
We let $\mathfrak{F}$ denote the finite set of $\mathbf{K}$-types determining the finite-dimensional subspace of $\Ind_{B_E\bk{\A}\cap\mathbf{K}}^{\mathbf{K}} \bk{\chi_s \big\vert_{B_E\bk{\A}\cap\mathbf{K}}}$ generated by $f_s\big\vert_{\mathbf{K}}$.
Note that both $\Ind_{B_E\bk{\A}\cap\mathbf{K}}^{\mathbf{K}} \bk{\chi_s \big\vert_{B_E\bk{\A}\cap\mathbf{K}}}$ and $f=f_s\big\vert_{\mathbf{K}}$ are independent of $s$.
For $w\in W\bk{P_E,H_E}$, we let $M_{\mathfrak{F}}\bk{w,\chi_s}$ be the associated intertwining operator on $\Ind_{B_E\bk{\A}\cap\mathbf{K}}^{\mathbf{K}} \bk{\chi_s \big\vert_{B_E\bk{\A}\cap\mathbf{K}}}$.

We then may write
\begin{align*}
& \lim\limits_{s\to \frac{1}{2}} \bk{s-\frac{1}{2}} \Eisen_E\bk{\chi_K, f_s, s, g}_{\bfCT} \\
& = \bk{w_{23}^{-1}\cdot\chi_s}M_{\mathfrak{F}}\bk{w_{23},\chi_s} f \\
& + \bk{w_{213}^{-1}\cdot\chi_s}M_{\mathfrak{F}}\bk{w_{1},w_{23}^{-1}\cdot\chi_s} M_{\mathfrak{F}}\bk{w_{23},\chi_s} f \\
& + \bk{w_{232}^{-1}\cdot\chi_s}M_{\mathfrak{F}}\bk{w_{2},w_{23}^{-1}\cdot\chi_s} M_{\mathfrak{F}}\bk{w_{23},\chi_s} f \\
& + \coset{Id+\bk{\frac{w_{21323}^{-1}\cdot\chi_s}{w_{2132}^{-1}\cdot\chi_s}} M_{\mathfrak{F}}\bk{w_{3},w_{2132}^{-1}\cdot\chi_s}}
 \bk{w_{2132}^{-1}\cdot\chi_s} M_{\mathfrak{F}}\bk{w_{12},w_{23}^{-1}\cdot\chi_s} M_{\mathfrak{F}}\bk{w_{23},\chi_s} f \\
& + \coset{Id+\bk{\frac{w_{213213}^{-1}\cdot\chi_s}{w_{21321}^{-1}\cdot\chi_s}} M_{\mathfrak{F}}\bk{w_{3},w_{21321}^{-1}\cdot\chi_s}}
 \bk{w_{21321}^{-1}\cdot\chi_s} M_{\mathfrak{F}}\bk{w_{121},w_{23}^{-1}\cdot\chi_s} M_{\mathfrak{F}}\bk{w_{23},\chi_s} f \\
& + \coset{Id+\bk{\frac{w_{2321232}^{-1}\cdot\chi_s}{w_{2321}^{-1}\cdot\chi_s}} M_{\mathfrak{F}}\bk{w_{232},w_{2321}^{-1}\cdot\chi_s}}
 \bk{w_{2321}^{-1}\cdot\chi_s} M_{\mathfrak{F}}\bk{w_{21},w_{23}^{-1}\cdot\chi_s} M_{\mathfrak{F}}\bk{w_{23},\chi_s} f . \\
%\bk{\operatorname{Id}+w_1^{-1}\cdot\eta_s M_{\mathfrak{F}}\bk{w_1,\eta_s} } \bk{\operatorname{Id}+w_3^{-1}\cdot\eta_s M_{\mathfrak{F}}\bk{w_3,\eta_s} } \eta_s M_{\mathfrak{F}}\bk{w_{2132},\chi_s} f .
\end{align*}
In particular, $Res\bk{\frac{1}{2},\Id,F\times K}$ is a quotient of
\[
\coset{\lim\limits_{s\to\frac{1}{2}} \bk{s-\frac{1}{2}} M\bk{w_{23},\chi_s}}  \bk{I_{P_E}\bk{\chi_K,\frac{1}{2}}} .
\]

\begin{Lem}
The local representation $I_{P_E}\bk{\chi_{K_\nu},\frac{1}{2}}$ admits a unique irreducible quotient $\pi_{1,\nu}$ for any $\nu\in\Places$ and
\[
N_\nu\bk{w_{23},\chi_s}  \bk{I_{P_E}\bk{\chi_K,\frac{1}{2}}} = \pi_{1,\nu} .
\]
\end{Lem}
It follows that
\[
\coset{\lim\limits_{s\to\frac{1}{2}} \bk{s-\frac{1}{2}} M\bk{w_{23},\chi_s}}  \bk{I_{P_E}\bk{\chi_K,\frac{1}{2}}} = \operatorname*{\otimes}_{\nu\in\Places}{} ' \pi_{1,\nu}
\]
and hence
\begin{equation}
Res\bk{\frac{1}{2},\chi_K,F\times K} \cong \operatorname*{\otimes}_{\nu\in\Places}{} ' \pi_{1,\nu} .
\end{equation}

\begin{proof}[Proof of Lemma]
	First, assume that $K_\nu=F_\nu\times F_\nu$.
	In this case, the claim follows from the computation of $Res\bk{\frac{1}{2},\Id,F\times F\times F}$.
	
%	$\pi_{\epsilon,\nu}^{\bk{\chi_\nu,E_\nu,s_0}}$
	
	The fact that $Res\bk{\frac{1}{2},\Id,F\times F\times F}=\operatorname*{\otimes}_{\nu\in\Places}{} ' \pi_{1,\nu}^{\bk{\Id,F_\nu\times F_\nu\times F_\nu,\frac{1}{2}}}$ is irreducible and that the pole of $\Eisen_{F\times F\times F}\bk{\Id, f_s, s, g}$ at $s=\frac{1}{2}$ implies that for any $\coset{w}\in\Sigma^{P_E}_{\bk{F\times F\times F,\Id,\frac{1}{2},1}}\rmod\sim_{\Id,\frac{1}{2}}$ it holds that
	\[
	\coset{\lim\limits_{s\to\frac{1}{2}} \bk{s-\frac{1}{2}} \suml_{w\in\coset{w}} M\bk{w,\chi_s}} \bk{I\bk{\Id,\frac{1}{2}}} \subseteq \operatorname*{\otimes}_{\nu\in\Places}{} ' \pi_{1,\nu}^{\bk{\Id,F_\nu\times F_\nu\times F_\nu,\frac{1}{2}}} .
	\]
	
	Since $\set{w_{23}},\set{w_{24}}\in \Sigma^{P_E}_{\bk{F\times F\times F,\Id,\frac{1}{2},1}}\rmod\sim_{\Id,\frac{1}{2}}$, it follows that
	\begin{align*}
	& N_\nu\bk{\widetilde{w_{23}},\chi_s}  \bk{I_{P_E}\bk{\chi_K,\frac{1}{2}}} \subseteq \pi_{1,\nu}^{\bk{\Id,F_\nu\times F_\nu\times F_\nu,\frac{1}{2}}}, \\
	& N_\nu\bk{\widetilde{w_{24}},\chi_s}  \bk{I_{P_E}\bk{\chi_K,\frac{1}{2}}} \subseteq \pi_{1,\nu}^{\bk{\Id,F_\nu\times F_\nu\times F_\nu,\frac{1}{2}}}.
	\end{align*}
	Equalities follow since both intertwining operators sends a non-zero spherical vector to a non-zero spherical vector at $s=\frac{1}{2}$.
	
	\mbox{}
		
	Now, assume that $K_\nu$ is a field.
	If $F_\nu=\R$, the claim is proved in \Cref{Appendix:Archimedean}; we assume that $\nu\not\vert\infty$.
	
	The idea, in this case, is similar to that of \cite[Lem. 4.5]{SegalEisen}.
	As $N_\nu\bk{w_{2},\chi_s}$ is an isomorphism at $s=\frac{1}{2}$, it is enough to show that the unique irreducible subrepresentation of $I_\nu\bk{\chi_{K,\nu},\frac{1}{2}}$ is a subquotient of the kernel of $N_\nu\bk{w_{3},w_{2}^{-1}\cdot\chi_s}$ at $s=\frac{1}{2}$.
	
	A straight-forward computation of the Jacquet functor
	$\mathcal{J}_{T_E}^{H_E} \bk{\Ind_{B_E}^{H_E}\chi_\frac{1}{2}}$ shows that any character appearing in it will appear with multiplicity $2$; indeed, it follows immediately from the fact that $\FNorm{\Stab_W\bk{\chi_{\frac{1}{2},\nu}}}=2$.
	
	On the other hand, by Frobenius reciprocity, $\chi_{\frac{1}{2}}$ will appear in the Jacquet functor of the unique irreducible subrepresentation of $I_\nu\bk{\chi_{K,\nu},\frac{1}{2}}$.
	One then checks that the multiplicity of $\chi_{\frac{1}{2}}$ in the kernel of $N_\nu\bk{w_{3},w_{2}^{-1}\cdot\chi_s}$ is $2$ and hence the unique irreducible subrepresentation of $I_\nu\bk{\chi_{K,\nu},\frac{1}{2}}$ lies in the kernel of $N_\nu\bk{w_{3},w_{2}^{-1}\cdot\chi_s}$.
	The claim then follows.

\end{proof}

\subsection{Summary}
We record the results of this section as a theorem.
\begin{Thm}
	\label{Thm:Main_non_Sq_Int}
	The non-square integrable residues $Res\bk{s_0,\chi,E}$ with $\Real\bk{s_0}>0$ are given as follows:
	\begin{enumerate}
		\item If $s_0=\frac{3}{2}$, $\chi=\Id$ and $E=F\times K$, where $K$ is a field, the residue is given by
		\[
		Res\bk{\frac{3}{2},\Id,F\times K} \cong I_{P_{\set{2,3,4}}}\bk{1} .
		\]
		
		\item If $s_0=\frac{1}{2}$, $\chi=\Id$ and $E=F\times K$ then residue is given by
		\[
%		Res\bk{\frac{1}{2},\Id,F\times K} \cong \operatorname*{\otimes}_{\nu\in\Places}{} ' \pi_{1,\nu}^{\bk{\Id,F_\nu\times K_\nu,\frac{1}{2}}} .
		Res\bk{\frac{1}{2},\Id,F\times K} \cong \operatorname*{\otimes}_{\nu\in\Places}{} ' \pi_{1,\nu} .
		\]
		
		\item If $s_0=\frac{1}{2}$, $\chi=\chi_{K}$ and $E=F\times K$, where $K$ is a quadratic field extension of $F$, then residue is given by
		\[
%		Res\bk{\frac{1}{2},\Id,F\times K} \cong \operatorname*{\otimes}_{\nu\in\Places}{} ' \pi_{1,\nu}^{\bk{\chi_{K,\nu},F_\nu\times K_\nu,\frac{1}{2}}} .
		Res\bk{\frac{1}{2},\Id,F\times K} \cong \operatorname*{\otimes}_{\nu\in\Places}{} ' \pi_{1,\nu} .
		\]
	\end{enumerate}
\end{Thm}

%\vfill
%\pagebreak

%\bibliographystyle{abbrv}
%\bibliographystyle{amsplain}
\bibliographystyle{alpha}
\bibliography{bib}

\vfill
\pagebreak

\appendix
%\appendixpage

\section{Structure of Local Degenerate Principal Series at Archimedean Places}
\label{Appendix:Archimedean}
All the results in this section were obtained using the \textit{atlas of lie groups} \cite{ATLAS}.
We consdier the cases where $s_0=\frac{1}{2}$ since the cases where $s_0=\frac{3}{2}$ and $s_0=\frac{5}{2}$ were dealt with in \Cref{Sec_Local_Representations}.
We also wish to remind the reader that $\R$ has no cubic field extensions and $\C$ has no proper field extension and that the only finite-order characters of $\R^\times$ are $\Id$ and $\sgn$ while the only finite-order character of $\C^\times$ is $\Id$.
\begin{Prop}
\begin{enumerate}
\item Assume $F_\nu=\R$ and $E_\nu=\R\times \R\times \R$.
\begin{itemize}
\item $I_\nu\bk{\Id_\nu,\frac{1}{2}}$ has length $5$ with a unique irreducible quotient, which is spherical.
%\todonum{The proof her is not complete}{arg2}
%atals1.txt
\item $I_\nu\bk{\sgn_\nu,\frac{1}{2}}$ has length $5$ with a unique irreducible subrepresentation; the maximal semisimple quotient is a direct sum of four irreducible representations.
%atlas5.txt
\end{itemize}
\item Assume $F_\nu=\R$ and $E_\nu=\R\times \C$.
Both $I_\nu\bk{\Id_\nu,\frac{1}{2}}$ and $I_\nu\bk{\sgn_\nu,\frac{1}{2}}$ has length $2$ with a unique irreducible subrepresentation and a unique irreducible quotient.
%atlas2.txt and atlas5.txt
%\begin{itemize}
%\item $I_\nu\bk{\Id_\nu,\frac{1}{2}}$ has length $2$ with a unique irreducible subrepresentation and a unique irreducible quotient.
%%atals2.txt
%\item $I_\nu\bk{\sgn_\nu,\frac{1}{2}}$ has length $2$ with a unique irreducible subrepresentation and a unique irreducible quotient.
%%atlas6.txt
%\end{itemize}
\item Assume $F_\nu=\C$ and $E_\nu=\C\times \C\times \C$. $I_\nu\bk{\Id_\nu,\frac{1}{2}}$ is irreducible.
%atlas3.txt
\end{enumerate}
\end{Prop}
%\todonum{Check that we proved everything stated here and update \Cref{Sec_Local_Representations}}

In what follows, we supply the output from the ATLAS software regarding the above mentioned representations.
It should be noted before hand that ATLAS thinks of a irreducible representations of a real reductive Lie group $G$ in terms of the Langlands parametrization, i.e. triples $\bk{x,\lambda,\nu}$, called \emph{parameters}, where:
\begin{itemize}
\item $x$ is an element in $K\lmod G\rmod B$, where $K$ is a fixed maximal compact subgroup and $B$ is a fixed Borel subgroup.
The element $x$ gives rise to a Borel subgroup $B_x$, a Cartan subgroup $H_x$ and an involution $\theta_x$ of $G$.
\item $\lambda\in X^\ast\rmod\bk{1-\theta_x}X^\ast$, where $X^\ast$ is the lattice of algebraic characters of $G_x$.
This is a character on the maximal compact subgroup of $H_x$.
When $H_x$ is split, it can be identified with the group of connected components of $H_x$.
\item $\nu\in \bk{X^\ast}^{-\theta_x}$.
\item The parameter $p=\bk{x,\lambda,\nu}$ corresponds to the Langlands quotient of $J\bk{p}=\Ind_{B_x}^H\bk{\lambda\otimes\nu}$.
\item The infinitesimal character of $J\bk{p}$ is $\gamma=\frac{1+\theta_x}{2}\lambda+\frac{1-\theta_x}{2}\nu=\frac{1+\theta_x}{2}\lambda+\nu$.
\end{itemize}
For more details the reader may consult \cite{UnitaryRepresentationofRealReductiveGroups} (in particular Proposition 4.3 there), \cite{MR2485793} and \cite{MR2454331}.
The reader is also advised to consider the documentation in \url{http://www.liegroups.org/}.
The calculations bellow were performed on version 1.0.1 of ATLAS.

%cd ~/Dropbox/ATLAS/atlas_1.01/ ;./atlas --path=atlas-scripts

\subsection{$F_\nu=\R$, $E_\nu=\R\times \R\times \R$ $s=\frac{1}{2}$ and $\chi_\nu=\Id_\nu$}
%$\Ind_{B_E}^{H_E}\mu_\chi\otimes\eta_s \twoheadrightarrow I\bk{\chi,s}$
We start by defining the group $\mathtt{H}=H_E=Spin\bk{4,4}$ and the subgroups $\mathtt{B}=B_E$, $\mathtt{T}=T_E$, $\mathtt{P}=P_E$ and $\mathtt{M}=M_E$:
\begin{verbatim}
atlas> set H=Spin(4,4)
Variable H: RealForm
atlas> #KGB(H)
Value: 109
atlas> set x=KGB(H,108)
Variable x: KGBElt
atlas> set P=Parabolic :([0,2,3],x)
Variable P: ([int],KGBElt)
atlas> set M=Levi(P)
Variable M: RealForm
atlas> M
Value: disconnected split real group with Lie algebra 
      sl(2,R).sl(2,R).sl(2,R).gl(1,R)'
atlas> set B=Parabolic :([],x)
Variable B: ([int],KGBElt)
atlas> set T=Levi(B)
Variable T: RealForm
\end{verbatim}

Then, we consider $I_\nu\bk{\Id_\nu,\frac{1}{2}}$ as a quotient of $\Ind_{B_E}^{H_E}\eta_{\frac{1}{2}}$.
First, we define $\eta_{\frac{1}{2}}=\bk{1,-1,1,1}$ and consider the induction $\Ind_{B_E}^{M_E}\eta_{\frac{1}{2}}$ and then we pick up the unique irreducible quotient of $\Ind_{B_E}^{M_E}\eta_{\frac{1}{2}}$; this is the one-dimensional representation $\FNorm{{det}_{M_E}}^{\frac{1}{2}}$ of $M_E$.
\begin{verbatim}
atlas> set z=KGB(T,0)
Variable z: KGBElt
atlas> set u=vec:[1,-1,1,1]
Variable u: vec
atlas> set p=parameter(z,null(rank(H)),u)
Variable p: Param
atlas> real_induce_standard (p,M)
Value: final parameter (x=13,lambda=[2,-3,2,2]/2,nu=[1,-1,1,1]/1)
atlas> dimension(real_induce_standard (p,M))
Value: 1
atlas> set q=real_induce_standard (p,M)
Variable q: Param
\end{verbatim}

We then consider the induced representation $I_\nu\bk{\Id_\nu,\frac{1}{2}}$:
\begin{verbatim}
atlas> real_induce_irreducible(q,H)
Value: 
1*final parameter (x=16,lambda=[0,1,0,0]/1,nu=[-1,2,-1,-1]/2)
1*final parameter (x=17,lambda=[0,1,0,0]/1,nu=[-1,2,-1,-1]/2)
1*final parameter (x=18,lambda=[0,1,0,0]/1,nu=[-1,2,-1,-1]/2)
1*final parameter (x=19,lambda=[0,1,0,0]/1,nu=[-1,2,-1,-1]/2)
1*final parameter (x=108,lambda=[1,1,1,1]/1,nu=[0,1,0,0]/1)
\end{verbatim}
Indeed we see that it has length $5$ and we turn to show that the last parameter represents the unique irreducible quotient of $I_\nu\bk{\Id_\nu,\frac{1}{2}}$. One can further show that the direct sum of the other representations is the maximal semi-simple subrepresentation of $I_\nu\bk{\Id_\nu,\frac{1}{2}}$.

By \Cref{Thm:local_results}, $\Ind_B^H \lambda_{\bk{1,-1,1,1}}$ admits two irreducible quotients.
Since $I_\nu\bk{\Id_\nu,\frac{1}{2}}$ is a quotient of $\Ind_B^H \lambda_{\bk{1,-1,1,1}}$, any irreducible quotient of $I_\nu\bk{\Id_\nu,\frac{1}{2}}$ is an irreducible quotient of $\Ind_B^H \lambda_{\bk{1,-1,1,1}}$.
We now compute the parameters of the two irreducible quotients of $\Ind_B^H \lambda_{\bk{1,-1,1,1}}$ and see that only one appears in the Jordan-H\"older series of $I_\nu\bk{\Id_\nu,\frac{1}{2}}$.
\begin{verbatim}
atlas> set Q=Parabolic :([1],x)
Variable Q: ([int],KGBElt)
atlas> set L=Levi(Q)
Variable L: RealForm
atlas> void: for q in monomials(real_induce_irreducible(p,L)) 
      do prints(q," ",is_finite_dimensional (q)) od
final parameter (x=0,lambda=[-1,2,-1,-1]/2,nu=[1,0,1,1]/2) false
final parameter (x=1,lambda=[-1,2,-1,-1]/2,nu=[0,1,0,0]/1) true
atlas> set par0=monomials(real_induce_irreducible(p,L))
Variable par0: [Param]
atlas> real_induce_standard (par0[0],H)
Value: non-normal parameter (x=106,lambda=[1,1,1,1]/1,nu=[1,0,1,1]/2)
atlas> finalize($)
Value: [final parameter (x=65,lambda=[-1,4,-1,-1]/1,nu=[-1,3,-1,-1]/2)]
atlas> real_induce_standard (par0[1],H)
Value: final parameter (x=108,lambda=[1,1,1,1]/1,nu=[0,1,0,0]/1)
\end{verbatim}
Namely, $\pi_{-2}$ is not a constituent of $I_\nu\bk{\Id_\nu,\frac{1}{2}}$.
It follows that $\pi_1$ is the unique irreducible quotient of $I_\nu\bk{\Id_\nu,\frac{1}{2}}$.

\subsection{$F_\nu=\R$, $E_\nu=\R\times \R\times \R$ $s=\frac{1}{2}$ and $\chi_\nu=\sgn_\nu$}

%We start by defining the group $H=Spin\bk{4,4}$ and the subgroups $B_E$, $T_E$, $P_E$ and $M_E$:
%\begin{verbatim}
%atlas> set H=Spin(4,4)
%Variable H: RealForm 
%atlas> #KGB(H)
%Value: 109
%atlas> set x=KGB(H,108)
%Variable x: KGBElt
%atlas> set P=Parabolic :([0,2,3],x)
%Variable P: ([int],KGBElt)
%atlas> set M=Levi(P)
%Variable M: RealForm
%atlas> set B=Parabolic :([],x)
%Variable B: ([int],KGBElt)
%atlas> set T=Levi(B)
%Variable T: RealForm
%\end{verbatim}

%This case follows completely from \Cref{Sec_Local_Representations}.
%However, for the benefit of the reader, we demonstrate the calculation performed by ATLAS showing the result.
This case was proven in \Cref{Sec_Local_Representations}.
For the benefit of the reader we include the relevant calculation in ATLAS.

We set $\mathtt{H}$, $\mathtt{B}$, $\mathtt{T}$, $\mathtt{P}$ and $\mathtt{M}$ as in the previous case.
Then, we consider $I_\nu\bk{\sgn,\frac{1}{2}}$ as a quotient of $\Ind_{B_E}^{H_E}\mu_{\sgn}\otimes\eta_{\frac{1}{2}}$.
First, we define $\eta_{\frac{1}{2}}=\bk{1,-1,1,1}$ and consider the induction $\Ind_{B_E}^{M_E}\eta_{\frac{1}{2}}$ and then we pick up the unique irreducible quotient of $\Ind_{B_E}^{M_E}\eta_{\frac{1}{2}}$; this is the one-dimensional representation $\FNorm{{det}_{M_E}}^{\frac{1}{2}}$ of $M_E$.

\begin{verbatim}
atlas> set z=KGB(T,0)
Variable z: KGBElt
atlas> set u=vec:[1,-1,1,1]
Variable u: vec
atlas> set ud=dominant(H,u)
Variable ud: vec
atlas> ud
Value: [ 0, 1, 0, 0 ]
atlas> set psgn=parameter(z,ud,u)
Variable psgn: Param
atlas> set par0=monomials(real_induce_irreducible(psgn,M))
Variable par0: [Param]
atlas>void:for q in par0 do prints(q," ",is_finite_dimensional (q)) od
final parameter (x=0,lambda=[2,-3,2,2]/2,nu=[0,1,0,0]/2) false
final parameter (x=1,lambda=[2,-3,2,2]/2,nu=[0,1,0,0]/2) false
final parameter (x=2,lambda=[2,-3,2,2]/2,nu=[0,1,0,0]/2) false
final parameter (x=3,lambda=[2,-3,2,2]/2,nu=[0,1,0,0]/2) false
final parameter (x=4,lambda=[2,-3,2,2]/2,nu=[0,0,0,1]/1) false
final parameter (x=5,lambda=[2,-3,2,2]/2,nu=[0,0,0,1]/1) false
final parameter (x=6,lambda=[2,-3,2,2]/2,nu=[0,0,1,0]/1) false
final parameter (x=7,lambda=[2,-3,2,2]/2,nu=[0,0,1,0]/1) false
final parameter (x=8,lambda=[2,-3,2,2]/2,nu=[1,0,0,0]/1) false
final parameter (x=9,lambda=[2,-3,2,2]/2,nu=[1,0,0,0]/1) false
final parameter (x=10,lambda=[2,-3,2,2]/2,nu=[0,-1,2,2]/2) false
final parameter (x=11,lambda=[2,-3,2,2]/2,nu=[2,-1,0,2]/2) false
final parameter (x=12,lambda=[2,-3,2,2]/2,nu=[2,-1,2,0]/2) false
final parameter (x=13,lambda=[2,-1,2,2]/2,nu=[1,-1,1,1]/1) true
atlas> set q=par0[13]
Variable q: Param
\end{verbatim}
Another way to create the parameter $q$ is as follows:
\begin{verbatim}
atlas> set q=real_induce_standard(psgn,M)
Variable q: Param
\end{verbatim}

We then consider the composition series of the induced representation $I_\nu\bk{\sgn,\frac{1}{2}}$:
\begin{verbatim}
atlas> real_induce_irreducible(q,H)
Value: 
1*final parameter (x=65,lambda=[-1,4,-1,-1]/1,nu=[-1,3,-1,-1]/2)
1*final parameter (x=91,lambda=[0,3,0,0]/1,nu=[0,1,0,0]/1)
1*final parameter (x=92,lambda=[0,3,0,0]/1,nu=[0,1,0,0]/1)
1*final parameter (x=93,lambda=[0,3,0,0]/1,nu=[0,1,0,0]/1)
1*final parameter (x=94,lambda=[0,3,0,0]/1,nu=[0,1,0,0]/1)
\end{verbatim}
Indeed, we see that it has length $5$ and we turn to show that the last four parameters represent quotients of $I_\nu\bk{\sgn,\frac{1}{2}}$ while the first one does not.

We also recall, from \Cref{Sec_Local_Representations}, that $\Ind_{B_E}^{H_E}\mu_{\sgn}\otimes\eta_{\frac{1}{2}}$ admits a maximal semi-simple quotient of length $4$.
We compute this quotient and show that it is the quotient of $I_\nu\bk{\sgn,\frac{1}{2}}$.

We first decompose $\Ind_{B_E\bk{\R}\cap M_E\bk{\R}}^{M_E\bk{\R}} \widetilde{\chi}$
\begin{verbatim}
atlas> set psgnd=parameter(z,u,ud)
Variable psgnd: Param
atlas> real_induce_irreducible(psgnd,M)
Value: 
1*final parameter (x=0,lambda=[0,1,0,0]/2,nu=[0,1,0,0]/1)
1*final parameter (x=1,lambda=[0,1,0,0]/2,nu=[0,1,0,0]/1)
1*final parameter (x=2,lambda=[0,1,0,0]/2,nu=[0,1,0,0]/1)
1*final parameter (x=3,lambda=[0,1,0,0]/2,nu=[0,1,0,0]/1)
\end{verbatim}

For each of the irreducible constituents of $\Ind_{B_E\bk{\R}\cap M_E\bk{\R}}^{M_E\bk{\R}} \widetilde{\chi}$ we compute its unique irreducible quotient.
\begin{verbatim}
atlas> void:for sigma_i in monomials(real_induce_irreducible(psgnd,M)) 
      do prints(real_induce_standard(sigma_i,H)) od
final parameter (x=91,lambda=[0,3,0,0]/1,nu=[0,1,0,0]/1)
final parameter (x=94,lambda=[0,3,0,0]/1,nu=[0,1,0,0]/1)
final parameter (x=93,lambda=[0,3,0,0]/1,nu=[0,1,0,0]/1)
final parameter (x=92,lambda=[0,3,0,0]/1,nu=[0,1,0,0]/1)
\end{verbatim}
It remains to show that the multiplicity of each one of this quotients in $\Ind_{B_E}^{H_E}\mu_{\sgn}\otimes\eta_{\frac{1}{2}}$ is $1$.
\begin{verbatim}
atlas> real_induce_irreducible(psgn,H)
Value: 
1*final parameter (x=0,lambda=[0,1,0,0]/1,nu=[0,0,0,0]/1)
3*final parameter (x=1,lambda=[0,1,0,0]/1,nu=[0,0,0,0]/1)
3*final parameter (x=3,lambda=[0,1,0,0]/1,nu=[0,0,0,0]/1)
3*final parameter (x=4,lambda=[0,1,0,0]/1,nu=[0,0,0,0]/1)
1*final parameter (x=5,lambda=[0,1,0,0]/1,nu=[0,0,0,0]/1)
1*final parameter (x=6,lambda=[0,1,0,0]/1,nu=[0,0,0,0]/1)
1*final parameter (x=7,lambda=[0,1,0,0]/1,nu=[0,0,0,0]/1)
3*final parameter (x=9,lambda=[0,1,0,0]/1,nu=[0,0,0,0]/1)
1*final parameter (x=16,lambda=[0,1,0,0]/1,nu=[-1,2,-1,-1]/2)
1*final parameter (x=17,lambda=[0,1,0,0]/1,nu=[-1,2,-1,-1]/2)
1*final parameter (x=18,lambda=[0,1,0,0]/1,nu=[-1,2,-1,-1]/2)
1*final parameter (x=19,lambda=[0,1,0,0]/1,nu=[-1,2,-1,-1]/2)
2*final parameter (x=47,lambda=[0,3,-2,0]/1,nu=[0,1,-1,0]/1)
2*final parameter (x=48,lambda=[0,3,-2,0]/1,nu=[0,1,-1,0]/1)
2*final parameter (x=49,lambda=[-1,2,-1,1]/1,nu=[-1,1,0,0]/1)
2*final parameter (x=50,lambda=[-1,2,-1,1]/1,nu=[-1,1,0,0]/1)
2*final parameter (x=51,lambda=[0,3,0,-2]/1,nu=[0,1,0,-1]/1)
2*final parameter (x=52,lambda=[0,3,0,-2]/1,nu=[0,1,0,-1]/1)
4*final parameter (x=65,lambda=[-1,4,-1,-1]/1,nu=[-1,3,-1,-1]/2)
1*final parameter (x=91,lambda=[0,3,0,0]/1,nu=[0,1,0,0]/1)
1*final parameter (x=92,lambda=[0,3,0,0]/1,nu=[0,1,0,0]/1)
1*final parameter (x=93,lambda=[0,3,0,0]/1,nu=[0,1,0,0]/1)
1*final parameter (x=94,lambda=[0,3,0,0]/1,nu=[0,1,0,0]/1)
\end{verbatim}

\begin{Remark}
%	\todonum{Rephrase this remark.}
	We wish to remark on the output of ATLAS in case that induction from the parameter $p$, entered to \emph{real\_induce\_standard}, is not a standard module.
	
	\begin{itemize}
		\item If there exists a parameter $p'$ conjugate to $p$ such that induction from $p'$ is a standard module, then ATLAS would compute the unique irreducible quotient of that induction.
		
		\item If there is no such parameter conjugate to $p$, one can still find a "weakly-dominant" parameter $p'$ conjugate to $p$. In which case, ATLAS would return a parameter for the maximal semi-simple quotient of the induction from $p'$.
		It is then possible to compute the parameters for the irreducible quotients of the induction from $p'$.
		
		This is the case of the parameter $\bk{x=0,lambda=[0,1,0,0]/1,nu=[1,-1,1,1]/1}$.
		\begin{verbatim}
		atlas> real_induce_standard(psgn,H)
		Value: non-dominant parameter (x=108,lambda=[1,2,1,1]/1,nu=[1,-1,1,1]/1)
		\end{verbatim}
		%The fact that the parameter $\bk{x=108,lambda=[1,2,1,1]/1,nu=[1,-1,1,1]/1}$ is non-dominant means that it is not necessarily irreducible and should be finalized in order to get its constituents.
		As we already know, this induction has a maximal semi-simple quotient of length 4.
		In order to find the parameters of the irreducible quotients in this induction we need to finalize the parameter as follows.
		\begin{verbatim}
		atlas> set standard_quotient=finalize(real_induce_standard (psgn,H))
		Variable standard_quotient: [Param]
		atlas> void:for t in standard_quotient do prints(t) od
		final parameter (x=91,lambda=[0,3,0,0]/1,nu=[0,1,0,0]/1)
		final parameter (x=92,lambda=[0,3,0,0]/1,nu=[0,1,0,0]/1)
		final parameter (x=93,lambda=[0,3,0,0]/1,nu=[0,1,0,0]/1)
		final parameter (x=94,lambda=[0,3,0,0]/1,nu=[0,1,0,0]/1)
		\end{verbatim}
	\end{itemize}
\end{Remark}

%\noindent\rule{\textwidth}{2pt}
%\begin{verbatim}
%atlas> real_induce_standard (psgnd,H)
%Value: non-final parameter (x=108,lambda=[2,2,2,2]/1,nu=[0,1,0,0]/1)
%atlas> set standard_quotient=finalize(real_induce_standard (psgnd,H))
%Variable standard_quotient: [Param]
%atlas> void:for t in standard_quotient do prints(t) od
%final parameter (x=91,lambda=[0,3,0,0]/1,nu=[0,1,0,0]/1)
%final parameter (x=92,lambda=[0,3,0,0]/1,nu=[0,1,0,0]/1)
%final parameter (x=93,lambda=[0,3,0,0]/1,nu=[0,1,0,0]/1)
%final parameter (x=94,lambda=[0,3,0,0]/1,nu=[0,1,0,0]/1)
%\end{verbatim}
%The above means that the standard quotient of $\Ind_{B_E}^{H_E}w_2\cdot\bk{\mu_{\sgn}\otimes\eta_{\frac{1}{2}}}$ is not irreducible but, in fact, is the direct sum of these four irreducible representations.
%
%\noindent\rule{\textwidth}{2pt}
%\begin{verbatim}
%atlas> set q=monomials(real_induce_irreducible(p,M))[13]
%Variable q: Param
%atlas> is_finite_dimensional (q)
%Value: true
%atlas> rho(M)
%Value: [  2, -3,  2,  2 ]/2
%atlas> q
%Value: final parameter (x=13,lambda=[2,-1,2,2]/2,nu=[1,-1,1,1]/1)
%atlas> real_induce_standard (q,H)
%Value: non-dominant parameter (x=108,lambda=[1,2,1,1]/1,nu=[1,-1,1,1]/1)
%atlas> finalize($)
%Value: [final parameter (x=91,lambda=[0,3,0,0]/1,nu=[0,1,0,0]/1),final parameter (x=92,lambda=[0,3,0,0]/1,nu=[0,1,0,0]/1),final parameter (x=93,lambda=[0,3,0,0]/1,nu=[0,1,0,0]/1),final parameter (x=94,lambda=[0,3,0,0]/1,nu=[0,1,0,0]/1)]
%\end{verbatim}

\subsection{$F_\nu=\R$, $E_\nu=\R\times \C$ $s=\frac{1}{2}$ and $\chi_\nu=\Id_\nu$}
%atlas2.txt
%has length two with UIQ
We start by defining the group $\mathtt{H}=H_E=Spin\bk{5,3}$ and the subgroups $\mathtt{B}=B_E$, $\mathtt{T}=T_E$, $\mathtt{P}=P_E$ and $\mathtt{M}=M_E$.
\begin{verbatim}
atlas> set H=Spin(5,3)
Variable H: RealForm
atlas> #KGB(H)
Value: 40
atlas> set x=KGB(H,39)
Variable x: KGBElt
atlas> set P=Parabolic :([0,2,3],x)
Variable P: ([int],KGBElt)
atlas> set M=Levi(P)
Variable M: RealForm
atlas> set B=Parabolic :([],x)
Variable B: ([int],KGBElt)
atlas> set T=Levi(B)
Variable T: RealForm
\end{verbatim}
Then, we consider $I_\nu\bk{\Id_\nu,\frac{1}{2}}$ as a quotient of $\Ind_{B_E}^{H_E}\eta_{\frac{1}{2}}$.
First, we define $\eta_{\frac{1}{2}}=\bk{1,-1,1}$ and consider the induction $\Ind_{B_E}^{M_E}\eta_{\frac{1}{2}}$ and then we pick up the unique irreducible quotient of $\Ind_{B_E}^{M_E}\eta_{\frac{1}{2}}$; this is the one-dimensional representation $\FNorm{{det}_{M_E}}^{\frac{1}{2}}$ of $M_E$.
\begin{verbatim}
atlas> set z=KGB(T,0)
Variable z: KGBElt
atlas> set u=vec:[1,-1,1,1]
Variable u: vec
atlas> set p=parameter(z,null(rank(H)),u)
Variable p: Param
atlas> set par0=monomials(real_induce_irreducible(p,M))
Variable par0: [Param]
atlas> void:for q in par0 do prints(q," ",is_finite_dimensional (q)) od
final parameter (x=0,lambda=[2,-3,2,2]/2,nu=[0,1,0,0]/2) false
final parameter (x=1,lambda=[2,-3,2,2]/2,nu=[1,0,0,0]/1) false
final parameter (x=2,lambda=[2,-3,2,2]/2,nu=[0,-1,2,2]/2) false
final parameter (x=3,lambda=[2,-3,2,2]/2,nu=[1,-1,1,1]/1) true
atlas> set q=par0[3]
Variable q: Param
\end{verbatim}
We then consider the induced representation $I_\nu\bk{\Id_\nu,\frac{1}{2}}$:
\begin{verbatim}
atlas> real_induce_irreducible(q,H)
Value: 
1*final parameter (x=16,lambda=[0,3,-1,-1]/1,nu=[-1,3,-1,-1]/2)
1*final parameter (x=36,lambda=[1,2,1,-1]/1,nu=[0,1,0,0]/1)
\end{verbatim}

As explained in \Cref{Sec_Local_Representations}, $\Ind_B^H \lambda_{\bk{1,-1,1}}$ admits two irreducible quotients, $\pi_{1,\nu}$ and $\pi_{-2,\nu}$.
Since $I_\nu\bk{\Id_\nu,\frac{1}{2}}$ is a quotient of $\Ind_B^H \lambda_{\bk{1,-1,1}}$, any irreducible quotient of $I_\nu\bk{\Id_\nu,\frac{1}{2}}$ is an irreducible quotient of $\Ind_B^H \lambda_{\bk{1,-1,1}}$.

We now compute the parameters of $\pi_{1,\nu}$ and $\pi_{-2,\nu}$.
It turns out, as opposed to the split case $E_\nu=\R\times \R\times\R$, that both are constituents of $I_\nu\bk{\Id_\nu,\frac{1}{2}}$.
We then show that, in fact, the multiplicity of $\pi_{-2,\nu}$ in $\Ind_B^H \lambda_{\bk{1,-1,1}}$ is $2$.
It turns out that the $\pi_{-2,\nu}$ constituent of $I_\nu\bk{\Id_\nu,\frac{1}{2}}$ is not a quotient.
Namely, the exact sequence
\begin{equation}
\label{eq:Exact_sequence_E=RxC_1/2_trivial_character}
0 \to \pi_{-2,\nu} \to I_\nu\bk{\Id_\nu,\frac{1}{2}} \to \pi_{1,\nu} \to\ 0
\end{equation}
does not split.

First, we construct the parameters of $\pi_{1,\nu}$ and $\pi_{-2,\nu}$:
\begin{verbatim}
atlas> set Q=Parabolic:([1],x)
Variable Q: ([int],KGBElt)
atlas> set L=Levi(Q)
Variable L: RealForm

atlas> real_induce_standard (par0[0],H)
Value: non-normal parameter (x=38,lambda=[1,1,1,1]/1,nu=[1,0,1,1]/2)
atlas> finalize($)
Value: [final parameter (x=16,lambda=[0,3,-1,-1]/1,nu=[-1,3,-1,-1]/2)]

atlas> real_induce_standard (par0[1],H)
Value: non-normal parameter (x=39,lambda=[1,1,1,1]/1,nu=[0,1,0,0]/1)
atlas> finalize($)
Value: [final parameter (x=36,lambda=[1,2,1,-1]/1,nu=[0,1,0,0]/1)]
\end{verbatim}

And now, we show that the multiplicity of $\pi_{-2,\nu}$ in $\Ind_B^H \lambda_{\bk{1,-1,1}}$ is $2$, and hence it does not follow that $i_\nu\bk{\Id,\frac{1}{2}}$ is semi simple; in fact we will show that it is not.
\begin{verbatim}
atlas> real_induce_irreducible(p,H)
Value: 
1*final parameter (x=0,lambda=[0,1,0,0]/1,nu=[0,0,0,0]/1)
1*final parameter (x=1,lambda=[0,1,0,0]/1,nu=[0,0,0,0]/1)
2*final parameter (x=3,lambda=[0,1,0,0]/1,nu=[-1,2,-1,-1]/2)
2*final parameter (x=12,lambda=[-1,2,0,0]/1,nu=[-1,1,0,0]/1)
1*final parameter (x=14,lambda=[-1,2,0,0]/1,nu=[-1,1,0,0]/1)
2*final parameter (x=16,lambda=[0,3,-1,-1]/1,nu=[-1,3,-1,-1]/2)
1*final parameter (x=36,lambda=[1,2,1,-1]/1,nu=[0,1,0,0]/1)
\end{verbatim}
%\todonum{It seems like $I_\nu\bk{\Id_\nu,\frac{1}{2}}$ is in fact a direct sum of its two constituents}
%\todonum{These are not exactly the correct points to take these inductions to $L$, bu plugging in $\lambda_{\bk{-1,1,-1-1}}$ yields the same results}

%\vertline

%\todonum{Finish the following argument}

The idea of showing that the sequence in \Cref{eq:Exact_sequence_E=RxC_1/2_trivial_character} does not split is similar to the ideas in \cite{MR1329899}.
The sequence splits if the subrepresentation $\pi_{-2,\nu}$ of $I_\nu\bk{\Id_\nu,\frac{1}{2}}$ is also a quotient.
In order to show that it is not a quotient, it is enough to show that, given a $K$-types $\rho_1$ and $\rho_{-2}$ of $\pi_{1,\nu}$ and $\pi_{-2,\nu}$, there is a non-zero element in the universal enveloping algebra $\bk{\mathfrak{H}_\C}$ of the complexified Lie algebra $\mathfrak{H}_\C=Lie\bk{H}\otimes \C$, sending vectors from $\rho_1$ to $\rho_{-2}$.
We will show the existence of such an operator.

The maximal compact subgroup of $H_E\bk{F_\nu}$ is $K=Spin\bk{5}\times SU\bk{2}\rmod \mu_2 \equiv Sp\bk{2}\times SU\bk{2}\rmod\mu_2$.
Any finite dimensional irreducible representation of $K$ is of the form $V_{\bk{x,y}}\boxtimes V_z$, where:
\begin{itemize}
	\item $V_{\bk{x,y}}$ is an irreducible representation of $Sp\bk{2}$ with highest weight $\bk{x,y}$; in particular $x,y\in\N$ and $x\geq y\geq 0$.
	\item $V_z$ is an irreducible representation of $SU\bk{2}$ with highest weight $z$; in particular $z\in\N$ and the dimension of $V_z$ is $z+1$.
	\item $x+y+z$ is even.
	%page 14
\end{itemize}

We recall \cite[Lemma 6.6]{AdamsGanPaulSavin-D4Real}.
\begin{Lem}
	The type $V_{\bk{0,0}}\boxtimes V_n$ with $n>0$ appears in $\pi_1$ for $n$ odd and in $\pi_{-2}$ for $n$ even, with multiplicity $1$. 
\end{Lem}

We also note that the trivial representation of $K$, $V_{\bk{0,0}}\boxtimes V_0$, appears with multiplicity $1$ in $\pi_1$ as $\pi_1$ is spherical and $\pi_{-2}$ is not.

Fix a highest weight vector $v_n$ in $V_{\bk{0,0}}\boxtimes V_n$ and let $X_{+}$ denote a raising operator in $\mathfrak{SU}\bk{2}$.
For $s\in\C$ we consider the $\bk{\mathfrak{H},K}$-module associated to it.
The operator $X_{+}^2$ will send $v_0$ to $f\bk{s}v_4$, where $f(s)$ is a quadratic polynomial in $s$.

Since $\pi_{1}$ is a subrepresentation of $I_\nu\bk{\Id,-\frac{1}{2}}$ it follows that $f\bk{-\frac{1}{2}}=0$.
Also, since the unique irreducible representation of $I_\nu\bk{\Id,-\frac{5}{2}}$ is trivial, it also follows that $f\bk{-\frac{5}{2}}=0$.
We conclude that $f\bk{\frac{1}{2}}\neq 0$ from which the claim follows.

\subsection{$F_\nu=\R$, $E_\nu=\R\times \C$ $s=\frac{1}{2}$ and $\chi_\nu=\sgn_\nu$}
%atlas5.txt
%has length two with UIQ
This case is proved in \Cref{Sec_Local_Representations}.
However, for the benefit of the reader, we demonstrate how to find the irreducible constituents of $I_\nu\bk{\sgn,\frac{1}{2}}$ and the parameter of its unique irreducible quotient using ATLAS.
We further perform a simple calculation which is used in \Cref{Sec:NonSquareIntegrableResidues}.

We set $\mathtt{H}$, $\mathtt{B}$, $\mathtt{T}$, $\mathtt{P}$ and $\mathtt{M}$ as in the previous case.
We then compute $I_\nu\bk{\sgn,\frac{1}{2}}$ as a quotient of $\Ind_{B_E}^{H_E}\mu_{\sgn}\otimes\eta_{\frac{1}{2}}$.
We define $\eta_{\frac{1}{2}}=\bk{1,-1,1}$ and consider the induction $\Ind_{B_E}^{M_E}\mu_{\sgn}\otimes\eta_{\frac{1}{2}}$ and then we pick up the unique irreducible quotient of $\Ind_{B_E}^{M_E}\mu_{\sgn}\otimes\eta_{\frac{1}{2}}$; this is the one-dimensional representation $\bk{\chi\circ{det}_{M_E}}\otimes\FNorm{{det}_{M_E}}^{\frac{1}{2}}$ of $M_E$.
\begin{verbatim}
atlas> set z=KGB(T,0)
Variable z: KGBElt
atlas> set u=vec:[1,-1,1,1]
Variable u: vec
atlas> set usgn=vec:[0,1,0,0]
Variable ud: vec
atlas> set psgn=parameter(z,usgn,u)
Variable psgn: Param
atlas> set par0=monomials(real_induce_irreducible(psgn,M))
Variable par0: [Param]
atlas> void:for q in par0 do prints(q," ",is_finite_dimensional (q)) od
final parameter (x=0,lambda=[2,-3,2,2]/2,nu=[0,1,0,0]/2) false
final parameter (x=1,lambda=[2,-3,4,0]/2,nu=[1,0,0,0]/1) false
final parameter (x=2,lambda=[2,-3,2,2]/2,nu=[0,-1,2,2]/2) false
final parameter (x=3,lambda=[2,-1,0,0]/2,nu=[1,-1,1,1]/1) true
atlas> set q=par0[3]
Variable q: Param
\end{verbatim}
We then consider the induced representation $I_\nu\bk{\sgn,\frac{1}{2}}$:
\begin{verbatim}
atlas> real_induce_irreducible(q,H)
Value: 
1*final parameter (x=14,lambda=[-1,2,0,0]/1,nu=[-1,1,0,0]/1)
1*final parameter (x=32,lambda=[0,3,0,0]/1,nu=[0,1,0,0]/1)
\end{verbatim}
We check that the last parameter is the unique irreducible quotient of $\Ind_{B_E}^{H_E}\mu_{\sgn}\otimes\eta_{\frac{1}{2}}$
(and hence of $I_\nu\bk{\sgn,\frac{1}{2}}$):
%\begin{verbatim}
%atlas> real_induce_standard (q,H)
%Value: non-dominant parameter (x=39,lambda=[1,2,0,0]/1,nu=[1,-1,1,1]/1)
%atlas> finalize($)
%Value: [final parameter (x=32,lambda=[0,3,0,0]/1,nu=[0,1,0,0]/1)]
%\end{verbatim}
\begin{verbatim}
atlas> real_induce_standard(psgn ,H)
Value: non-dominant parameter (x=39,lambda=[1,2,0,0]/1,nu=[1,-1,1,1]/1)
atlas> finalize($)
Value: [final parameter (x=32,lambda=[0,3,0,0]/1,nu=[0,1,0,0]/1)]
\end{verbatim}

Finally, we compute the irreducible constituents of $\Ind_{B_E}^{H_E}\mu_{\sgn}\otimes\eta_{\frac{1}{2}}$.
\begin{verbatim}
atlas> real_induce_irreducible(psgn ,H)
Value: 
1*parameter(x=0,lambda=[0,1,0,0]/1,nu=[0,0,0,0]/1)
1*parameter(x=1,lambda=[0,1,0,0]/1,nu=[0,0,0,0]/1)
2*parameter(x=3,lambda=[0,1,0,0]/1,nu=[-1,2,-1,-1]/2)
1*parameter(x=12,lambda=[-1,2,0,0]/1,nu=[-1,1,0,0]/1)
1*parameter(x=14,lambda=[-1,2,0,0]/1,nu=[-1,1,0,0]/1)
1*parameter(x=24,lambda=[-1,1,1,1]/1,nu=[-1,1,0,0]/1)
1*parameter(x=16,lambda=[0,3,-1,-1]/1,nu=[-1,3,-1,-1]/2)
1*parameter(x=32,lambda=[0,3,0,0]/1,nu=[0,1,0,0]/1)
\end{verbatim}

In \Cref{Sec:NonSquareIntegrableResidues} we use the fact that the unique irreducible quotient of $I_\nu\bk{\sgn,\frac{1}{2}}$ is the image of $N_\nu\bk{w_{23},\chi_s}$.
In order to prove this, it is enough to show that $N_\nu\bk{w_{2},\chi_s}$ acts as an isomorphism on $I_\nu\bk{\sgn,\frac{1}{2}}$
and that the irreducible subrepresentation of $I_\nu\bk{\sgn,\frac{1}{2}}$
lies in the kernel of $N_\nu\bk{w_{3},w_2^{-1}\cdot\chi_s}$.
Indeed, $N_\nu\bk{w_{2},\chi_s}$ is an isomorphism:
\begin{verbatim}
atlas> set P2=Parabolic :([1],x)
Variable P2: ([int],KGBElt)
atlas> set M2=Levi(P2)
Variable M2: RealForm
atlas> set u=[-1,2,-1,-1]
Variable u: [int]
atlas> set sgn=vec:[0,1,0,0]
Variable sgn: vec
atlas> set p=parameter(z,sgn,u)
Variable p: Param
atlas> set q=real_induce_standard (p,M2)
Variable q: Param
atlas> real_induce_irreducible(q,H)
Value: 
2*final parameter (x=3,lambda=[0,1,0,0]/1,nu=[-1,2,-1,-1]/2)
1*final parameter (x=12,lambda=[-1,2,0,0]/1,nu=[-1,1,0,0]/1)
1*final parameter (x=14,lambda=[-1,2,0,0]/1,nu=[-1,1,0,0]/1)
1*final parameter (x=16,lambda=[0,3,-1,-1]/1,nu=[-1,3,-1,-1]/2)
1*final parameter (x=24,lambda=[-1,1,1,1]/1,nu=[-1,1,0,0]/1)
1*final parameter (x=32,lambda=[0,3,0,0]/1,nu=[0,1,0,0]/1)
\end{verbatim}
and the irreducible subrepresentation of $I_\nu\bk{\sgn,\frac{1}{2}}$
lies in the kernel of $N_\nu\bk{w_{3},w_2^{-1}\cdot\chi_s}$:
\begin{verbatim}
atlas> set P3=Parabolic :([2,3],x)
Variable P3: ([int],KGBElt)
atlas> set M3=Levi(P3)
Variable M3: RealForm
atlas> set u=[1,-2,1,1]
Variable u: [int]
atlas> set sgn=vec:[1,1,0,0]
Variable sgn: vec
atlas> set p=parameter(z,sgn,u)
Variable p: Param
atlas> set q=real_induce_standard(p,M3)
Variable q: Param
atlas> real_induce_irreducible(q,H)
Value: 
1*final parameter (x=1,lambda=[0,1,0,0]/1,nu=[0,0,0,0]/1)
1*final parameter (x=16,lambda=[0,3,-1,-1]/1,nu=[-1,3,-1,-1]/2)
1*final parameter (x=24,lambda=[-1,1,1,1]/1,nu=[-1,1,0,0]/1)
1*final parameter (x=32,lambda=[0,3,0,0]/1,nu=[0,1,0,0]/1)
atlas> set q=monomials(real_induce_irreducible(p,M3))[0]
Variable q: Param
atlas> real_induce_irreducible(q,H)
Value: 
1*final parameter (x=0,lambda=[0,1,0,0]/1,nu=[0,0,0,0]/1)
2*final parameter (x=3,lambda=[0,1,0,0]/1,nu=[-1,2,-1,-1]/2)
1*final parameter (x=12,lambda=[-1,2,0,0]/1,nu=[-1,1,0,0]/1)
1*final parameter (x=14,lambda=[-1,2,0,0]/1,nu=[-1,1,0,0]/1)
\end{verbatim}

\subsection{$F_\nu=\C$, $E_\nu=\C\times\C\times\C$ $s=\frac{1}{2}$ and $\chi_\nu=\Id_\nu$}
%atlas3.txt
%Irreducible
We start this case by reminding the reader that a complex reductive Lie group $G\bk{\C}$ is realized in ATLAS as the product $G\bk{\R}\times G\bk{\R}$, where $G\bk{\R}$ is the real split form of $G\bk{\C}$.

We start by defining the group $H=Spin\bk{8,\C}$ and the subgroups $B_E$, $T_E$, $P_E$ and $M_E$:
\begin{verbatim}
atlas> set G=Spin(4,4)
Variable G: RealForm
atlas> set H=complex(G)
Variable H: RealForm
atlas> #KGB(H)
Value: 192
atlas> set x=KGB(H,191)
Variable x: KGBElt
atlas> set P=Parabolic :([0,2,3,4,6,7],x)
Variable P: ([int],KGBElt)
atlas> set M=Levi(P)
Variable M: RealForm
atlas> M
Value: connected quasisplit real group with Lie algebra
      'sl(2,C).sl(2,C).sl(2,C).gl(1,C)'
atlas> set B=Parabolic :([],x)
Variable B: ([int],KGBElt)
atlas> set T=Levi(B)
Variable T: RealForm
\end{verbatim}

Then, we consider $I_\nu\bk{\Id_\nu,\frac{1}{2}}$ as a quotient of $\Ind_{B_E}^{H_E}\eta_{\frac{1}{2}}$.
First, we define $\eta_{\frac{1}{2}}=\bk{1,-1,1,1}$ and consider the induction $\Ind_{B_E}^{M_E}\eta_{\frac{1}{2}}$ and then we pick up the unique irreducible quotient of $\Ind_{B_E}^{M_E}\eta_{\frac{1}{2}}$; this is the one-dimensional representation $\FNorm{{det}_{M_E}}^{\frac{1}{2}}$ of $M_E$.
\begin{verbatim}
atlas> set z=KGB(T,0)
Variable z: KGBElt
atlas> set u=vec:[1,-1,1,1,1,-1,1,1]
Variable u: vec
atlas> set p=parameter(z,null(rank(H)),u)
Variable p: Param
atlas> set par0=monomials(real_induce_irreducible(p,M))
Variable par0: [Param]
atlas> void:for q in par0 do prints(q," ",is_finite_dimensional (q)) od
final parameter (x=0,lambda=[2,-3,2,2,2,-3,2,2]/2,nu=[0,1,0,0,0,1,0,0]/2) false
final parameter (x=1,lambda=[2,-3,2,2,2,-3,2,2]/2,nu=[0,0,0,1,0,0,0,1]/1) false
final parameter (x=2,lambda=[2,-3,2,2,2,-3,2,2]/2,nu=[0,0,1,0,0,0,1,0]/1) false
final parameter (x=3,lambda=[2,-3,2,2,2,-3,2,2]/2,nu=[1,0,0,0,1,0,0,0]/1) false
final parameter (x=4,lambda=[2,-3,2,2,2,-3,2,2]/2,nu=[0,-1,2,2,0,-1,2,2]/2) false
final parameter (x=5,lambda=[2,-3,2,2,2,-3,2,2]/2,nu=[2,-1,0,2,2,-1,0,2]/2) false
final parameter (x=6,lambda=[2,-3,2,2,2,-3,2,2]/2,nu=[2,-1,2,0,2,-1,2,0]/2) false
final parameter (x=7,lambda=[2,-3,2,2,2,-3,2,2]/2,nu=[1,-1,1,1,1,-1,1,1]/1) true
atlas> set q=par0[7]
Variable q: Param
\end{verbatim}
Inducing to $H_E\bk{\C}$, we see that $I_\nu\bk{\Id_\nu,\frac{1}{2}}$ is irreducible.
\begin{verbatim}
atlas> real_induce_irreducible(q,H)
Value: 1*final parameter (x=163,lambda=[2,0,1,1,-2,4,-1,-1]/1,nu=[0,1,0,0,0,1,0,0]/1)
\end{verbatim}

\newpage
\section{Calculation for Non-Square Integrable Residues}
\label{App:Evaluation_of_Normalized_Eisenstein_Series}

In this section we evaluate the normalized Eisenstein series of \Cref{Eq:NormalizedEisensteinSeries} at certain points, useful to the proof of the Siegel-Weil identities of \Cref{Sec:NonSquareIntegrableResidues}.

For any number field $L$ we write
\[
\zfun_L\bk{s} = \frac{R_L}{s-1} + a_0 + a_1\bk{s-1} + ...
\]

From the functional equation $\zfun_L\bk{s}=\zfun_L\bk{1-s}$ we deduce that
\[
\zfun_L\bk{s} = \zfun_L\bk{1-s} = \frac{R_L}{\bk{1-s}-1} + a_0 + a_1\bk{\bk{1-s}-1} + ... = \frac{-R_L}{s} + a_0 - a_1s + ...
\]
From both, one can deduce the following identities
\[
\begin{array}{ll}
 \lim_{s\to -1}\bk{s+1}\zfun_L\bk{s+1} = -R_L 
& \lim_{s\to -1}\bk{s+1}\zfun_L\bk{s+1} = -R_L \\
 \lim_{s\to 1}\bk{s-1}\zfun_L\bk{s-1} = -R_L 
& \lim_{s\to 1}\bk{s-1}\zfun_L\bk{s} = R_L \\
 \lim_{s\to 2}\bk{s-2}\zfun_L\bk{s-1} = R_L 
& \lim_{s\to 2}\bk{s-2}\zfun_L\bk{s-2} = -R_L \\
 \lim_{s\to 1}\bk{2s-2}\zfun_L\bk{2s-1} = R_L 
& \lim_{s\to 1}\bk{2s-2}\zfun_L\bk{2s-2} = -R_L \\
 \lim_{s,s'\to 1}\bk{s+s'-1}\zfun_L\bk{s+s'-1} = R_L 
& \lim_{s,s'\to 1}\bk{s+s'-2}\zfun_L\bk{s+s'-2} = R_L \\
\end{array}
\]
%\begin{align*}
%& \lim_{s\to -1}\bk{s+1}\zfun_L\bk{s+1} = -R_L \\
%& \lim_{s\to -1}\bk{s+1}\zfun_L\bk{s+1} = -R_L \\
%& \lim_{s\to 1}\bk{s-1}\zfun_L\bk{s-1} = -R_L \\
%& \lim_{s\to 1}\bk{s-1}\zfun_L\bk{s} = R_L \\
%& \lim_{s\to 2}\bk{s-2}\zfun_L\bk{s-1} = R_L \\
%& \lim_{s\to 2}\bk{s-2}\zfun_L\bk{s-2} = -R_L \\
%& \lim_{s\to 1}\bk{2s-2}\zfun_L\bk{2s-1} = R_L \\
%& \lim_{s\to 1}\bk{2s-2}\zfun_L\bk{2s-2} = -R_L \\
%& \lim_{s,s'\to 1}\bk{s+s'-1}\zfun_L\bk{s+s'-1} = R_L \\
%& \lim_{s,s'\to 1}\bk{s+s'-2}\zfun_L\bk{s+s'-2} = R_L \\
%\end{align*}

\subsection{$K=F\times F$}

First, we write the normalized Eisenstein series in this case:
\begin{align*}
\Eisen_{B_E}^\sharp\bk{\lambda,g} = 
& \bk{s_1-1} \bk{s_1+1} \zfun_F\bk{s_1+1} \\
& \bk{s_2-1} \bk{s_2+1} \zfun_F\bk{s_2+1} \\
& \bk{s_3-1} \bk{s_3+1} \zfun_F\bk{s_3+1} \\
& \bk{s_4-1} \bk{s_4+1} \zfun_F\bk{s_4+1} \\
& \bk{s_1+s_2-1} \bk{s_1+s_2+1} \zfun_F\bk{s_1+s_2+1} \\
& \bk{s_2+s_3-1} \bk{s_2+s_3+1} \zfun_F\bk{s_2+s_3+1} \\
& \bk{s_2+s_4-1} \bk{s_2+s_4+1} \zfun_F\bk{s_2+s_4+1} \\
& \bk{s_1+s_2+s_3-1} \bk{s_1+s_2+s_3+1} \zfun_F\bk{s_1+s_2+s_3+1} \\
& \bk{s_1+s_2+s_4-1} \bk{s_1+s_2+s_4+1} \zfun_F\bk{s_1+s_2+s_4+1} \\
& \bk{s_2+s_3+s_4-1} \bk{s_2+s_3+s_4+1} \zfun_F\bk{s_2+s_3+s_4+1} \\
& \bk{s_1+s_2+s_3+s_4-1} \bk{s_1+s_2+s_3+s_4+1} \zfun_F\bk{s_1+s_2+s_3+s_4+1} \\
& \bk{s_1+2s_2+s_3+s_4-1} \bk{s_1+2s_2+s_3+s_4+1} \zfun_F\bk{s_1+2s_2+s_3+s_4+1} \Eisen_{B_E}\bk{f^0_\lambda,\lambda,g}. \\
\end{align*}

Plugging in $\lambda_{\bk{-1,s_2,-1,-1}}$, we get:
\begin{align*}
& \Eisen_{B_E}^\sharp\bk{\lambda_{\bk{-1,s_2,-1,-1}},g}  \\
%= & \bk{2R_F}^3 \bk{s_2-1} \bk{s_2+1} \zfun_F\bk{s_2+1}  \bk{s_2-2}^3 s_2^3 \zfun_F\bk{s_2}^3 \\
%& \bk{s_2-3}^3 \bk{s_2-1}^3 \zfun_F\bk{s_2-1}^3 \bk{s_2-4} \bk{s_2-2} \zfun_F\bk{s_2-2} \\
%& \bk{2s_2-4} \bk{2s_2-2} \zfun_F\bk{2s_2-2} \Eisen_{E}\bk{f^0,s_2-\frac{3}{2},g} \\
& = 2^5 R_F^3 \bk{s_2-1}^5 \bk{s_2+1} \zfun_F\bk{s_2+1}  \bk{s_2-2}^5 s_2^3 \zfun_F\bk{s_2}^3 \\
& \bk{s_2-3}^3 \zfun_F\bk{s_2-1}^3 \bk{s_2-4} \zfun_F\bk{s_2-2} \zfun_F\bk{2s_2-2} \Eisen_{E}\bk{f^0,s_2-\frac{3}{2},g}. \\
\end{align*}

Taking the limit as $s_2\to 2$ yields
\begin{align*}
& \Eisen_{B_E}^\sharp\bk{\lambda_{\bk{-1,2,-1,-1}},g} = 
2^9 \cdot R_F^3 \cdot 3\cdot \zfun_F\bk{3} \cdot \zfun_F\bk{2}^4\\ 
& \lim_{s_2\to 2}\coset{\bk{s_2-2}^5 \zfun_F\bk{s_2-1}^3 \zfun_F\bk{s_2-2} \Eisen_{E}\bk{f^0,s_2-\frac{3}{2},g}} \\
& = -2^9  \cdot 3\cdot \zfun_F\bk{3} \cdot \zfun_F\bk{2}^4 \cdot R_F^7 \cdot
\lim_{s\to \frac12}\coset{\bk{s-\frac12} \Eisen_{E}\bk{f^0,s,g}}. \\
\end{align*}

Plugging in $\lambda_{\bk{-1,-1,s_3,s_4}}$, we get:

\begin{align*}
& \Eisen_{B_E}^\sharp\bk{\lambda_{\bk{-1,-1,s_3,s_4}},g} = 2^2 \cdot 3 \cdot \zfun_F\bk{2} \cdot R_F^2 \\
& \bk{s_3-3} \bk{s_3-2} \bk{s_3-1}^2 s_3 \bk{s_3+1} \zfun_F\bk{s_3-1} \zfun_F\bk{s_3} \zfun_F\bk{s_3+1} \\
& \bk{s_4-3} \bk{s_4-2} \bk{s_4-1}^2 s_4 \bk{s_4+1}  \zfun_F\bk{s_4-1} \zfun_F\bk{s_4} \zfun_F\bk{s_4+1} \\
& \bk{s_3+s_4-4} \bk{s_3+s_4-3} \bk{s_3+s_4-2}^2  \bk{s_3+s_4-1} \bk{s_3+s_4} \\
& \zfun_F\bk{s_3+s_4-1} \zfun_F\bk{s_3+s_4} \zfun_F\bk{s_3+s_4-2} \Eisen_{P_{\set{1,2}}}\bk{f^0,\lambda_{s_3-1,s_4-1},g} . \\
\end{align*}

%\begin{align*}
%& \Eisen_{B_E}^\sharp\bk{\lambda_{\bk{-1,-1,s_3,s_4}},g} = 
%2^2 \cdot 3 \cdot \zfun_F\bk{2} \cdot R_F^2 \\
%& \bk{s_3-1} \bk{s_3+1} \zfun_F\bk{s_3+1} \bk{s_4-1} \bk{s_4+1} \zfun_F\bk{s_4+1} \\
%& \bk{s_3-2} s_3 \zfun_F\bk{s_3} \bk{s_4-2} s_4 \zfun_F\bk{s_4} \\
%& \bk{s_3-3} \bk{s_3-1} \zfun_F\bk{s_3-1} \bk{s_4-3} \bk{s_4-1} \zfun_F\bk{s_4-1} \\
%& \bk{s_3+s_4-2} \bk{s_3+s_4} \zfun_F\bk{s_3+s_4} \\
%& \bk{s_3+s_4-3} \bk{s_3+s_4-1} \zfun_F\bk{s_3+s_4-1} \\
%& \bk{s_3+s_4-4} \bk{s_3+s_4-2} \zfun_F\bk{s_3+s_4-2} \Eisen_{P_{\set{1,2}}}\bk{f^0,\lambda_{s_3-1,s_4-1},g} \\
%& = 2^2 \cdot 3 \cdot \zfun_F\bk{2} \cdot R_F^2 \\
%& \bk{s_3-3} \bk{s_3-2} \bk{s_3-1}^2 s_3 \bk{s_3+1} \bk{s_4-3} \bk{s_4-2} \bk{s_4-1}^2 s_4 \bk{s_4+1} \\
%&  \zfun_F\bk{s_3+1} \zfun_F\bk{s_4+1} \zfun_F\bk{s_3} \zfun_F\bk{s_4} \zfun_F\bk{s_3-1} \zfun_F\bk{s_4-1} \\
%& \bk{s_3+s_4-4} \bk{s_3+s_4-3} \bk{s_3+s_4-2}^2 \bk{s_3+s_4-1} \bk{s_3+s_4} \\
%& \zfun_F\bk{s_3+s_4-1} \zfun_F\bk{s_3+s_4} \zfun_F\bk{s_3+s_4-2} \Eisen_{P_{\set{1,2}}}\bk{f^0,\lambda_{s_3-1,s_4-1},g}. \\
%\end{align*}

Taking the limit as $s_3,s_4\to 1$ yields
\begin{align*}
& \Eisen_{B_E}^\sharp\bk{\lambda_{\bk{-1,-1,s_3,s_4}},g} = 2^8 \cdot 3 \cdot \zfun_F\bk{2}^4 \cdot R_F^2 \\
& \lim_{s_3\to 1} \coset{\bk{s_3-1}^2 \zfun_F\bk{s_3-1} \zfun_F\bk{s_3}} \cdot \lim_{s_4\to 1} \coset{ \bk{s_4-1}^2 \zfun_F\bk{s_4-1} \zfun_F\bk{s_4} } \\
& \lim_{s_3,s_4\to 1} \coset{\bk{s_3+s_4-2}^2 \zfun_F\bk{s_3+s_4-1} \zfun_F\bk{s_3+s_4-2}}   \Eisen_{P_{\set{1,2}}}\bk{f^0,\bar{0},g} \\
& = -2^8 \cdot 3 \cdot \zfun_F\bk{2}^4 \cdot R_F^8 \cdot \Eisen_{P_{\set{1,2}}}\bk{f^0,\bar{0},g}. \\
\end{align*}

%
%\begin{align*}
%& \Eisen_{B_E}^\sharp\bk{\lambda_{\bk{-1,-1,1,1}},g} = 
%2^8 \cdot 3 \cdot \zfun_F\bk{2}^4 \cdot R_F^2 \\
%& \lim_{s_3,s_4\to 1}\left[ \bk{s_3-1}^2 \bk{s_4-1}^2 \bk{s_3+s_4-2}^2 \zfun_F\bk{s_3} \zfun_F\bk{s_4} \zfun_F\bk{s_3-1} \zfun_F\bk{s_4-1} \right. \\ 
%& \left. \zfun_F\bk{s_3+s_4-1} \zfun_F\bk{s_3+s_4-2} \Eisen_{P_{\set{1,2}}}\bk{f^0,\lambda_{s_3-1,s_4-1},g} \right] \\
%& = 
%- 2^8 \cdot 3 \cdot \zfun_F\bk{2}^8 \cdot R_F^4 \Eisen_{P_{\set{1,2}}}\bk{f^0,\bar{0},g}. \\
%\end{align*}

\subsection{$K$ is a Field}
First, we write the normalized Eisenstein series in this case:
\begin{align*}
\Eisen_{B_E}^\sharp\bk{\lambda,g} = 
& \bk{s_1-1} \bk{s_1+1} \zfun_F\bk{s_1+1} \\
& \bk{s_2-1} \bk{s_2+1} \zfun_F\bk{s_2+1} \\
& \bk{s_3-1} \bk{s_3+1} \zfun_K\bk{s_3+1} \\
& \bk{s_1+s_2-1} \bk{s_1+s_2+1} \zfun_F\bk{s_1+s_2+1} \\
& \bk{s_2+s_3-1} \bk{s_2+s_3+1} \zfun_K\bk{s_2+s_3+1} \\
& \bk{s_1+s_2+s_3-1} \bk{s_1+s_2+s_3+1} \zfun_K\bk{s_1+s_2+s_3+1} \\
& \bk{s_2+2s_3-1} \bk{s_2+2s_3+1} \zfun_F\bk{s_2+2s_3+1} \\
& \bk{s_1+s_2+2s_3-1} \bk{s_1+s_2+2s_3+1} \zfun_F\bk{s_1+s_2+2s_3+1} \\
& \bk{s_1+2s_2+2s_3-1} \bk{s_1+2s_2+2s_3+1} \zfun_F\bk{s_1+2s_2+2s_3+1} \Eisen_{B_E}\bk{f^0_\lambda,\lambda,g}. \\
\end{align*}

Plugging in $\lambda_{\bk{-1,s_2,-1}}$, we get:
\begin{align*}
& \Eisen_{B_E}^\sharp\bk{\lambda_{\bk{-1,s_2,-1}},g} 
= 2^2 \cdot R_F \cdot R_K \\
& \bk{2s_2-4} \bk{2s_2-2} \bk{s_2-4} \bk{s_2-3}^2 \bk{s_2-2}^3 \bk{s_2-1}^3 s_2^2 \bk{s_2+1} \\
& \zfun_F\bk{2s_2-2} \zfun_F\bk{s_2-2} \zfun_F\bk{s_2-1} \zfun_F\bk{s_2} \zfun_F\bk{s_2+1} \\
& \zfun_K\bk{s_2-1} \zfun_K\bk{s_2} \Eisen_E\bk{f^0,s_2-\frac{3}{2},g}. \\
\end{align*}

Taking the limit as $s_2\to 2$ yields
\begin{align*}
& \Eisen_{B_E}^\sharp\bk{\lambda_{\bk{-1,2,-1}},g} 
= -2^7 \cdot 3 \cdot \zfun_F\bk{2}^2 \cdot \zfun_F\bk{3} \cdot \zfun_K\bk{2} \cdot R_F \cdot R_K \\
& \lim\limits_{s_2\to 2} \coset{\bk{s_2-2}^4 \zfun_F\bk{s_2-2} \zfun_F\bk{s_2-1} \zfun_K\bk{s_2-1}  \Eisen_E\bk{f^0,s_2-\frac{3}{2},g}} \\
& = 2^7 \cdot 3 \cdot \zfun_F\bk{2}^2 \cdot \zfun_F\bk{3} \cdot \zfun_K\bk{2} \cdot R_F^3 \cdot R_K^2 \lim\limits_{s\to \frac12} \coset{\bk{s-\frac12} \Eisen_E\bk{f^0,s,g}}. \\ \\
\end{align*}

Plugging in $\lambda_{\bk{-1,-1,s_3}}$, we get:
\begin{align*}
& \Eisen_{B_E}^\sharp\bk{\lambda_{\bk{-1,-1,s_3}},g} = 2^4 \cdot 3 \cdot \zfun_F\bk{2} \cdot R_F^2 \\
& \bk{2s_3-3} \bk{2s_3-2}^2 \bk{2s_3-1} \bk{s_3-3} \bk{s_3-2}^2 \bk{s_3-1}^2 s_3^2 \bk{s_3+1} \\
& \zfun_F\bk{2s_3-2} \zfun_F\bk{2s_3-1} \zfun_F\bk{2s_3} \zfun_K\bk{s_3-1} \zfun_K\bk{s_3} \zfun_K\bk{s_3+1} \\
& \Eisen_{P_{\set{1,2}}}\bk{f^0,\lambda_{s_3-1},g}. \\
\end{align*}

%\todo{Redo the following:  }
%\begin{align*}
%& \Eisen_{B_E}^\sharp\bk{\lambda_{\bk{-1,-1,s_3}},g} = 2^3 \cdot 3 \cdot \zfun_F\bk{2} \cdot R_F \cdot R_K \\
%& \bk{2s_3-4} \bk{2s_3-3} \bk{2s_3-2}^2 \bk{2s_3-1} \bk{s_3-3} \bk{s_3-2} \bk{s_3-1}^2 s_3^2 \bk{s_3+1} \\
%& \zfun_F\bk{2s_3-2} \zfun_F\bk{2s_3-1} \zfun_F\bk{2s_3} \zfun_K\bk{s_3-1} \zfun_K\bk{s_3} \zfun_K\bk{s_3+1} \\
%& \Eisen_{P_{\set{1,2}}}\bk{f^0,\lambda_{s_3-1},g}. \\
%\end{align*}

Taking the limit as $s_3\to 1$ yields
\begin{align*}
& \Eisen_{B_E}^\sharp\bk{\lambda_{\bk{-1,-1,1}},g} = 2^6 \cdot 3 \cdot \zfun_F\bk{2}^2 \cdot \zfun_K\bk{2} \cdot R_F^2 \\
& \lim\limits_{s_3\to 1} \coset{\bk{2s_3-2}^2 \bk{s_3-1}^2 \zfun_F\bk{2s_3-2} \zfun_F\bk{2s_3-1} \zfun_K\bk{s_3-1} \zfun_K\bk{s_3} }\\
& \Eisen_{P_{\set{1,2}}}\bk{f^0,0,g} \\
&= 2^6 \cdot 3 \cdot \zfun_F\bk{2}^2 \cdot \zfun_K\bk{2} \cdot R_F^4\cdot R_K^2 \Eisen_{P_{\set{1,2}}}\bk{f^0,0,g} .
\end{align*}

%\begin{align*}
%& \Eisen_{B_E}^\sharp\bk{\lambda_{\bk{-1,-1,1}},g} = 2^6 \cdot 3 \cdot \zfun_F\bk{2}^2 \cdot \zfun_K\bk{2} \cdot R_F \cdot R_K \\
%& \lim\limits_{s_3\to 1} \coset{\bk{2s_3-2}^2 \bk{s_3-1}^2 \zfun_F\bk{2s_3-2} \zfun_F\bk{2s_3-1} \zfun_K\bk{s_3-1} \zfun_K\bk{s_3}} \\
%& \Eisen_{P_{\set{1,2}}}\bk{f^0,0,g} \\
%& = 2^6 \cdot 3 \cdot \zfun_F\bk{2}^2 \cdot \zfun_K\bk{2} \cdot R_F^3 \cdot R_K^3 \Eisen_{P_{\set{1,2}}}\bk{f^0,0,g}. \\
%\end{align*}

%\newpage
%\listoftodos [TODO LIST - \arabic{todocounter} TODOS]

\end{document}